\documentclass{amsart} 

\usepackage[latin1]{inputenc} 
\usepackage[colorlinks=true]{hyperref} 
\hypersetup{urlcolor=blue, citecolor=blue} 
\usepackage{enumerate} 
\usepackage{graphicx} 
\usepackage{subfig}

\newtheorem{thm} {Theorem} [section]
\newtheorem{Def} {Definition} 
\newtheorem{lem} [thm]{Lemma} 
\newtheorem{prop} [thm]{Proposition} 
\newtheorem{coro} [thm]{Corollary} 
 
\newtheorem{rem} [thm]{Remark} 

\def\s{{\rm s} } 
\def\sig{\sigma} 
\def\a{\alpha} 
\def\k{\kappa} 
\def\grad{\boldsymbol{\nabla} } 
\def\d{{\rm d} } 
\def\x{{\boldsymbol{x} } } 
\def\p{\partial} 
\def\div{{\grad \cdot} } 
 
\def\n{{\boldsymbol n} } 
\def\O{\Omega} 
\def\M{{\bf H} } 
\def\Ii{{\mathcal I} } 
\def\L{{\boldsymbol \Lambda} } 
\def\tf{{t_{\rm f} } } 
\def\Mm{\mathcal M} 
\def\Mmm{\mathfrak M}

\def\Dd{\mathcal D} 
\def\Ff{\mathcal F} 
\def\Ee{\mathcal E}

\def\Vv{\mathcal V} 
\def\Tt{\mathcal T}

\def\A{{\bf A}} 
\def\B{{\bf B} } 
\def\BV{{\rm BV} } 
\def\bV{{\bf V} } 
\def\G{\Gamma} 
\def\R{\mathbb{R} } 
\def\P{\mathbb{P} } 
\def\be{\begin{equation} } 
\def\ee{\end{equation} }

\def\bfeta{\boldsymbol{\eta} } 
\def\bpsi{\boldsymbol{\psi} } 
\def\bphi{\boldsymbol{\phi} } 
 
\def\eps{\epsilon} 
 
\def\M{{\bf M} } 
\def\dt{{\Delta t} } 
\def\bv{{\boldsymbol v} } 
\def\bh{{\boldsymbol h} } 
\def\bw{{\boldsymbol w} } 
\def\u{{\boldsymbol u} } 
\def\bmu{{\boldsymbol \mu} } 
\def\e{{\rm e} } 
\def\bfe{{\bf e} } 
\def\bd{{\boldsymbol \delta} } 
 
\def\ds{\displaystyle} 
 
\def\b{\beta} 
\def\h#1{{\widehat{#1} } } 
\def\ov#1{{\overline{#1} } } 
 
\def\wt#1{{\widetilde{#1} } } 
\def\meas{{\rm meas} } 
\def\bdt{{\boldsymbol{\Delta} t} } 
\def\1{{\bf 1} } 
\def\0{{\bf 0} } 
\def\hf{\mathfrak{h} } 
\def\gf{\mathfrak{g} } 
\def\Eee{\mathfrak{E} } 

\def\ICI{\begin{center} \color{red} \rule{\textwidth} {1pt} \\ICI \\\rule{\textwidth} {1pt} \end{center} }

\newcounter{cst}

\begin{document} 
\title[Free Energy diminishing FV scheme for parabolic equations]{Numerical analysis of a robust free energy diminishing Finite Volume scheme for parabolic equations with gradient structure} 
\author{Cl\'ement Canc\`es} 
\address{
Cl\'ement Canc\`es ({\tt clement.cances@inria.fr}).
Team RAPSODI, Inria Lille -- Nord Europe, 40 av. Halley, F-59650 Villeneuve d'Ascq, France.
} 
\curraddr{} 
\thanks{This work was supported by the French National Research Agency ANR (project GeoPor, grant ANR-13-JS01-0007-01).} 
\author{Cindy Guichard} 
\address{
Cindy Guichard ({\tt guichard@ljll.math.upmc.fr}).
Sorbonne Universit\'es, UPMC Univ. Paris 06, CNRS, INRIA, CEREMA, UMR 7598, 
Laboratoire Jacques-Louis Lions, \'Equipe ANGE, 4, place Jussieu 75005, Paris, France.
} 

\begin{abstract} 
We present a numerical method for approximating the solutions of degenerate 
parabolic equations with a formal gradient flow structure. The numerical method we propose 
preserves  at the discrete level the formal gradient flow structure, allowing the use of 
some nonlinear test functions in the analysis. 
The existence of a solution to and the convergence of the scheme 
are proved under very general assumptions on the continuous problem 
(nonlinearities, anisotropy, heterogeneity) and on the mesh. 
Moreover, we provide numerical evidences of the efficiency and of the robustness of our approach. 
\end{abstract} 

\maketitle

\section{Introduction} 

Many problems coming from physics (like e.g. porous media flows modeling~\cite{Bear72,AKM90,CP12}) 
or biology (like e.g. chemotaxis modeling~\cite{KS71}) lead to 
degenerate parabolic equations or systems. 
Many of these models can be interpreted as gradient flows in appropriate geometries. 
For instance, such variational structures were depicted for porous media 
flows~\cite{Otto01, LM13_Muskat, CGM15}, chemotaxis processes in biology~\cite{Bla14_KS}, 
superconductivity~\cite{AS08,AMS11}, 
or semiconductor devices modeling~\cite{Mie11, KMX15} (this list is far from being complete).

Designing accurate numerical schemes for approximating their solutions is therefore a major issue. 
In the case of porous media flow models --- used e.g. in oil-engineering, water resources management or 
nuclear waste repository management --- the problems may moreover be highly anisotropic and heterogeneous. 
As an additional difficulty, the meshes are often prescribed by geological data, yielding non-conformal grids 
made of elements of various shapes. This situation can also be encountered in mesh adaptation procedures. 
Hence, the robustness of the method w.r.t. anisotropy and to the grid is 
an important quality criterion for a numerical method in view of practical applications.

In this contribution, 
we focus on the numerical approximation of a single nonlinear Fokker-Planck equation.
Since it contains crucial difficulties arising in the applications, namely degeneracy and possibly strong anisotropy, 
the discretization of this nonlinear Fokker-Planck equation appears to be a keystone for the approximation of 
more complex problems.

\subsection{Presentation of the continuous problem} 
Let $\O$ be a polyhedral connected open bounded subset of $\R^d$ ($d=2$ or $3$), and let $\tf>0$ be a finite time horizon.  
In this contribution, we focus on the discretization of the model problem 
\be\label{eq:1} 
\begin{cases} 
\p_t u - \div\left(\eta(u) \L \grad (p(u) + V) \right) = 0 &\text{ in } Q_\tf := \O\times(0,\tf), \\
 \left(\eta(u) \L \grad (p(u) + V) \right)\cdot \n = 0 & \text{ on } \p\O\times(0,\tf), \\
 u_{|_{t=0} } = u_0 & \text{ in } \O,
\end{cases} 
\ee
which appears to be a keystone before discretizing more complex problems.
We do the following assumptions on the data of the continuous problem~\eqref{eq:1}.
\begin{itemize} 
\item[(${\bf A}1$)]  The function $\eta :\R_+\to\R_+$ is a continuous function such that $\eta(0) = 0$, 
	$\eta(u) >0$ if $ u \neq 0$ and $\eta$ is non-decreasing on $\R_+$. The function $\eta$ is continuously 
	extended on the whole $\R$ into an even function. 
	It is called the {\em mobility} function in reference to the porous media flow context.
	\medskip 
	
\item[(${\bf A}2$)] The so-called {\em (entropy) pressure} function $p\in L^1_{\rm loc} (\R_+)$ 
	is absolutely continuous and increasing on 
	$(0,+\infty)$ (i.e., $0 < p' \in L^1_{\rm loc} ((0,+\infty))$), and satisfies $\lim_{u\to+\infty} p(u) = +\infty$.
	In the case where $p(0) = \lim_{u \searrow0} p(u)$ is finite, the function $p$ is extended into an increasing absolutely 
	continuous function $p:\R\to\R$ defined by 
	\be\label{eq:extend-p} 
	p(u) = 2p(0) -p (-u), \qquad \forall u \le 0.
	\ee
	We denote by 
	$$\Ii_p = \begin{cases} 
	\R_+^\ast & \text{ if } p(0) = -\infty,\\
	\R & \text{ if } p(0) > -\infty.\\
	\end{cases} $$
	and by $\ov \Ii_p$ its closure in $\R$.
	We additionally require that the function $\sqrt{\eta} p'$ belongs to $L^1_{\rm loc} (\R_+)$ 
	(and is in particular integrable near $0$) and that $\lim_{u \searrow 0} \sqrt{\eta(u)} p(u) = 0$.
	
	\medskip
\item[(${\bf A}3$)] The tensor field $\L:\O \to (L^\infty(\R))^{d\times d} $ is such that $\L(\x)$ is symmetric for almost every $\x \in \O$. 
	Moreover, we assume that there exist $\lambda_\star>0$ and $\lambda^\star \in [\lambda_\star,+\infty)$ such that 
	\be\label{A:L} 
	\lambda_\star |{\bf v} |^2 \le \L(\x) {\bf v} \cdot {\bf v} \le \lambda^\star |{\bf v} |^2, \qquad \forall {\bf v} \in \R^d, \; \text{for a.e. } \x \in \O.
	\ee
	$\L$ is called the {\em intrinsic permeability} tensor field still in reference to the porous media flow context.
	\medskip
\item[(${\bf A}4$)] The initial data $u_0$ is assumed to belong to $L^1(\O)$. 
	Moreover, defining the convex function $\Gamma:\ov \Ii_p \to \R_+$ (called \emph{entropy function} in the following) 
	by 
	\be\label{eq:G} 
	\Gamma(u) = \int_{1} ^u (p(a) - p(1)) \d a, \qquad \forall u \in \Ii_p, 
	\ee
	we assume that the following positivity and finite entropy conditions are fulfilled:
	\be\label{A:u0} 
	u_0\ge 0 \text{ a.e. in } \O, \qquad \int_\O u_0 \d\x >0, \qquad \int_\O \G(u_0)\d\x < +\infty.
	\ee
	\medskip
\item[(${\bf A}5$)] The {\em exterior potential} $V: \O \to \R$ is Lipschitz continuous. 
\end{itemize} 
Throughout this paper, we adopt the convention 
\be\label{eq:conv-G} 
\text{$\G(u) = +\infty$\quad  if $p(0) = -\infty$ and $u <0$.} 
\ee

In order to give a proper mathematical sense to the solution of~\eqref{eq:1}, 
we need to introduce the function $\xi:\R_+ \to \R_+$ defined by 
\be\label{eq:xi} 
\xi(u) = \int_0^u \sqrt{\eta(a)} p'(a) \d a, \qquad \forall u \ge 0.
\ee
Note that $\xi$ is well defined since we assumed that $\sqrt{\eta} p'$ belongs to $L^1_{\rm loc} (\R_+)$. 
Moreover, in the case where $p(0)$ is finite, then the formula~\eqref{eq:xi} can be extended to the whole $\R$, 
leading to an odd function. We additionally assume that the following  relations between $\xi$, $\eta$ and $\G$ hold:
\begin{itemize} 
\item[(${\bf A}6$)] There exists $C>0$ such that 
\be\label{eq:C_xi} 
0 \le \xi(u) \le C\left(1+\G(u)\right), \qquad \forall u \in [0,+\infty).
\ee
Moreover, we assume that 
\be\label{eq:dlVP} 
\frac{\G(u)} {\eta(u)} \to +\infty \quad \text{ as } u \to +\infty,  
\ee
and that the function
\be\label{eq:H-unif-cont} 
 \sqrt{\eta\circ \xi^{-1} } \;\text{is uniformly continuous on the range of $\xi$}.
\ee
\end{itemize}

\begin{Def} [weak solution]\label{def:weak} 
A measurable function $u$ is said to be a weak solution to problem~\eqref{eq:1} if 
\begin{enumerate} [{\bf i.}]
\item the functions $u$ and $\eta(u)$ belong to $L^\infty((0,\tf);L^1(\O))$; 
\item the function $\xi(u)$ belongs to $L^2((0,\tf);H^1(\O))$; 
\item for all function $\psi \in C^\infty_c(\ov \O \times [0,\tf);\R)$, one has 
\begin{multline} 
\label{eq:weak} 
\iint_{Q_\tf} u \p_t \psi \, \d\x \d t + \int_\O u_0 \psi(\cdot,0)\d\x \\
- \iint_{Q_\tf} \eta(u) \L \grad V \cdot \grad \psi \d\x \d t - \iint_{Q_\tf} \L \grad \xi(u)  \cdot \sqrt{\eta(u)} \grad \psi \d\x\d t = 0.
\end{multline} 
\end{enumerate} 
\end{Def} 

Following the seminal work of~\cite{AL83}, there exists at least one  weak solution $u$ to the problem~\eqref{eq:1}. 
Denoting by 
$$\phi(u) = \int_0^u \eta(a) p'(a) \d a, \qquad \forall u \in \Ii_p,$$
the uniqueness of the solution (and even a $L^1$-contraction principle) is ensured as soon as $\eta\circ \phi^{-1} \in C^{0,1/2} $ (cf. \cite{Otto96}, see also~\cite{AB04} 
for a slightly weaker condition in the case of a smooth domain $\O$).
Moreover, $u$ belongs to $C([0,\tf];L^1(\O))$ (cf.~\cite{Cont_L1}) and $u(\cdot, t) \ge 0$ for all $t \in [0,\tf]$ 
thanks to classical monotonicity arguments.

\begin{rem}\label{Rem:A}
Assumptions (${\bf A}1$)--(${\bf A}6$) formulated above deserve some comments. 
\begin{itemize}
\item First of all, let us stress that Assumptions  (${\bf A}1$)-- (${\bf A}2$) and  (${\bf A}6$) 
are satisfied if $\eta(u) = u$ and $p(u) = \log(u)$ as in the seminal paper of 
Jordan, Kinderlehrer, and Otto~\cite{JKO98}. One can also deal with power like pressure functions 
$p(u) = u^{m-1}$, but only for $m >1$. Our study does not cover the case of the fast-diffusion equation $m<1$ 
with linear mobility function (see e.g.~\cite{Otto01}) because of the technical assumption (${\bf A}6$).

\smallskip
\item The most classical choice for the mobility function $\eta:\R_+ \to \R_+$ is $\eta(u) = u$. 
In this case, the convection is linear. In this situation, the formal 
gradient flow structure highlighted in \S\ref{ssec:GF-cont} can be made rigorous following the program
proposed in~\cite{JKO98, Otto01, Agueh05, AGS08, Lisini09} and many others.
The gradient flow structure can also be made rigorous in the case where $\eta$ is concave (cf.~\cite{DNS09}) 
and in the non-degenerate case $\eta(u) \ge \alpha >0$.
\smallskip

\item One assumes in (${\bf A}1$) that $\eta$ is nondecreasing on $\R_+$. This assumption is natural in all 
the applications we have in mind. However, it is not mandatory in the proof and can be easily relaxed: 
it would have been sufficient to assume that there exists $\gamma >0$ such that 
$$
\frac{\eta(a) + \eta(b)}2 = \gamma \max_{s \in [a,b]} \eta(s), \qquad \forall [a,b] \subset \R.
$$
This relation is clearly satisfied with $\gamma = 1/2$ when $\eta$ is nondecreasing.
\smallskip

\item In Assumption (${\bf A}2)$, the condition $\lim_{u\to \infty} p(u) = +\infty$  ensures that 
$$\lim_{u\to \infty} \frac{\G(u)}{u} = +\infty,$$ 
where $\G$ was defined in~(${\bf A}4)$.
Given a sequence $\left(u_n\right)_n \subset L^1(\O)$ with bounded entropy, i.e., 
such that $\int_\O \G(u_n) \d\x$ is bounded, then $\left(u_n\right)_n$ is uniformly equi-integrable thanks to 
the de La Vall\'ee Poussin's theorem~\cite{DM78}. Therefore, a sequence $\left(u_n\right)_n$ with 
bounded entropy relatively compact for the weak topology of $L^1(\O)$.

\smallskip
\item Since the unique weak solution to the problem remains non-negative, the extension of $\eta$ and $p$ 
on the whole $\R$ could seem to be useless. However, in the case where $p(0)$ is finite, the non-negativity of the 
solution may not be preserved by the numerical method we propose. The extension of the functions 
$\eta$ and $p$ on $\R_-$ is then necessary. 

\smallskip
\item Only the regularity of the potential $V$ is prescribed by (${\bf A}5$). Confining potentials like 
e.g. $V(\x) = \frac{|\x - \x_\star|^2}2$ for some $\x_\star \in \O$, or gravitational potential 
$V(\x) = -\boldsymbol{g} \cdot \x$, where $\boldsymbol{g}$ is the (downward) gravity vector can be considered. 
\end{itemize}
\end{rem}

\subsection{Formal gradient flow structure of the continuous problem} \label{ssec:GF-cont} 

Let us highlight the (formal) gradient flow structure of the system~\eqref{eq:1}. 
Following the path of~\cite[\S1.3]{Otto01} (see also~\cite{Mie11,Pel-lecture}), 
the calculations carried out in this section are formal. 
They can be made rigorous under the non-degeneracy assumption $\eta(u) \ge \alpha >0$ for all $u \ge 0$. 

Define the affine space 
$$
\Mmm = \left\{ u: \O \to \R \; \left| \; \int_\O u(\x) \d\x = \int_\O u_0(\x) \d\x \right\} \right.
$$
of the admissible states, called \emph{state space}. 

In order to define a Riemannian geometry on $\Mmm$, 
we need to introduce the tangent space $T_u\Mmm$, given by 
$$
T_u\Mmm = \left\{ w:\O\to \R \; \left| \; \int_\O w(\x) \d\x = 0 \right\} \right..
$$
We also need to define the metric tensor $\gf_u : T_u \Mmm \times T_u\Mmm \to \R$, 
which consists in a scalar product on $T_u \Mm$ (depending on the state $u$)
\be\label{eq:g_u} 
\gf_u (w_1, w_2) = \int_\O \phi_1 w_2 \, \d\x = \int_\O w_1 \phi_2 \d\x  = \int_\O \eta(u) \grad \phi_1 \cdot \L \grad \phi_2 \, \d\x, 
\ee
for all $w_1,  w_2 \in T_u\Mmm$, where $\phi_i$ are defined \emph{via} the elliptic problem
\be\label{eq:id-elliptic} 
\begin{cases} 
-\div(\eta(u) \L \grad \phi_i) = w_i \text{ in } \O, \\[5pt]
\eta(u) \L \grad \phi_i \cdot \n = 0 \text{ on } \p\O, \\[5pt]
\ds \int_\O \phi_i\,  \d\x = 0.
\end{cases} 
\ee
Note that $T_u\Mmm$ does not depend on $u$ (at least in the non-degenerate case), but 
the metric tensor $\gf_u(\cdot,\cdot)$ does. So
we are not in a Hilbertian framework. 

Define the \emph{free energy} functional (cf.~\cite{JKO97})
\be\label{eq:Eee} 
\Eee: \begin{cases} 
\Mmm \to \R \cup \{+\infty\} \\
\ds u \mapsto \Eee(u) = \int_\O \left( \Gamma(u(\x)) + u(\x) V(\x)\right)\d\x, 
\end{cases} 
\ee
and the \emph{hydrostatic pressure} function 
$$
\hf : \begin{cases} 
\Ii_p \times \O \to \R \\
(u,\x) \mapsto \hf (u,\x) = p(u) + V(\x) = D_u\Eee(u).
\end{cases} 
$$
Then given $w \in T_u\Mmm$, one has 
\be\label{eq:DEe} 
D_u \Eee(u) \cdot w =  \int_\O \hf(u(\x),\x) w(\x) \; \d\x 
= \int_\O \eta(u) \grad \hf(u,\cdot) \cdot \L \grad \phi \; \d\x, 
\ee
where $\phi$ is deduced from $w$ using the elliptic problem~\eqref{eq:id-elliptic}. 
Moreover, thanks to~\eqref{eq:g_u}, one has 
\be\label{eq:g_u-p_tu} 
\gf_u (\p_t u, w) = \int_\O \p_t u\, \phi \; \d\x, \qquad \forall w \in T_u\Mmm.
\ee
In view of~\eqref{eq:DEe}--\eqref{eq:g_u-p_tu}, the problem~\eqref{eq:1} is equivalent to
\be\label{eq:GF-g_u} 
\gf_u(\p_t u,w) = - D_u \Eee(u) \cdot w = - \gf_u( \grad_u \Eee(u), w), \qquad \forall w \in T_u\Mmm, 
\ee
where the cotangent vector $D_u \Eee(u) \in \left(T_u\Mmm\right)^\ast$ has been identified 
to the tangent vector $\grad_u \Eee(u) \in T_u\Mmm$ 
thanks to  Riesz theorem applied on $T_u\Mmm$ with the scalar product $\gf_u$.
This relation can be rewritten as
\be\label{eq:GF}
\p_t u = - \grad_u \Eee(u) = \div(\eta(u) \L \grad \hf(u,\cdot)) \quad \text{in } T_u\Mmm, 
\ee
justifying the gradient flow denomination. 

Choosing $w = \p_t u$ in~\eqref{eq:GF-g_u} and using~\eqref{eq:GF}, we get that 
\begin{multline*}
\frac{\d}{\d t} \Eee(u) = D_u \Eee(u) \cdot \p_t u = - D_u \Eee(u) \cdot \grad_u \Eee(u) 
 = \int_\O \eta(u) \grad \hf(u, \cdot) \cdot \L \grad \hf(u, \cdot) \d\x.
\end{multline*}
An integration w.r.t. time yields the classical energy/dissipation relation: 
$\forall t \in [0,\tf]$, 
\begin{multline} \label{eq:NRJ-V} 
\Eee(u(\cdot,t))-\Eee(u_0) \\+\int_0^{t} \int_\O \eta(u(\x,\tau))\L(\x)\grad\hf(u(\x,\tau),\x)
\cdot\grad\hf(u(\x,\tau),\x) \d\x \d \tau = 0.
\end{multline} 

The fact that a physical problem has a gradient flow structure provides some informations
concerning its evolution. The physical system aims at decreasing its free energy as fast a 
possible. As highlighted by~\eqref{eq:NRJ-V}, the whole energy decay corresponds to the dissipation. 
As a byproduct, the free energy is a Liapunov functional and the total dissipation (integrated w.r.t. time) 
is bounded by the free energy associated to the initial data. 
The variational structure was exploited for instance in~\cite{Otto01,BGG12,BGG13,ZM15} to 
study the long-time asymptotic of the system.

\subsection{Goal and positioning of the paper} \label{ssec:goal} 

The goal of this paper is to propose and analyse a numerical scheme that mimics at the discrete level the 
gradient flow structure highlighted in \S\ref{ssec:GF-cont}. 
Since the point of view adopted in our presentation concerning the gradient flow structure is formal, 
the rigorous numerical analysis of the scheme will rather rely 
on the well established theory of weak solutions in the sense of Definition~\ref{def:weak}. 
But as a byproduct of the formal gradient flow structure, the discrete free energy 
will decrease along time, yielding the non-linear stability of the scheme. 

There are some existing numerical methods based on Eulerian coordinates 
(as the one proposed in this paper). 
This is for instance the case of monotone discretizations, that can be reinterpreted as 
Markov chains~\cite{Maas11, EM14}, for which one can even prove a rigorous gradient flow structure. 
Classical ways to construct monotone discretizations in the isotropic setting $\L(\x) = \lambda(\x) {\bf I}_d$ 
are to use finite volumes schemes with two-points flux approximation (TPFA, see e.g.~\cite{EGH00,Tipi}) or 
finite elements on Delaunay's meshes (see e.g.~\cite{EF-L1}). 
An advanced second order in space finite volume method was proposed in~\cite{BCF12}
and discontinuous Galerkin schemes in~\cite{LY15,LW16}. 
However, these approaches --- as well as the finite difference scheme proposed in~\cite{LW14} ---
require strong assumptions on the mesh. 
Moreover, the extension to the anisotropic framework 
of TPFA finite volume scheme fails for consistency reasons (cf.~\cite{Tipi}), while 
finite elements are no longer monotone on a prescribed mesh for general anisotropy tensors $\L$. 
In~\cite{bench_FVCA5, Bench_FVCA6}, it appears that all the linear finite volume schemes 
(i.e., schemes leading to linear systems when linear equations are approximated) able 
to handle general grids and anisotropic tensors may loose at the discrete level the monotonicity 
of the continuous problem. 

The monotonicity at the discrete level can be restored thanks to nonlinear 
corrections~\cite{BE04, LePotier10, CClP13, LePotier14}. Another approach 
consists in designing directly monotone nonlinear schemes (see, e.g.,~\cite{Kap07, YS08, LSV09, LSV10, DlP11, SY11}).
However, the monotonicity of the method is not sufficient to ensure the decay of non-quadratic energies. 
Moreover, the available convergence proofs~\cite{DlP11, CClP13} require some numerical assumptions 
involving the numerical solution itself. 
Finally, let us mention here the recent contribution~\cite{FM14} where linear monotone schemes are constructed 
on cartesian grids for possibly anisotropic tensors $\L$. 

In the case where $\L = {\bf I}_d$ and $\eta(u) = u$, the formal gradient flow structure can 
be made rigorous in the metric space 
$$\mathcal{P}(\O) := \left\{ u \in L^2(\O; \R_+)\; \middle|\; \int_\O u \d\x = \int_\O u_0\d\x \right\}$$ 
endowed with the Wasserstein metric with quadratic cost function.
Several approach were proposed in the last years for solving the JKO minimization scheme 
(cf.~\cite{JKO98, AGS08}). This requires the computation of Wasserstein distances. 
If $d=1$, switching to Lagrangian coordinates is a natural choice that has been exploited for example 
in~\cite{KW99, BCC08, MO14, MO15}. The case $d\ge 2$ is more intricate. Methods based on a so-called 
{\em entropic regularization}~\cite{BCCNP15, Peyre15} of the transport plan appear to be costly, 
but very tractable. Another approach consists in solving the so-called Monge-Amp\`ere equation 
in order to compute the optimal transport plan~\cite{BGMO15}. Let us finally mention the 
application~\cite{BCL15} to the Wasserstein gradient flows of the 
{\em CFD relaxation approach} of Benamou and Brenier~\cite{BB00} to solve the Monge-Kantorovich 
problem.

Motivated by applications in the context of complex porous media flows where irregular grids are often prescribed 
by the geology, the scheme we propose was designed to be able to handle highly anisotropic and heterogeneous 
diffusion tensor $\L$ and very general grids (non-conformal grids, cells of various shapes). 
It relies on  the recently developed {\em Vertex Approximate Gradient} (VAG) 
method~\cite{VAG_bench, EGHM12,EGH12, BM13}, but 
alternative versions can be inspired from most of the existing symmetric coercive methods for approximating 
the solutions of linear elliptic equations~\cite{bench_FVCA5, Bench_FVCA6}. Moreover, we want our scheme 
to mimic at the discrete level the gradient flow structure highlighted in \S\ref{ssec:GF-cont}. 
This ensures in particular the decay of the discrete counterpart of the free energy, and thus the nonlinear stability 
of the scheme. 

A nonlinearly stable  control volume finite elements (CVFE) scheme was proposed 
in our previous contribution~\cite{CG16_MCOM} (see also \cite{CG-FVCA7}). 
The nonlinear CVFE scheme~\cite{CG16_MCOM} is based 
on a suitable upwinding of the mobility. It is only first order accurate in space 
while linear CVFE schemes are second order accurate in space. 
Moreover, it appears in the numerical simulations presented in~\cite{CG16_MCOM} 
that this nonlinear CVFE scheme lacks robustness with respect to anisotropy: its 
convergence is slow in particularly unfavorable situations.

The main goal of this paper is to propose a scheme that preserves some important features of 
the one introduced in~\cite{CG16_MCOM} 
(possible use of some prescribed nonlinear test function, decay of the physically motivated energy, 
convergence proof for discretization parameters tending to $0$), 
without jeopardizing the accuracy of the scheme compared to the more classical approach based on 
formulations with Kirchhoff transforms (see for instance \cite{EHV06, RPK08, EHV10,ABH13}).
Convincing numerical results are provided in~\S\ref{sec:num} as an evidence of the efficiency of our approach.
Two theorems are stated in~\S\ref{ssec:main} (and proved in \S\ref{sec:apriori} and \S\ref{sec:conv}) 
in order to ensure the following properties:
\begin{enumerate} 
\item Theorem~\ref{thm:main-disc}. At a fixed mesh, the scheme, that consists in a nonlinear system, 
admits (at least) one solution. 
This allows in particular to speak about the discrete solution provided by the scheme. Moreover, we take 
advantage of the gradient structure of the scheme for deriving some nonlinear stability estimates.
\item Theorem~\ref{thm:main-conv}. Letting the discretization parameters tend to 0 (while controlling some 
regularity factors related to the discretization), the discrete solution converges in some appropriate sense 
towards the unique weak solution to the problem~\eqref{eq:1} in the sense of Definition~\ref{def:weak}. 
\end{enumerate}

\begin{rem}\label{rem:convect}
We only consider potential convection in the paper. A more general convection 
with speed ${\bf v}$ can be split into a potential part and a divergence free part:
$$
{\bf v}= - \nabla V + \nabla \times {\bf w}.
$$ 
We suggest for instance to use classical (entropy stable) finite volume schemes (see for instance~\cite{LeV02})
for the divergence free part combined with our method for the potential part. 
\end{rem}

\section{Definition of the scheme and main results} \label{sec:scheme} 

As already mentioned, the scheme we propose is based on the so-called VAG scheme~\cite{VAG_bench}. 
In~\S\ref{ssec:mesh}, we state our assumptions on the spatial mesh and the time discretization of $(0,\tf)$.
Then in~\S\ref{ssec:VAG-NL}, we define the nonlinear scheme we will study in this paper. 
The gradient flow structure of the discretized problem is highlighted in \S\ref{ssec:GF-disc}, where 
a variational interpretation is given to the scheme. 
Finally, in~\S\ref{ssec:main}, we state the existence of discrete solutions to the scheme 
and their convergence towards the unique weak solution as the discretization parameters tend to $0$.

\subsection{Discretization of $Q_\tf$ and discrete functional spaces} \label{ssec:mesh} 

\subsubsection{Spatial discretization and discrete reconstruction operators} \label{sssec:mesh-space} 

Following~\cite{EGH12, BM13}, we consider generalized polyhedral discretizations of $\O$.
Let $\Mm$ be the set of the cells, that are disjoint polyhedral open subsets of $\O$ such that 
$
\bigcup_{\k\in\Mm} \ov \k = \ov \O.
$ 
Each cell $\k \in \Mm$ is assumed to be star-shaped with respect to its so-called {\em center}, denoted by $\x_\k$.
We denote by $\Ff$ the set of the faces of the mesh, which are not assumed to be planar if $d=3$ (whence the 
term ``generalized polyhedral"). We denote by $\Vv$ the set of the vertices of the mesh. We denote by $\x_\s\in \ov \O$ 
the location of the vertex $\s \in \Vv$. 
The sets $\Vv_\k$, $\Ff_\k$ and $\Vv_\sig$ denote respectively the vertices and faces of a cell $\k$, and the 
vertices of a face $\sig$. For any face $\sig \in \Ff_\k$, one has $\Vv_\sig \subset \Vv_\k$. 
Let $\Mm_\s$ denote the set of the cells sharing the vertex $\s$.
The set of edges of the mesh (defined only if $d=3$)  is denoted by $\Ee$ and $\Ee_\sig$ denotes the set of edges of the face
$\sig \in \Ff$, while $\Ee_\k$ denotes the set of the edges of the cell $\k$. 
    The set $\Vv_\e$ denotes the pair of vertices at the extremities of the edge $\e\in\Ee$.
 In the $3$-dimensional case, it is assumed that for each face $\sig \in \Ff$, there exists a so-called ``center" of
the face $\x_\sig$ such that
\be\label{eq:x_sig} 
\x_\sig = \sum_{\s \in \Vv_\sig} \beta_{\sig,\s} \x_\s, \qquad \text{with } \quad\sum_{\s \in \Vv_\sig} \beta_{\sig,\s} =1, 
\ee
and $\beta_{\sig,\s} \ge 0$ for all $\s \in \Vv_\sig$. 
The face $\sig$ is then assumed to match with the union of the triangles $T_{\sig,\e} $ defined as 
by the face center $\x_\sig$ and each of its edge $\e\in\Ee_\sig$.
A two-dimensional example of mesh $\Mm$ is drawn on figure~\ref{fig:Mm}.

\medskip
The previous discretization is denoted by $\Dd$, and we define the discrete space 
$$
W_\Dd =  \left\{ \bv = \left(v_\k,v_\s\right)_{\k\in \Mm,\s\in \Vv} \in \R^{\#\Mm+\#\Vv} \right\}.
$$
In the $3$-dimensional case, we introduce for all $\sig\in\Ff$ the operator $I_\sig:W_\Dd \to \R$ defined by
$$
I_\sig(\bv) = \sum_{\s \in \Vv_\sig} \beta_{\sig,\s} v_\s, \qquad \forall \bv \in W_\Dd,
$$
yielding a second order interpolation at $\x_\sig$ thanks to the definition~\eqref{eq:x_sig} of $\x_\sig$.

\begin{figure}[htb]
\begin{center}
\includegraphics[width=7cm]{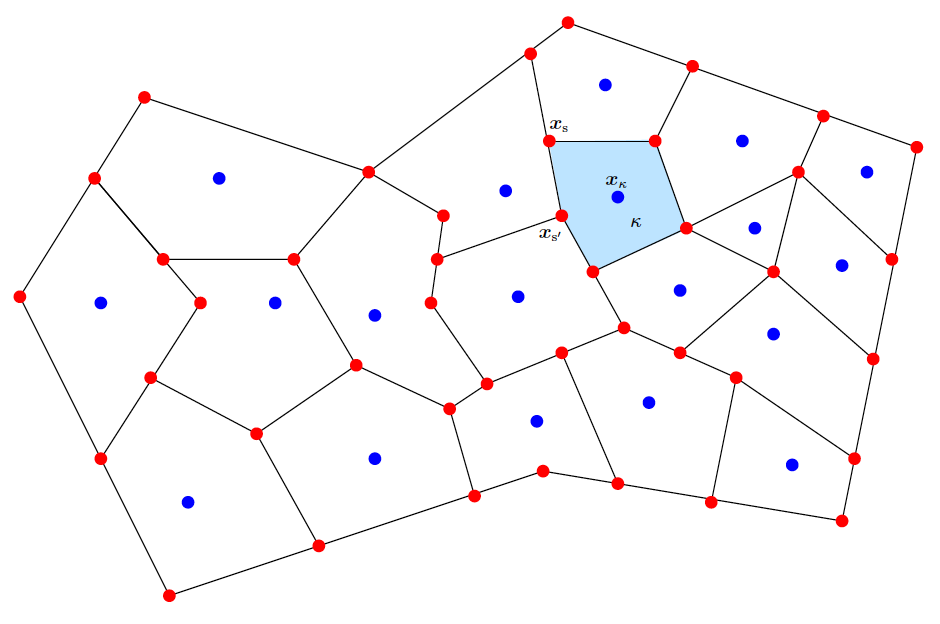}
\caption{The primal mesh $\Mm$ can be made of cells with various and general shapes. 
Degrees of freedom are located at the so-called cell center (blue dots) and at the so-called 
vertices (red dots). }
\label{fig:Mm}
\end{center}
\end{figure}

\medskip 
We introduce the simplicial submesh $\Tt$ (a two-dimensional illustration is provided on Figure~\ref{fig:Tt}) 
defined by 
\begin{itemize} 
\item $\Tt = \{T_{\k,\sig}, \k\in \Mm, \sig\in\Ff_\k\} $ in the two-dimensional case, where $T_{\k,\sig} $ denotes the triangle 
whose vertices are $\x_\k$ and $\x_\s$ for $\s \in \Vv_\sig$; 
\item $\Tt = \{T_{\k,\sig,\e}, \k\in \Mm, \sig\in\Ff_\s, \,\e\in\Ee_\sig\} $ in the three-dimensional case, where $T_{\k,\sig,\e} $ denotes the tetrahedron whose vertices are $\x_\k$, $\x_\sig$ and $\x_\s$ for  $\s \in \Vv_\e$.
\end{itemize} 
\begin{figure}[htb]
\begin{center}
\includegraphics[width=7cm]{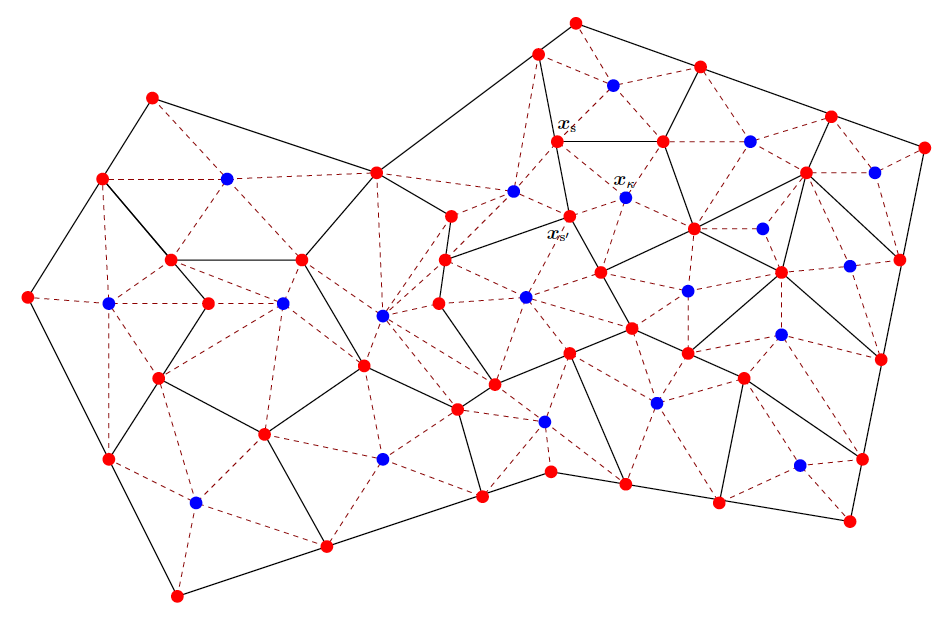}
\caption{The simplicial submesh $\Tt$ is derived from the primal mesh $\Mm$ by decomposing the 
primal cells $\k \in \Mm$ into triangles if $d=2$ or tetrahedra if $d=3$. This construction is possible since 
$\k$ was assumed to be star-shaped with respect to a ball centered in $\x_\k$. }
\label{fig:Tt}
\end{center}
\end{figure}

We define the regularity $\theta_\Tt$ of the simplicial mesh $\Tt$ by 
\be\label{eq:theta_T} 
\theta_\Tt = \max_{T\in\Tt} \frac{h_T} {\rho_T}, 
\ee
where $h_T$ and $\rho_T$ respectively denote the diameter of $T$ and the insphere diameter of $T$. 
We denote by 
\be\label{eq:hstar} 
h_\Tt = \max_{T \in \Tt} h_T
\ee
the maximum diameter of the simplicial mesh.
We also define the 
quantities $\ell_\k$ and $\ell_\s$ quantifying the number of vertices of the cell $\k$ and the number of 
neighboring cells for the vertex $\s$ respectively: 
\be\label{eq:chi_ks} 
\ell_\k = \#\Vv_\k, \quad \ell_\s = \#\Mm_\s, \quad \forall \k \in \Mm, \; \forall \s \in \Vv.
\ee
This allows to introduce the quantity
\be\label{eq:gamma_M} 
\ell_\Dd = \max \left( \max_{\b \in \Mm\cup\Vv} \ell_\b, \max_{\k \in \Mm} \{\#\Ff_\k\} \right),
\ee
controlling the regularity of the general discretization $\Dd$ of $\O$.

\medskip
Denoting by $H_\Tt\subset W^{1,\infty} (\O)$ the usual $\P_1$-finite element space on the simplicial mesh~$\Tt$, we define the 
reconstruction operator $\pi_\Tt : W_\Dd \to H_\Tt$ by setting, for all $\bv \in W_\Dd$ and all 
$(\s,\k,\sig) \in \Vv\times\Mm\times\Ff$, 
\be\label{eq:pi_Tt} 
\pi_\Tt \bv(\x_\s) = v_\s, \quad  \pi_\Tt \bv(\x_\k) = v_\k,\quad \text{and} \quad \pi_\Tt \bv(\x_\sig) = I_\sig(\bv).
\ee
This allows to define the operator $\grad_\Tt : W_\Dd \to (L^\infty(\O))^d$ by
\be\label{eq:grad_Tt} 
\grad_\Tt \bv = \grad \pi_\Tt \bv, \qquad \forall \bv \in W_\Dd.
\ee

\medskip
We aim now to reconstruct piecewise constant functions. To this end, we introduce 
a so-called \emph{mass lumping} mesh $\Dd$ depending on additional parameters 
that appear to play an important role in practical applications~\cite{EGHM12}.
Let $\k \in \Mm$, then introduce some weights $\left(\a_{\k,\s} \right)_{\s\in\Vv_\k} $ such that 
\be\label{eq:aks} 
 \a_{\k,\s} \ge 0, \quad\text{and} \quad \sum_{\s \in \Vv_\k} \a_{\k,\s} \le 1, \quad \forall \s \in \Vv_\k.
\ee
Denoting by ${\rm meas} (\k) = \int_\k \d\x$ the volume of $\k$, then we define the quantities 
\be\label{eq:mks} 
\begin{cases} 
m_{\k,\s} = \a_{\k,\s} {\rm meas} (\k), & \forall \k \in \Mm, \forall \s \in \Vv_\k, \\
m_\s = \sum_{\k \in \Mm_\s} m_{\k,\s}, & \forall \s \in \Vv, \\
m_\k =  {\rm meas} (\k) - \sum_{\s \in \Vv_\k} m_{\k,\s} 
, & \forall \k \in \Mm, 
\end{cases} 
\ee
so that one has 
$$
\sum_{\b \in \Mm\cup\Vv} m_\b = {\rm meas} (\O).
$$

For all $\k \in \Mm$, we denote by $\omega_{\k} $ and $\omega_{\k,\s} $ some disjointed open subsets of $\k$, 
such that 
$$
\ov \omega_\k \cup \left(\bigcup_{\s \in \Vv_\k} \ov \omega_{\k,\s} \right) = \ov \k, 
$$
and such that 
$$
\meas (\omega_\k) = m_\k \quad \text{and} \quad \meas(\omega_{\k,\s}) = m_{\k,\s}, \qquad  \forall \k \in \Mm, \;\forall \s \in \Vv_\k.
$$
Note that such a decomposition always exists thanks to~\eqref{eq:aks}--\eqref{eq:mks}. Then we denote by 
$$
\omega_\s =\bigcup_{\k \in \Mm_\s} \omega_{\k,\s}, \qquad \forall \s \in \Vv.
$$
The mass-lumping mesh $\Dd$ is made of the cells $\omega_\k$ and $\omega_\s$ for $\k \in \Mm$ and 
$\s \in \Vv$. An illustration is presented in Figure~\ref{fig:Dd}.
\begin{figure}[htb]
\begin{center}
\includegraphics[width=7cm]{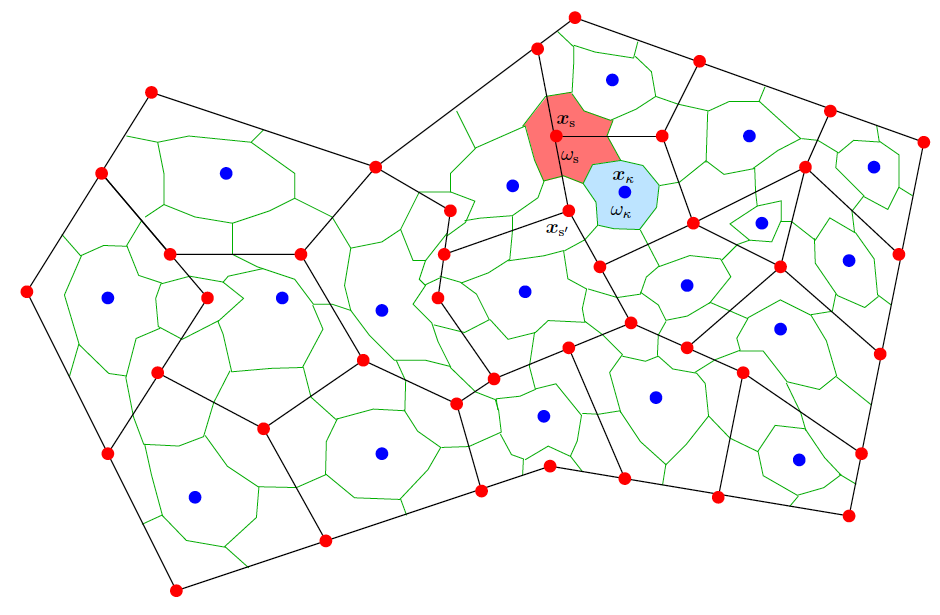}
\caption{The mass-lumping mesh $\Dd$ contains one cell $\omega_\k$ and $\omega_\s$ by degree of fredom. 
The mass repartition depends on the factors $\alpha_{\k,\s}$ introduced at~\eqref{eq:aks}.}
\label{fig:Dd}
\end{center}
\end{figure}
In \cite{EGHM12}, a study focused on the repartition of the porous volume between nodes and centers is proposed, 
in the case of a coupled problem (transport of a species in the flow of a fluid in a porous medium). 
The influence of this repartition mainly concerns the transport part, and not the fluid flow. 
In the framework of the present paper, we numerically observe that this repartition has not a 
strong influence in the cases studied. 
However, since the function $\eta$ can vanish, the limit choice $m_\k = 0$ (resp. $m_\s = 0$) can lead to singular cases, and therefore must be prevented.

In what follows, we denote by 
\be\label{eq:zeta_Dd} 
\zeta_\Dd=\min_{\beta\in\Mm\cup\Vv} \frac{m_\beta} {\int_\O\pi_\Tt\bfe_\beta\d\x},
\ee
where $\bfe_\beta$, $\beta \in \Mm \cup  \Vv$ is the unique element of $W_\Dd$ such that 
$$
\pi_\Tt \bfe_\beta(\x_\gamma) = \delta_\beta^\gamma, \qquad
\forall \gamma \in \Mm \cup \Vv, 
$$
and $\delta_\beta^\gamma$ is the kronecker symbol.

We can now define the piecewise constant reconstruction operators $\pi_\Dd: W_\Dd \to L^\infty\cap \BV(\O)$ and 
$\pi_\Mm: W_\Dd \to L^\infty\cap \BV(\O)$ by 
\be\label{eq:pi_Dd} 
\pi_\Dd(\bv)(\x) = \sum_{\k \in  \Mm} v_\k \1_{\omega_\k} (\x) +   \sum_{\s \in  \Vv} v_\s \1_{\omega_\s} (\x), \qquad \forall \x \in \O, \quad \forall \bv \in W_\Dd,
\ee
and 
\be\label{eq:pi_Mm} 
\pi_\Mm(\bv)(\x) = \sum_{\k \in  \Mm} v_\k \1_{\k} (\x)  \qquad \forall \x \in \O, \quad \forall \bv \in W_\Dd.
\ee

\medskip
Let $f : \R \to \R$ be a possibly nonlinear function, then denote by 
$$
f(\bv) = \left( f(v_\k), f(v_\s)\right)_{\k\in\Mm,\s\in\Vv}, \qquad \forall \bv =  \left( v_\k, v_\s\right)_{\k\in\Mm,\s\in\Vv} \in W_\Dd.
$$
Notice that in general, 
$$\pi_\Tt (f(\bv)) \neq f(\pi_\Tt(\bv))\quad  \text{and} \quad \grad_\Tt (f(\bv)) \neq \grad f(\pi_\Tt(\bv)),$$
but that 
\be\label{eq:nonlin-D} 
\pi_\Dd (f(\bv)) = f(\pi_\Dd(\bv)) \quad \text{and } \quad \pi_\Mm (f(\bv)) = f(\pi_\Mm(\bv)),  \qquad \forall \bv \in W_\Dd.
\ee

\subsubsection{Time-and-space discretizations and discrete functions} \label{sssec:mesh-time} 

Let $N\ge 1$ and let $0 = t_0 < t_1 < \dots < t_{N-1} < t_N = \tf$ be some subdivision of $[0,\tf]$.
We denote by $\dt_n = t_n - t_{n-1} $ for all $n \in \{1,\dots N\} $, by $\bdt = (\dt_1, \dots, \dt_N)^T \in \R^N,$
and by 
\be\label{eq:dtstar} 
\ov \dt = \max_{1\le n \le N} \dt_n. 
\ee

The time and space discrete space is then defined by 
$$
W_{\Dd,\bdt} = \left\{ \bv = \left(v_\k^n,v_\s^n\right)_{\k\in \Mm,\s\in \Vv, 1 \le n \le N} \in \R^{N(\#\Mm+\#\Vv)} \right\}.
$$
For $\bv\in W_{\Dd,\bdt} $ and $n \in \{1,\dots,N\} $, we denote by 
$$\bv^n =  \left(v_\k^n,v_\s^n\right)_{\k\in \Mm,\s\in \Vv} \in W_\Dd.$$

We deduce from the space reconstruction operators $\pi_\Dd$, $\pi_\Mm$ and $\pi_\Tt$ some time and space reconstructions operators 
$\pi_{\Dd,\bdt},\pi_{\Mm,\bdt}, \pi_{\Tt,\bdt} : W_{\Dd,\bdt} \to L^\infty(Q_\tf)$ mapping the elements of 
$W_{\Dd,\bdt}$ into constant w.r.t. time functions 
defined by 
$$
\pi_{\Dd,\bdt} \bv(\cdot,t) = \pi_\Dd(\bv^n),  \quad\pi_{\Mm,\bdt} \bv(\cdot,t) = \pi_\Mm(\bv^n)\quad  \text{and} \quad
\pi_{\Tt,\bdt} \bv(\cdot,t) = \pi_\Tt(\bv^n) 
$$
if $t \in (t_{n-1},t_n]$.  
The gradient reconstruction operator $\grad_{\Tt,\bdt} : W_{\Dd,\bdt} \to \left(L^\infty(Q_\tf)\right)^d$ is then defined by 
$$
\grad_{\Tt,\bdt} \bv = \grad \pi_{\Tt,\bdt} \bv, \qquad \forall \bv \in W_{\Dd,\bdt}.
$$

\subsection{The nonlinear scheme for degenerate parabolic equations} \label{ssec:VAG-NL} 
For $\k \in \Mm$, we denote by $\A_\k = \left(a^\k_{\s,\s'} \right)_{\s,\s' \in \Vv_\k} \in \R^{\ell_\k\times\ell_\k} $ 
the symmetric positive definite matrix whose coefficients are defined by 
\be\label{eq:akss} 
a_{\s,\s'} ^\k = \int_\k \L(\x) \grad_\Tt \bfe_\s(\x) \cdot \grad_\Tt \bfe_{\s'} (\x) \d\x = a^\k_{\s',\s}.
\ee
It results from the relation 
$$
\pi_\Tt \bfe_\k(\x) + \sum_{\s \in \Vv_\k} \pi_\Tt \bfe_\s(\x) = 1, \qquad \forall \x \in \k, \; \forall \k \in \Mm,
$$
that, for all $\u,\bv \in W_\Dd$ and all $\k \in \Mm$, one has 
\be\label{eq:bilin} 
\int_\k \L(\x) \grad_\Tt \u(\x) \cdot \grad_\Tt \bv(\x) \d\x = \sum_{\s \in \Vv_\k} \sum_{\s'\in\Vv_\k} a_{\s,\s'} ^\k (u_\k - u_\s) (v_\k - v_{\s'}).
\ee
For $\k \in \Mm$, we denote by $\bd_\k:W_\Dd \to \R^{\ell_\k} $ the linear operator defined by
$$
\left(\bd_\k \bv\right)_\s = v_\k - v_\s, \qquad \forall \s \in \Vv_\k, \; \forall \bv \in W_\Dd.
$$
With this notation, we obtain that~\eqref{eq:bilin} rewrites as
$$
\int_\k \L(\x) \grad_\Tt \u(\x) \cdot \grad_\Tt \bv(\x) \d\x =  \bd_\k \bv \cdot \A_\k\bd_\k \u, \qquad \forall \u,\bv \in W_\Dd,\; \forall \k \in \Mm. 
$$
\medskip

In order to deal with the nonlinearities of the problem, we introduce the sets $W^{\rm ad} _\Dd\subset W_\Dd$ and 
$W^{\rm ad} _{\Dd, \bdt} \subset W_{\Dd,\bdt} $ of the admissible states defined by 
$$
\bv \in W^{\rm ad} _\Dd \quad \text{iff} \qquad v_\nu \in \Ii_p, \quad \forall \nu \in \Mm\cup\Vv,
$$
and 
$$
\bv \in W^{\rm ad} _{\Dd,\bdt} \quad \text{iff} \qquad \bv^n \in W^{\rm ad} _\Dd, \quad \forall n \in \{1,\dots,N\}, 
$$
while we denote by $W_\Dd^{\rm en} \subset W_\Dd$ the set of finite entropy vectors:
\be\label{eq:Eee_Dd} 
\bv \in W^{\rm en} _\Dd \quad \text{iff} \quad
\Eee_\Dd(\bv) := \int_\O \left(\G(\pi_\Dd \bv) + \pi_\Dd(\bv) \pi_\Dd(\bV)\right) \d\x < \infty, 
\ee
where $\bV = (V_\k,V_\s)_{\k,\s} \in W_\Dd$ is defined by 
\be\label{eq:bV} 
V_\k = V(\x_\k)
\quad 
V_\s = V(\x_\s)
\qquad \forall \k \in \Mm, \forall \s \in \Vv. 
\ee
It is easy to check that, thanks to assumptions $({\bf A}2)$ and to the definition~\eqref{eq:G} of the 
convex function $\G$, one has $W^{\rm ad} _{\Dd} \subset W^{\rm en} _\Dd$.

Given $\u \in W_\Dd^{\rm ad} $ and $\bV\in W_\Dd$, we define the discrete hydrostatic pressure
$\bh(\u) = \left( \hf_\k(u_\k), \hf_\s(u_s)\right)_{\k,\s} \in W_\Dd$ 
by 
\be\label{eq:fhks} 
\hf_\k(u_\k) = p(u_\k) + V_\k, \qquad \hf_\s(u_\s) = p(u_\s) + V_\s, \qquad 
\forall \k \in \Mm, \, \forall \s \in \Vv.
\ee
\medskip

The initial data $u_0$ is discretized into an element $\u^0 \in {W^{\rm en} _\Dd} $ by 
\be\label{eq:u0-disc} 
u_\b^0 = \frac1{m_\b} \int_{\omega_\b} u_0(\x)\d\x, \qquad \forall \b \in \Mm \cup \Vv, 
\ee
so that, thanks to~\eqref{A:u0}, 
\be\label{eq:u0-mass} 
\int_\O \pi_\Dd (\u^0) \d\x = \int_\O u_0 \d\x >0. 
\ee
Let us state a first lemma that ensure that the discretized initial data has a finite 
discrete entropy. 
\begin{lem} \label{lem:entro-0} 
Let $u_0 \in L^1(\O)$ be such that~$({\bf A}{4})$ holds, $V$ be such that~$({\bf A}{5})$ holds, 
$\u^0$ be defined by~\eqref{eq:u0-disc}, and $\bV$ defined by~\eqref{eq:bV}. Then there 
exists $C$ depending only on $\|u_0\|_{L^1(\O)} $ and $\|\grad V\|_{L^\infty(\O)^d} $ such that 
\be\label{eq:entro-0_app} 
\Eee_\Dd(\u^0) \le \Eee(u_0) + C h_\Tt \le  \Eee(u_0) + C {\rm diam} (\O), 
\ee
where the entropy functional $\Eee$ is defined by~\eqref{eq:Eee} and its discrete counterpart $\Eee_\Dd$ is 
defined by~\eqref{eq:Eee_Dd}.
In particular, $\u^0$ belongs to ${W^{\rm en} _\Dd} $.
\end{lem} 
\begin{proof} 
We deduce from Jensen inequality that 
$$
\G(u^0_\beta) \le \frac1{m_\beta} \int_{\omega_\b} \G(u_0)\d\x, 
$$
whence, thanks to~$({\bf A}{4})$ and to the definition~\eqref{eq:Eee_Dd} of the discrete entropy functional~$\Eee_\Dd$, 
one has 
$$
\Eee_\Dd(\u^0) \le \Eee(u_0) 
+ \int_\O u_0 (\pi_\Dd \bV - V)\d\x + \int_\O (\pi_\Dd\u^0 - u_0) \pi_\Dd \bV \d\x. 
$$
The last term in the above inequality is equal to zero thanks to~\eqref{eq:u0-mass}. 
The Lipschitz regularity of $V$ yields 
$$\|\pi_\Dd \bV - V\|_{L^\infty(\O)} \le \|\grad V\|_{L^\infty(\O)^d} h_\Tt,$$
so that  
$$
\Eee_\Dd(\u^0) \le \Eee(u_0) + \|u_0\|_{L^1(\O)} \|\grad V\|_{L^\infty(\O)^d} h_\Tt.
$$
\end{proof} 
\medskip

With all this setting, we can present the scheme we will analyze in this contribution. 
For $\u\in W^{\rm ad} _{\Dd,\bdt} $, we introduce the notation
\be\label{eq:etaks}
\eta_{\k,\s} ^n = \frac{\eta(u^n_\k)+\eta(u^n_\s)} 2, \qquad \forall \k \in \Mm, \; \forall \s \in \Vv_\k, \forall n \in \{1,\dots, N\}.
\ee
Given $\u^{n-1} \in {W_\Dd^{\rm en} } $, the vector $\u^n \in W_\Dd^{\rm ad} $ is obtained by solving the 
following nonlinear system:
\begin{subequations} \label{eq:syst} 
\be\label{eq:k} 
m_\k \frac{u_\k^n - u_\k^{n-1} } {\dt_n} + \sum_{\s \in \Vv_\k} F_{\k,\s} (\u^n) = 0, \qquad \forall \k \in \Mm, 
\ee
\be\label{eq:s} 
m_\s \frac{u_\s^n - u_\s^{n-1} } {\dt_n} + \sum_{\k \in \Mm_\s} F_{\s,\k} (\u^n) = 0, \qquad \forall \s\in\Vv, 
\ee
\be
\label{eq:Fks} 
F_{\k,\s} (\u^n) = \sqrt{\eta_{\k,\s} ^n} \sum_{\s'\in\Vv_\k} a^\k_{\s,\s'} \sqrt{\eta_{\k,\s'} ^n} 
\left(\hf_\k(u_\k^n) - \hf_{\s'}(u_{\s'}^n) \right),
\quad \forall \k \in \Mm, \; \forall \s \in \Vv_\k, 
\ee
\be\label{eq:cons} 
F_{\k,\s} (\u^n) + F_{\s,\k} (\u^n) = 0, \qquad \forall \k \in \Mm, \; \forall \s \in \Vv_\k.
\ee
\end{subequations} 
The scheme~\eqref{eq:syst} can be interpreted as a finite volume scheme, the conservation being 
established on the cells $\omega_\k$ and $\omega_\s$ for $\k \in \Mm$ and $\s \in \Vv$.
As a direct consequence of the conservativity of the scheme, one has 
\be\label{eq:cons-mass} 
\int_\O \pi_\Dd(\u^n) \d\x = \int_\O \pi_\Dd(\u^{n-1}) \d\x.
\ee
However, contrarily to usual finite volume schemes, 
the fluxes $F_{\kappa,\s}^n$ are not issued from the computation of 
$\int_\sigma \eta(u)\Lambda\grad ( p(u) +V)\cdot \n_\sigma$
on a specific boundary $\sigma$ between identified control volumes. 
They result from the variational formulation of the scheme and  
are viewed as fluxes between control volumes located at nodes and centers. 

Defining, for all $\k \in \Mm$ and $\u = (u_\k, u_\s)_{\k,\s} \in W_\Dd$, the diagonal matrix $\M_\k(\u)\in \R^{\ell_\k \times \ell_\k} $ by 
$$
\left(\M_\k(\u)\right)_{\s,\s'} = \begin{cases} 
\sqrt{\frac{\eta(u_\k) + \eta(u_\s)} 2} & \text{ if } \s = \s', \\
0 & \text{ otherwise}, 
\end{cases} 
$$
the systems~\eqref{eq:syst} is equivalent to the following compact formulation: $\forall \bv \in W_\Dd$, 
\be\label{eq:syst-w} 
\int_\O \pi_\Dd \u^n  \pi_\Dd \bv \d\x  + \dt_n \sum_{\k \in \Mm} \bd_\k \bh(\u^n) \cdot \B_\k(\u^n) \bd_\k \bv  
= \int_\O   \pi_\Dd \u^{n-1} \pi_\Dd \bv \d\x, 
\ee
where 
\be\label{eq:Bkn} 
\B_\k(\u) := \M_\k(\u) \A_\k \M_\k(\u), \qquad  \forall \k \in \Mm,\; \forall \u \in W_\Dd,
\ee
 is a symmetric semi-positive matrix since $\A_\k$ and $\M_\k(\u)$ have this property.

\subsection{Gradient flow interpretation for the scheme} \label{ssec:GF-disc} 
The goal of this section is to transpose the formal variational structure pointed out 
in~\S\ref{ssec:GF-cont} to the discrete setting. A natural discretization of the manifold $\Mmm$ consists in 
\be\label{eq:Mmm_Dd} 
\Mmm_\Dd = \left\{ \u \in W_\Dd \; \left|\; \int_\O \pi_\Dd \u \d\x = \int_\O u_0 \d\x \right\} \right., 
\ee
leading to 
\be\label{eq:TMmm_Dd} 
T_\u \Mmm_\Dd = \left\{ \bv \in W_\Dd \; \left|\; \int_\O \pi_\Dd \bv \d\x = 0 \right\} \right..
\ee
In order to define the discrete counterpart $\gf_{\Dd,\u} $ of the metric tensor $\gf_u$ defined 
by~\eqref{eq:g_u}--\eqref{eq:id-elliptic}, one needs a discrete counterpart of
\begin{itemize} 
\item  the classical $L^2(\O)$ scalar product: we will use
$$
(\bw_1, \bw_2) \mapsto \int_\O \pi_\Dd \bw_1 \, \pi_\Dd \bw_2 \d\x, \qquad \forall \bw_1, \bw_2 \in W_\Dd; 
$$
\item the weighted ${H} ^1(\O)$ ``scalar product" with weight $\eta(u)$: we use
$$
(\bw_1, \bw_2) \mapsto \sum_{\k \in \Mm} \bd_\k\bw_1 \cdot \B_\k(\u) \bd_\k \bw_2, 
\qquad \forall \bw_1, \bw_2 \in W_\Dd.
$$
\end{itemize} 
This allows to define the discrete metric tensor $\gf_{\Dd, \u} $ by: 
$\forall \u \in W_\Dd$, $\forall \bw_1, \bw_2 \in T_\u \Mmm_\Dd$,
\be\label{eq:gDdu} 
\gf_{\Dd,\u} (\bw_1, \bw_2) = \int_\O \pi_\Dd \bw_1 \pi_\Dd \bphi_2 \d\x= \int_\O \pi_\Dd \bphi_1 \pi_\Dd \bv_2 \d\x
 = \sum_{\k \in \Mm} \bd_\k \bphi_1 \cdot \B_\k(\u) \bd_\k \bphi_2, 
\ee
where $\bphi_i \in W_\Dd$ solves the discrete counterpart of~\eqref{eq:id-elliptic}, that is 
\be\label{eq:id_Dd} 
\sum_{\k \in \Mm} \bd_\k \bphi_i  \cdot \B_\k(\u) \bd_\k \bpsi = \int_\O \pi_\Dd \bw_i \,\pi_\Dd \bpsi \, \d\x, 
\qquad \forall \bpsi \in W_\Dd.
\ee 
In this setting, we can define the semi-discrete in space gradient flow by 
\begin{multline} \label{eq:semi-space} 
\gf_{\Dd,\u} (\p_t \u, \bw) = \int_\O  \pi_\Dd (\p_t\u)\,  \pi_\Dd \bv \, \d\x \\ 
=  \gf_{\Dd,\u} ( - \grad_u \Eee(\u), \bw) = \int_\O \pi_\Dd \bh(\u) \, \pi_\Dd \bw\, \d\x 
= \sum_{\k} \bd_\k \bh(\u) \cdot \B_\k(\u) \bd_\k \bv, 
\end{multline} 
where $\bv$ solve the discrete elliptic problem 
$$
\sum_{\k \in \Mmm} \bd_\k \bv \cdot \B_\k(\u) \bd_\k \bpsi = \int_\O \pi_\Dd \bw \,\pi_\Dd \bpsi \,\d\x, 
\qquad \forall \bpsi \in W_\Dd.
$$
In order to recover~\eqref{eq:syst-w} from~\eqref{eq:semi-space}, one applies the backward Euler scheme. 

\begin{rem} \label{rem:JKO} 
In their seminal paper~\cite{JKO98}, Jordan, Kinderlehrer and Otto proposed to approximate the solution 
of gradient flows thanks to the minimizing movement scheme 
\be\label{eq:JKO} 
\u^n \in \underset{\u \in \Mmm_\Dd} {\rm argmin} \Big\{ \frac{\mathfrak{d} (\u,\u^{n-1})} {2\dt} + \Eee_\Dd(\u)\Big\}, 
\ee
where $\mathfrak{d} $ denotes the distance on $\Mmm_\Dd$ induced by the metric tensor field 
$\mathfrak g_\Dd$. Several practical and theoretical difficulties arise when one aims at using~\eqref{eq:JKO}. 
First of all, the Riemannian structure is formal, even in the continuous case. It is unclear if one can 
define rigorously a distance $\mathfrak{d} $ if $\eta(0) = 0$ even if $\eta$ is concave (cf.~\cite{DNS09, Maas11}). 
But even if $\mathfrak{d} $ is a distance, 
yielding a metric structure for $\Mmm_\Dd$, computing this distance is a complex problem we avoid 
by using an backward Euler scheme rather than~\eqref{eq:JKO}.  
\end{rem}

\subsection{Main results} \label{ssec:main} 

The first result we want to point out concerns the scheme for a fixed mesh. 
The following theorem states that the scheme~\eqref{eq:syst} admits at least one solution, and justifies 
the \emph{free energy diminishing} denomination for the scheme.

\begin{thm} \label{thm:main-disc} 
Let $\u^{n-1} \in W_\Dd^{\rm en} $, then there exists (at least) one vector $\u^n \in W_\Dd^{\rm ad} $ 
solution to the system~\eqref{eq:syst}, and the following dissipation property holds:
\be\label{eq:entro-disc-V} 
\Eee_\Dd(\u^n) + \dt_n \sum_{\k \in \Mm} \bd_\k \bh(\u^n) \cdot  \B_\k(\u^n) \bd_\k \bh(\u^n)  \le \Eee_\Dd(\u^{n-1}), 
\ee
where $\Eee_\Dd$ is defined by~\eqref{eq:Eee_Dd} and 
$\bh(\u^n) = \left(\hf_\k(u_\k^n), \hf_\s(u_\s^n)\right)_{\k,\s} $ is defined by~\eqref{eq:fhks}. 
\end{thm} 

Since $\u^0\in W_\Dd^{\rm en} $ and since $W_\Dd^{\rm ad} \subset W_\Dd^{\rm en} $, 
Theorem~\ref{thm:main-disc} allows to define the \emph{iterated solution} $\u = (\u^n)_{1\le n \le N} \in W_{\Dd,\bdt} ^{\rm ad} $
to the scheme~\eqref{eq:syst}.

\medskip
The proof of Theorem~\ref{thm:main-disc} is contained in~\S\ref{sec:apriori}, together with some supplementary material
that allows to carry out the convergence analysis when the discretization steps tend to $0$. 
More precisely, we consider  a sequence $\left(\Dd_m\right)_{m\ge 1} = \left(\Mm_m, \Tt_m\right)_{m\ge1} $ 
of discretizations of $\O$ as introduced in \S\ref{sssec:mesh-space}, such that 
\begin{subequations} \label{eq:conv-mesh} 
\be\label{eq:hmto0} 
 h_{\Tt_m} = \max_{T\in\Tt_m} h_T \underset{m \to +\infty} {\longrightarrow} 0, 
\ee
and such that there exists $\theta^\star >0$ and $\ell^\star>0$ satisfying
\be\label{eq:gamma-theta-star} 
\sup_{m\ge1} \theta_{\Tt_m} \le \theta^\star, \qquad \sup_{m\ge1} \ell_{\Dd_m} \le \ell^\star, 
\ee
where $\theta_{\Tt_m} $ and $\ell_{\Dd_m} $ are defined by~\eqref{eq:theta_T} and \eqref{eq:gamma_M} respectively.

Even though it can be avoided in some specific situations, we also do the following assumption, 
allowing to circumvent some technical difficulties: 
\be\label{eq:zeta*} 
\inf_{m\ge 1} \zeta_{\Dd_m} = \zeta^\star >0.
\ee
This means that there is a minimum ratio of
volume allocated to the cell centers and to the nodes in the mass lumping procedure.

Concerning the time-discretizations, we consider a sequence $\left(\bdt_m\right)_{m\ge1} $ of discretizations of $(0,\tf)$ 
as prescribed in~\S\ref{sssec:mesh-time} : 
$$
\bdt_m = \left( \dt_{1,m}, \dots, \dt_{N_m, m} \right), \qquad \forall m \ge 1. 
$$
We assume that the time discretization step tends to $0$, i.e., 
\be\label{eq:dtmto0} 
\ov \dt_m = \max_{1 \le n \le N_m} \dt_{n,m} \underset{m \to +\infty} {\longrightarrow} 0.
\ee
\end{subequations} 
\begin{thm} \label{thm:main-conv} 
Let $\left(\Dd_m, \bdt_m\right)_m$ be a sequence of discretizations of $Q_\tf$ satisfying Assumptions~\eqref{eq:conv-mesh}, 
and let $\left(\u_m\right)_{m\ge1} $, with $\u_m \in W_{\Dd_m, \bdt_m} ^{\rm ad} $, be a corresponding sequence of iterated 
discrete solutions, then 
$$
\pi_{\Dd_m,\bdt_m} \u_m \underset{m\to+\infty} \longrightarrow u \quad \text{strongly in } L^1(Q_\tf), 
$$
where $u$ is the unique weak solution to~\eqref{eq:1} in the sense of Definition~\ref{def:weak}.
\end{thm} 
Proving the Theorem~\ref{thm:main-conv} is the purpose of \S\ref{sec:conv}. 
The practical implementation of the scheme~\eqref{eq:syst} is discussed in~\S\ref{sec:num}, 
where we also give evidences of the efficiency of the scheme.

\section{Proof of Theorem~\ref{thm:main-disc} and additional estimates} \label{sec:apriori} 

In order to ease the reading of the paper, several technical lemmas have been postponed to Appendix. 

\subsection{One-step \emph{A priori} estimates} \label{ssec:one-step} 

\begin{lem} \label{lem:entro1} 
Let $\u^{n-1} \in W_\Dd^{\rm en} $, and let $\u^n \in W_\Dd^{\rm ad} $ be a solution to the scheme~\eqref{eq:syst}, 
then~\eqref{eq:entro-disc-V} holds.
\end{lem} 
\begin{proof} 
Substituting $\bv$ by $\bh(\u^n) = \left(\hf_\k(u_\k^n), \hf_\s(u_\s^n)\right)_{\k,\s}$ defined by~\eqref{eq:fhks} 
in~\eqref{eq:syst-w} yields 
\be\label{eq:AB} 
 \int_\O \!\! \left(\pi_\Dd\u^n - \pi_\Dd\u^{n-1} \right) \pi_\Dd \bh(\u^n) \d\x
 + \dt_n\!\!\sum_{\k\in\Mm} \bd_\k \bh(\u^n) \cdot \B_\k(\u^n) \bd_\k \bh(\u^n) = 0.
 \ee
It follows from the convexity of $\G$ that 
$$
\G(a) - \G(b) \le (a-b) \left(p(a) - p(1)\right), \qquad \forall \; a,b\in \R \; \text{ s.t. } \; \G(a), \G(b) < + \infty.
$$
Hence, using~\eqref{eq:cons-mass}, one has 
\begin{multline*} \label{eq:A} 
 \int_\O \left(\pi_\Dd\u^n - \pi_\Dd\u^{n-1} \right) \pi_\Dd \bh(\u^n) \d\x \\
 =   \int_\O \left(\pi_\Dd\u^n - \pi_\Dd\u^{n-1} \right) \left(p(\pi_\Dd\u^n) + \pi_\Dd(\bV)\right) \d\x \\
 \ge  \int_\O \left(\G(\pi_\Dd \u^n) -\G(\pi_\Dd \u^{n-1}) + \pi_\Dd(\u^n - \u^{n-1}) \pi_\Dd \bV \right) \d\x \\
 =   \Eee_\Dd(\u^n) -  \Eee_\Dd(\u^{n-1}). 
 \end{multline*} 
 Using this inequality in~\eqref{eq:AB} provides~\eqref{eq:entro-disc-V}.
\end{proof} 

\begin{lem} \label{lem:G} 
For all $\eps>0$, there exists $C_\eps \in \R$ depending on $\eps$ and $p$ such that 
$$
|u| \le \eps \G(u) + C_\eps, \qquad \forall u \in W_\Dd^{\rm en}.
$$
\end{lem} 
\begin{proof} 
Fix $\eps>0$, then in view of Assumption~({\bf A}2), the intermediate value theorem ensures the existence of $u_\eps \ge 1$ such that 
$p(u_\eps) = p(1) + 1/\eps$. Then for all $u \in \ov \Ii_p$, one has 
$$
\G(u) = \int_1^u (p(a) - p(1)) \d a = \G(u_\eps) + \int_{u_\eps} ^u (p(a)-p(1)) \d a. 
$$
The function $p$ being increasing, we deduce that 
$$
\G(u) \ge \G(u_\eps) + (p(u_\eps) - p(1)) |u - u_\eps | \ge \G(u_\eps) + \frac1\eps\left(|u| - |u_\eps|\right), 
\qquad \forall u \in \ov \Ii_p.
$$
Lemma~\ref{lem:G} follows with $C_\eps= |u_\eps| - \eps \G(u_\eps).$ 
\end{proof} 

\begin{lem} \label{lem:G2} 
For all $\eps>0$, there exists $C_\eps \in \R$ depending on $\eps$, $\eta$ and $\G$ such that 
$$
\eta(u) \le \eps \G(u) + C_\eps, \qquad \forall u \in \ov \Ii_p. 
$$
\end{lem} 
\begin{proof} 
The function $u\mapsto \frac{\eta(u)} {\G(u)} $ tends to $0$ as $|u| \to \infty$ 
thanks to Assumption~({\bf A}6). Let $\eps >0$, then there exists $r_\eps >0$ such that 
$$
|u| > r_\eps \quad \implies \quad 0 \le \eta(u) \le \eps \G(u).
$$
The function $\eta$ being continuous and nonnegative according to Assumption~({\bf A}1), we know that 
$$
0 \le C_\eps := \max_{u \in [-r_\eps, r_\eps]} \eta(u) < +\infty.
$$
The result of Lemma~\ref{lem:G2} follows. 
\end{proof}

\begin{lem} \label{lem:Eee-G} 
There exist $C_1$ and $C_2$ depending only on $p$, $V$ and $\O$ such that 
\be\label{eq:Eee-G} 
\frac12 \Eee_\Dd(\u) + C_1 \le \int_\O \G(\pi_\Dd\u) \d\x \le 2 \Eee_\Dd(\u) + C_2, \qquad \forall \u \in W_\Dd^{\rm en}.
\ee
In particular, the discrete entropy functional $\Eee_\Dd$ is bounded from below uniformly w.r.t. the discretization $\Dd$.
\end{lem} 
\begin{proof} 
Recall that the discrete entropy  functional $\Eee_\Dd$ is defined by 
$$
\Eee_\Dd(\u) = \int_\O \left(\G(\pi_\Dd\u) + \pi_\Dd\bV\, \pi_\Dd \u\right)\d\x, \qquad \forall \u \in W_\Dd^{\rm en}.
$$
Hence, one has
\be\label{eq:Eee-G-1} 
\int_\O \G(\pi_\Dd\u) \d\x \le \Eee_\Dd(\u) + \|\pi_\Dd \bV\|_{L^\infty(\O)} \|\pi_\Dd\u\|_{L^1(\O)}, 
\qquad \forall \u \in W_\Dd^{\rm en}.
\ee
Let $\eps>0$ a parameter to be fixed later on.
Thanks to Lemma~\ref{lem:G}, there exists a quantity $C_\eps$ depending only on $p$ and $\eps$ such that 
$$
|u|\le \eps \G(u) + C_\eps, \qquad \forall u \in \ov\Ii_p, 
$$
ensuring that 
\be\label{eq:Eee-G-2} 
\|\pi_\Dd \u\|_{L^1(\O)} \le \eps \int_\O \G(\pi_\Dd\u)\d\x + C_\eps\meas(\O),\qquad \forall \u \in W_\Dd^{\rm en}. 
\ee
On the other hand, Assumption~({\bf A}5) together with the definition~\eqref{eq:bV} of 
$\bV = \left(V_\k, V_\s\right)_{\k,\s} $ ensure that 
$$
\|\pi_\Dd\bV\|_{L^\infty(\O)} \le \|V\|_{\infty}.
$$
Setting $\eps = \frac1{2\|V\|_\infty} $ in~\eqref{eq:Eee-G-2} and injecting 
the resulting estimate in~\eqref{eq:Eee-G-1} ends the proof of the second inequality of~\eqref{eq:Eee-G}.
The proof of the first inequality of~\eqref{eq:Eee-G} being similar, it is left to the reader. 
\end{proof} 

\begin{lem} \label{lem:abs-p1} 
There exists $C$ depending only on 
$\L$, $\O$, $\theta_\Tt$, $\zeta_\Dd$, $\ell_\Dd$, $\eta$, $p$ and $V$
 such that, for all 
$\bv=(v_\k, v_\s)_{\k,\s} \in W_\Dd^{\rm ad} $, one has 
\begin{multline} \label{eq:abs-p1} 
\sum_{\k \in \Mm} \sum_{\s \in \Vv_\k} \left(\sum_{\s' \in \Vv_\k} |a^\k_{\s,\s'} | \right) 
\eta_{\k,\s} (\bv) \left(p(v_\k) - p(v_\s)\right)^2 \\
\le C\left( 1 +\Eee_\Dd(\bv) + \sum_{\k \in \Mm} \bd_\k \bh(\bv) \cdot \B_\k(\bv) \bd_\k \bh(\bv) \right),
\end{multline} 
where we have set $\eta_{\k,\s} (\bv) = \frac{\eta(v_\k)  + \eta(v_\s)} 2$ for all $\k \in \Mm$ and all $\s \in \Vv_\k$.
\end{lem} 
\begin{proof} 
Let $\bv \in W_\Dd^{\rm ad} \subset W_\Dd^{\rm en} $, then 
it follows from the definition~\eqref{eq:fhks} of 
$\bh(\bv) = \left(\hf_\k(v_\k), \hf_\s(v_s)\right)_{\k,\s} \in W_\Dd$ that 
$$
 \bd_\k p(\bv) \cdot \B_\k(\bv) \bd_\k p(\bv) 
 \le 2\bd_\k \bh(\bv) \cdot \B_\k(\bv) \bd_\k \bh(\bv)
+ 2 \bd_\k \bV\cdot \B_\k(\bv) \bd_\k \bV, \qquad \forall \k \in \Mm.
$$
It follows from Lemma~\ref{lem:abs-Ak} 
that there exists $C$ 
depending only on $\L$, $\theta_\Tt$ and 
$\ell_\Dd$ such that 
$$
\sum_{\s\in \Vv_\k}\left(\sum_{\s' \in \Vv_\k} |a^\k_{\s,\s'} | \right) \eta_{\k,\s} (\bv) \left(p(v_\k) - p(v_\s)\right)^2 
\le C \bd_\k p(\bv) \cdot \B_\k(\bv) \bd_\k p(\bv), \qquad \forall \k \in \Mm.
$$
Therefore, it only remains to prove that 
\be\label{eq:abs-p1-0} 
\sum_{\k \in \Mm} \bd_\k \bV\cdot \B_\k(\bv) \bd_\k \bV \le \Eee_\Dd(\bv) + C, 
\ee
for some $C$ depending only on the prescribed data. 
Using Lemma~\ref{lem:poids-Aks}, we get that 
\be\label{eq:abs-p1-1} 
\sum_{\k \in \Mm} \bd_\k \bV\cdot \B_\k(\bv) \bd_\k \bV \le
\sum_{\k \in \Mm} 
\max_{\s \in \Vv_\k} \eta_{\k,\s} (\bv) \sum_{\s \in \Vv_\k} 
\left(\sum_{\s' \in \Vv_\k} |a_{\s,\s'} ^\k| \right) \left( V_\k-V_\s\right)^2.
\ee
It results from 
Lemma~\ref{lem:abs-Ak} that for all $\k \in \Mm$, 
\be\label{eq:abs-p1-2} 
\sum_{\s \in \Vv_\k} \left(\sum_{\s' \in \Vv_\k} |a_{\s,\s'} ^\k| \right) \left( V_\k-V_\s\right)^2 
\le C \int_\k \grad_\Tt \bV \cdot \L \grad_\Tt \bV \d\x \le C |\k| \lambda^\star  \|\grad V\|_{\infty} ^2.
\ee
Denote by 
 $ \ov \eta(\bv) = \left(\ov \eta_\k (\bv), \ov \eta_\s(\bv) \right)_{\k,\s} \in W_\Dd$ the vector defined by 
 $$
 \ov \eta_\k (\bv) =  \max\left(\eta(v_k) ; \max_{\s' \in \Vv_\k} \eta(v_\s')\right), \qquad 
 \ov \eta_\s(\bv) = 0, \qquad \forall \k \in \Mm, \; \forall \s \in \Vv, 
 $$
 and remark that 
 $$
 \max_{\s \in \Vv_\k} \eta_{\k,\s} (\bv) \le \ov \eta_\k (\bv), \qquad \forall \k \in \Mm.
 $$
 Hence, we deduce using~\eqref{eq:abs-p1-2} in~\eqref{eq:abs-p1-1} that 
 $$
\sum_{\k \in \Mm} \bd_\k \bV\cdot \B_\k(\bv) \bd_\k \bV \le C \int_\O \pi_\Mm \ov \eta(\bv) \d\x,
 $$
 for some $C$ depending only on $\theta_\Tt$, $\L$, $\ell_\Dd$ and 
 $\|\grad V\|_{\infty} $, the operator $\pi_\Mm$ being defined by~\eqref{eq:pi_Mm}.
 Let us now use Lemma~\ref{lem:Aovbv} to obtain that 
 \be\label{eq:abs-p1-3} 
\sum_{\k \in \Mm} \bd_\k \bV\cdot \B_\k(\bv) \bd_\k \bV \le \wt C \int_\O \pi_\Dd \eta(\bv) \d\x
 \ee
 for some $\wt C$ depending only on the prescribed data, namely $\theta_\Tt$, $\L$, 
 $\ell_\Dd$, $\zeta_\Dd$ and $\|\grad V\|_{\infty} $.
 Using Lemma~\ref{lem:G2}, we know that for all $\eps >0$, there exists $C_\eps$ 
 depending only on $\eps$, $\eta$, $\G$  and $\meas(\O)$ such that 
$$
\int_\O \pi_\Dd \eta(\bv) \d\x \le \eps \int_\O \G(\pi_\Dd\bv)\d\x + C_\eps.
$$
Combining this result with Lemma~\ref{lem:Eee-G} and~\eqref{eq:abs-p1-3}, 
we deduce that for all $\eps >0$, there exists $C_\eps$ depending only 
on $\eps$, $\L$, $\O$, $\theta_\Tt$, $\zeta_\Dd$, $\ell_\Dd$, $\eta$, $p$ and $V$
such that 
$$
\sum_{\k \in \Mm} \bd_\k \bV\cdot \B_\k(\bv) \bd_\k \bV \le \eps \wt C \Eee_\Dd(\bv) + C_\eps, 
\qquad \forall \bv \in W_\Dd^{\rm en}. 
$$
We obtain~\eqref{eq:abs-p1-0} by choosing $\eps = \frac1{\wt C} $. This ends the proof of Lemma~\ref{lem:abs-p1}. 
\end{proof} 

\begin{lem} \label{lem:abs-p2} 
Let $\u^{n-1} \in W_\Dd^{\rm en} $, and let $\u^n \in W_\Dd^{\rm ad} $ be a solution to the scheme~\eqref{eq:syst}.
There exist $C_1$ and $C_2$ 
depending on $\dt_n$, $\L$, $\O$, $\theta_\Tt$, $\zeta_\Dd$, $\ell_\Dd$, $\eta$, $p$ and $V$
such that 
\begin{multline*} 
\sum_{\k \in \Mm} \sum_{\s \in \Vv_\k} \left(\sum_{\s' \in \Vv_\k} |a^\k_{\s,\s'} | \right) 
\eta_{\k,\s} ^n \left(p(u_\k^n) - p(u_\s^n)\right)^2 \\
 \le C_1 \left(1+\Eee_\Dd(\u^{n-1})\right) \le C_2\left( 1+ \int_\O \G(\pi_\Dd \u^{n-1}) \d\x \right). 
\end{multline*} 
\end{lem} 

\begin{proof} 
Since $\u^n$ is a solution of the scheme~\eqref{eq:syst}, the nonlinear discrete stability 
estimate~\eqref{eq:entro-disc-V} holds. Therefore, taking~\eqref{eq:entro-disc-V} into 
account in~\eqref{eq:abs-p1} yields 
$$
\sum_{\k \in \Mm} \sum_{\s \in \Vv_\k} \left(\sum_{\s' \in \Vv_\k} |a^\k_{\s,\s'} | \right) 
\eta_{\k,\s} ^n \left(p(u_\k^n) - p(u_\s^n)\right)^2 
 \le C_1 \left(1+\Eee_\Dd(\u^{n-1})\right)
$$
for some $C_1$ depending on the prescribed data. 
Then it only remains to use Lemma~\ref{lem:Eee-G} to conclude the proof of Lemma~\ref{lem:abs-p2}.
\end{proof} 

\subsection{Existence of a discrete solution} \label{ssec:exist-disc} 

The scheme~\eqref{eq:syst} can be rewritten in the form of a nonlinear system
$$
\Ff(\u^n) = {\bf 0} _{\R^{\#\Mm+\#\Vv} }.
$$
In the case where $p(0) = -\infty$, the function $\Ff$ is continuous on $W_\Dd^{\rm ad} $, 
but not uniformly continuous. The existence proof for a discrete solution we propose relies on a 
topological degree argument (see e.g.~\cite{LS34,Dei85}), whence we need to restrict the 
set of the possible $\u^n$ for recovering the uniform continuity by avoiding the singularity near $0$. 
This is the purpose of the following lemma, which is an adaptation of~\cite[Lemma~3.10]{CG16_MCOM}.

\begin{lem} \label{lem:harnack} 
Let $\u^{n-1} \in W_\Dd^{\rm en} $ be such that $\int_\O \pi_\Dd \u^{n-1} \d\x >0$ and let $\u^n$ be a solution to the scheme~\eqref{eq:syst}.
Assume that $p(0) = -\infty$, then there exists $\eps_{\Dd,\dt_n} >0$ depending on $\dt_n$, $\Dd$, $\L$, $\O$, $\eta$, $p$, $V$, and $\Eee_\Dd(\u^{n-1})$ such that 
$$
u_\nu^n \ge \eps_{\Dd,\dt_n}, \qquad \forall \nu \in \Mm \cup \Vv.
$$
\end{lem} 
\begin{proof} 
First of all, remark that proving Lemma~\ref{lem:harnack} is equivalent to proving that there exists $C_{\Dd,\dt_n} >0$ such that 
\be\label{eq:harnack-p} 
p(u_\nu^n) \ge -C_{\Dd,\dt_n}, \qquad \forall \nu \in \Mm \cup \Vv.
\ee

Because of the conservation of mass~\eqref{eq:cons-mass}, we have 
$$
\int_\O \pi_\Dd \u^n \d\x = \int_\O \pi_\Dd \u^{n-1} \d\x >0.
$$
Therefore, we can claim that there exists $\nu_{\rm i} \in \Mm \cup \Vv$ such that 
\be\label{eq:u_nu_i} 
u_{\nu_{\rm i} } ^n \ge \frac1{\meas(\O)} \int_\O \pi_\Dd \u^{n-1} \d\x >0.
\ee
Let $(\nu_{\rm f})  \in \Mm \cup \Vv$ be arbitrary, and 
 $(\nu_q)_{q=0,\dots,\ell} $ be a path from $\nu_{\rm i} $ to $\nu_{\rm f} $, 
i.e. 
\begin{itemize} 
\item $\nu_0 = \nu_{\rm i} $, $\nu_\ell = \nu_{\rm f} $, and $\nu_p \neq \nu_q$ if $p \neq q$;
\item for all $q \in \{0,\dots, \ell-1\} $, one has: 
$$
\nu_q \in \Mm \implies \nu_{q+1} \in \Vv_{\nu_{q} }, \quad \text{ and } 
\quad \nu_q \in \Vv \implies \nu_{q+1} \in \Mm_{\nu_{q} }.
$$
\end{itemize} 
Let $q \in \{0,\dots,\ell-1\} $

It follows from Lemma~\ref{lem:abs-p2} that there exists $C_{\Dd,\dt_n} $ depending on $\Dd, \dt_n$ 
 $\L$, $\O$, $\eta$, $p$,  $V$, and $\Eee_\Dd(\u^{n-1})$ such that 
$$
\sum_{\k\in\Mm} \sum_{\s\in\Vv_\k} \eta_{\k,\s} ^n (p(u_\k^n) - p(u_\s^n))^2 \le C_{\Dd,\dt_n}.
$$
This ensures in particular that 
\be\label{eq:estim-path} 
\sum_{q = 0} ^{\ell-1} \eta_{\nu_{q},\nu_{q+1} } ^n (p(u_{\nu_q} ^n) - p(u_{\nu_{q+1} } ^n))^2 \le C_{\Dd,\dt_n}, 
\ee
where we have set $\eta_{\nu_{q},\nu_{q+1} } ^n = \eta_{\k,\s} ^n$ if $\{\nu_{q},\nu_{q+1} \} = \{\k,\s\} $.

We can now prove~\eqref{eq:harnack-p} thanks to an induction along the path.  
Assume that $u_{\nu_q} ^n > \eps_{\Dd,\dt_n} $ for some $\eps_{\Dd,\dt_n} >0$, whence 
$
\eta_{\nu_q,\nu_{q+1} } ^n \ge \frac{\eta(u_{\nu_{q} } ^n)} 2\ge  \eps_{\Dd,\dt_n} ' >0.
$ 
Then it follows from~\eqref{eq:estim-path} that 
$$
p(u_{\nu_{q+1} } ^n) \ge p(u_{\nu_q} ^n) - \sqrt{\frac{C_{\Dd, \dt_n} } {\eps_{\Dd,\dt_n} '} } \ge -C'_{\Dd, \dt_n} 
\quad \implies \quad u_{\nu_{q+1} } ^n \ge \eps_{\Dd,\dt_n} '' >0.
$$
We conclude as in~\cite[Lemma 3.10]{CG16_MCOM} thanks to the finite number of possible paths.
\end{proof} 

Thanks to Lemma~\ref{lem:harnack}, one can apply the same strategy as in~\cite{CG16_MCOM} for 
proving the existence of a solution to the scheme~\eqref{eq:syst}. 
\begin{prop} \label{prop:exists} 
Let $\u^{n-1} \in W_\Dd^{\rm en} $, then there exists (at least) one vector $\u^n \in W_\Dd^{\rm ad} $ 
solution to the system~\eqref{eq:syst}.
\end{prop} 
\begin{proof}
As in~\cite[Proposition~3.11]{CG16_MCOM}, the proof relies on a topological degree argument 
(cf.~\cite{LS34, Dei85}) applied twice. More precisely, start from the parametrized nonlinear 
problem that consists in looking for $\u^{n,\gamma}$ solution to  
\begin{multline}\label{eq:scheme-gamma}
\int_\O \pi_\Dd \left(\frac{\u^{n,\gamma}- \u^{n-1}}{\dt_n}\right) \pi_\Dd \bv \d\x 
	+ \gamma \sum_{\s \in \Vv_\k} \sum_{\s' \in \Vv_\k} a^\k_{\s,\s'} \left( \hf_\k(u_\k^{n,\gamma}) - \hf_\s(u_\s^{n,\gamma}) \right)(v_\k - v_\s') 
\\
+  
(1-\gamma)\sum_{\s \in \Vv_\k} \left(\sum_{\s' \in \Vv_\k} |a^\k_{\s,\s'}|\right) \left(\hf_\k(u_\k^{n,\gamma}) - \hf_\s(u_\s^{n,\gamma})\right) (v_\k - v_\s) 
	= 0, \qquad \forall \bv \in W_\Dd. 
\end{multline}
For $\gamma = 0$, the corresponding scheme is monotone, hence the nonlinear system~\eqref{eq:scheme-gamma} 
admits a unique solution $\u^{n,0}\in W_\Dd^{\rm ad}$, and its corresponding topological degree is equal to 1 
(see for instance~\cite{EGGH98} for the application of this argument to the case of a pure hyperbolic equation).
In the case where $\lim_{u\searrow 0} p(u) = -\infty$, one can prove as in Lemma~\ref{lem:harnack} that any solution 
$\u^{n,\gamma} \in W_\Dd^{\rm ad}$ to~\eqref{eq:scheme-gamma} satisfies 
$$
u_\beta^{n,\gamma} \ge \eps_{\Dd,\dt_n}, \quad  \forall \beta \in \Mm \cup \Vv
$$
for some $\eps_{\Dd,\dt_n}>0$ not depending on $\gamma$. The convex subset on which one looks for a 
solution $\u^{n,\gamma}$ can be restricted to the subset $W_\Dd^{\rm ad}$ defined by 
$$
K:=\left\{ \u \in W_\Dd^{\rm ad} \; \middle| \; u_\beta \ge \frac{\eps_{\Dd, \dt_n}}2 \; \text{and}\; \Eee_{\Dd}(\u) \le \Eee_\Dd(\u^{n-1})+1\right\}, 
$$
the corresponding topological degree being still equal to $1$. Note that the bound $u_\beta \ge \frac{\eps_{\Dd, \dt_n}}2$ 
must be removed if $p(0)$ is finite. This ensures the existence of at least one solution 
to the nonlinear system~\eqref{eq:scheme-gamma} when $\gamma$ is equal to 1. 

Starting from the system~\eqref{eq:scheme-gamma} with $\gamma=1$, one defines a second homotopy 
parametrized by $\mu \in [0,1]$ to get~\eqref{eq:syst-w}. 
More precisely (the superscript $\gamma=1$ has been removed for clarity), we set 
$$\eta^\mu: u \mapsto  1 + \mu (\eta(u) - 1), \qquad \forall \mu \in [0,1],$$
so that $\eta^0$ is constant equal to $1$ and $\eta^1\equiv \eta$.
Define $\u^{n,\mu} \in W_\Dd^{\rm ad}$ as a solution to
\begin{multline}\label{eq:scheme-mu}
\int_\O \pi_\Dd \left(\frac{\u^{n,\mu}- \u^{n-1}}{\dt_n}\right) \pi_\Dd \bv \d\x \\
	+ \sum_{\s \in \Vv_\k} \sqrt{\eta_{\k,\s}^{n,\mu}}\sum_{\s' \in \Vv_\k} a^\k_{\s,\s'} 
	\sqrt{\eta_{\k,\s'}^{n,\mu}} \left( \hf_\k(u_\k^{n,\mu}) - \hf_\s(u_\s^{n,\mu}) \right)(v_\k - v_\s')
	= 0,
\end{multline}
where $\bv \in W_\Dd$ is arbitrary, and where 
$$
\eta_{\k,\s'}^{n,\mu} = \frac{\eta^\mu(u_\k^n) + \eta^\mu(u_\s^n)}2, \qquad \forall \k \in \Mm, \, \forall \s \in \Vv_\k.
$$
{\em A priori} estimates similar to those derived previously in the paper ensure that the solutions to 
\eqref{eq:scheme-mu} cannot belong to $\p K$ whatever the value of $\mu \in [0,1]$. The existence of 
a solution to the scheme~\eqref{eq:syst} follows. 
\end{proof}

\subsection{Multistep \emph{a priori} estimates} \label{ssec:multi-step} 
As a byproduct of the existence of a discrete solution $\u^n$ for all $n\in\{1,\dots,N\} $, we can 
now derive \emph{a priori} estimates on functions reconstructed thanks to the discrete solution $\u \in W_{\Dd,\bdt} $.

The first estimate we get is obtained by summing Ineq.~\eqref{eq:entro-disc-V} w.r.t. $n$, and by using the 
positivity of the dissipation. This provides 
\be\label{eq:entro-disc-max} 
\max_{n\in \{1,\dots, N\} } \Eee_\Dd(\u^n) 
\le \Eee_\Dd(\u^0) \le \Eee(u_0) + C \le C, 
\ee
where $C$ only depends on $V$, $u_0$, $p$ and $\O$ thanks to Lemma~\ref{lem:entro-0}.
Since the discrete entropy functional $\Eee_\Dd$ is bounded from below by a quantity depending 
only on $p$, $V$ and $\O$ (cf. Lemma~\ref{lem:Eee-G}), we deduce also from the summation 
of~\eqref{eq:entro-disc-V} w.r.t. $n$ that there exists $C$ depending only on $p$, $V$, $\O$, and $u_0$ 
(but not on $\Dd$)
such that 
\be\label{eq:dissip-disc-int} 
 \sum_{n=1} ^N \dt_n \sum_{\k \in \Mm} \bd_\k \bh(\u^n) \cdot \B_\k(\u^n)  \bd_\k \bh(\u^n)
\le C.
\ee
Mimicking the proof of Lemma~\ref{lem:abs-p1}, this yields 
\be\label{eq:dissip-disc-int2} 
 \sum_{n=1} ^N \dt_n \sum_{\k \in \Mm} \bd_\k p(\u^n) \cdot \B_\k(\u^n)  \bd_\k p(\u^n) 
 \le C
\ee
for some quantity $C$ depending on $\L$, $\O$, $\theta_\Tt$, $\zeta_\Dd$, $\ell_\Dd$, $\eta$, $p$, $V$
and $\tf$.

The following lemma is a direct consequence of Estimate~\eqref{eq:entro-disc-max} and Lemma~\ref{lem:Eee-G}.
Its detailed proof is left to the reader.
\begin{lem} \label{lem:G-int} 
There exists $C$ depending only on $\O$, $p$, $V$, $u_0$ and $\O$ (but not on the discretization) such that 
$$
\left\| \G(\pi_{\Dd,\bdt} \u)\right\|_{L^\infty((0,\tf);L^1(\O))} \le C. 
$$
\end{lem} 
\medskip

The introduction of the Kirchhoff transform was avoided in our scheme. 
Its extension to complex problems (like e.g., systems, problems with hysteresis) in therefore easier.
However, the (semi-)Kirchhoff transform $\xi$ defined by~\eqref{eq:xi} is useful 
for carrying the analysis out. The purpose of the following lemma is to provide a 
discrete $L^2((0,T);H^{1}(\O))$ estimate on $\xi(\u)$.

\begin{lem} \label{lem:entro2} 
There exists $C>0$depending only on $\L$, $\eta$, $V$, $\theta_\Tt$ and $\ell_\Dd$, $\O$, $\tf$, $u_0$  such that 
$$
\iint_{Q_\tf} \L \grad_{\Tt,\bdt} \xi(\u) \cdot  \grad_{\Tt,\bdt} \xi(\u) \d\x \le 
C.
$$
\end{lem} 
\begin{proof} 
Since the function $\eta$ was assumed to be nondecreasing, see Assumption ({\bf A}1).
We know that for all interval $[a,b] \subset \Ii_p$, 
$
\max_{c \in [a,b]} \eta(c) = \max\{ \eta(a), \eta(b)\}, 
$
hence, denoting by $\Ii_{\k,\s} ^n$ the interval with extremities $u_\k^n$ and $u_\s^n$, we obtain that 
\be\label{eq:eta>max/2} 
\eta_{\k,\s} ^n \ge \frac12 \max_{c \in \Ii_{\k,\s} ^n} \eta(c), \qquad \forall \k \in \Mm, \; \forall \s \in \Vv_\k.
\ee
The definition~\eqref{eq:xi} of the function $\xi$ implies that 
$$
\left(\xi(u_\s^n) - \xi(u_\k^n)\right)^2 \le\left( \max_{c \in \Ii_{\k,\s} ^n} \eta(c)\right)  \left(p(u_\s^n) - p(u_\k^n)\right)^2, 
$$
whence we obtain that for all $\k \in \Mm$, 
\begin{align*} 
\left| \M_\k(\u^n) \bd_\k p(\u^n) \right|^2 = &\sum_{\s \in \Vv_\k} \eta_{\k,\s} ^n \left( p(u_\s^n) - p(u_\k^n) \right)^2 \\
\ge &\frac12 \sum_{\s \in \Vv_\k} \left(\xi(u_\s^n) - \xi(u_\k^n)\right)^2 = \frac12{\left| \bd_\k \xi(\u^n) \right|^2}.
\end{align*} 
Using that 
$${\bv} \cdot \A_\k {\bv} \ge \bw \cdot \A_\k \bw, \qquad \forall  \; \bv,\bw \in \R^{\ell_\k} 
 \;\text{ s.t. } \; |\bv|^2 \ge {\rm Cond} _2(\A_\k) |\bw|^2,$$   
 and that $\B_\k(\u) = \M_\k(\u) \A_\k  \M_\k(\u)$, 
 we get that for all $\k \in \Mm$, 
 \begin{align*} 
 \bd_\k p(\u^n) \cdot \B_\k(\u^n) \bd_\k p(\u^n)  \ge &
\; \frac{1} {2 {\rm Cond} _2(\A_\k)} \bd_\k \xi (\u^n) \cdot \A_\k \bd_\k \xi (\u^n) \\
 = &\;  \frac{1} {2 {\rm Cond} _2(\A_\k)} \int_\k \L \grad_\Tt \xi(\u^n) \cdot  \grad_\Tt \xi(\u^n) \d\x.
 \end{align*} 
Thanks to Lemma~\ref{lem:cond-Ak} stated in 
 appendix, we know that $C>0$ depending only on $\L, \theta_\Tt$ and $\ell_\Dd$ such 
 that ${\rm Cond} _2(\A_\k) \le C$, for all $\k \in \Mm$, so that: $\forall \k \in \Mm$, $\forall n \in \{1,\dots, N\} $,
 \be\label{eq:entro2_1} 
 \int_\k \grad_\Tt \xi(\u^n) \cdot \L \grad_\Tt \xi(\u^n) \d\x \le C  \bd_\k p(\u^n) \cdot \B_\k(\u^n) \bd_\k p(\u^n). 
 \ee
 In order to conclude the proof, it only remains to multiply~\eqref{eq:entro2_1} by $\dt_n$, to sum 
 over $\k \in \Mm$ and $n \in \{1,\dots, N\} $, and finally to use~\eqref{eq:dissip-disc-int2}.
 \end{proof} 

Combining Estimate~\eqref{eq:dissip-disc-int2} and Lemma~\ref{lem:abs-Ak} yields the following lemma, 
whose complete proof is left to the reader. 

\begin{lem} \label{lem:dissip-abs} 
There exists $C$ depending only on $\L$, $\O$, $\theta_\Tt$, $\zeta_\Dd$, $\ell_\Dd$, $\eta$, $p$, $V$
and $\tf$ such that 
$$
 \sum_{n=1} ^N \dt_n \sum_{\k \in \Mm} 
 \sum_{\s \in \Vv_\k} \left(\sum_{\s \in \Vv_\k} |a_{\s,\s'} ^\k|\right) \eta_{\k,\s} ^n \left(p(u_\k^n) - p(u_\s^n)\right)^2 \le C. 
$$
\end{lem} 

\section{Proof of Theorem~\ref{thm:main-conv} } \label{sec:conv} 

In what follows, we consider a sequence $(\Dd_m, \bdt_m)_{m\ge1} $ of discretizations of $Q_\tf$ 
such that~\eqref{eq:conv-mesh} holds. In order to prove the convergence of the reconstructed discrete solution 
$\pi_{\Dd_m, \bdt_m} \u_m$ towards the weak solution of~\eqref{eq:1} as $m$ tends to $\infty$, 
we adopt the classical strategy 
that consists in showing first that the family $\left(\pi_{\Dd_m, \bdt_m} \u_m\right)_{m\ge1} $ is precompact 
in $L^1(Q_\tf)$ (this is the purpose of \S\ref{ssec:compact}), 
then to identify in \S\ref{ssec:identify} the limit as a weak solution of~\eqref{eq:1} in the sense of 
Definition~\ref{def:weak}.

As a direct consequence of Theorem~\ref{thm:main-disc}, one knows that 
the scheme admits a solution $\u_m = \left(u^n_{\k,m}, u^n_{\s,m} \right)$ that, thanks to 
the regularity assumptions~\eqref{eq:gamma-theta-star}--\eqref{eq:zeta*} on the discretization and 
thanks to Lemmas~\ref{lem:G-int}, \ref{lem:entro2} and~\ref{lem:dissip-abs}, satisfies the 
following uniform estimates w.r.t. $m$:
\be\label{eq:G-int_m} 
\left\| \pi_{\Dd_m,\bdt_m} \G(\u_m)\right\|_{L^\infty((0,\tf);L^1(\O))} \le C, 
\ee
\be\label{eq:dissip-xi_m} 
\iint_{Q_\tf} \grad_{\Tt_m, \bdt_m} \xi(\u_m) \cdot \L \grad_{\Tt_m, \bdt_m} \xi(\u_m) \d\x\d t \le C, 
\ee
\be\label{eq:dissip-abs_m} 
 \sum_{n=1} ^{N_m} \dt_{n,m} \sum_{\k \in \Mm_m} 
 \sum_{\s \in \Vv_\k} \left(\sum_{\s \in \Vv_\k} |a_{\s,\s'} ^\k|\right) \eta_{\k,\s} ^n \left(p(u_{\k,m} ^n) - p(u_{\s,m} ^n)\right)^2 \le C, 
\ee
where $C$ may depend on the data of the continuous problem, and on the discretization 
regularity factors $\theta^\star$, $\ell^\star$ and $\zeta^\star$ but not on $m$.

\subsection{Compactness properties of the discrete solutions} \label{ssec:compact} 

\begin{lem} \label{lem:xi-L2} 
Let $(\Dd_m, \bdt_m)$ be a sequence of discretizations of $Q_\tf$ satisfying Assumptions~\eqref{eq:conv-mesh}, 
there exists $C$ depending only on $\L$, $\theta^\star$, $\ell^\star$, $\O$, $\tf$,  $p$  and $u_0$ such that, for all $m\ge1$, one has
$$
\left\| \pi_{\Tt_m, \bdt_m} \xi(\u_m) \right\|_{L^2((0,\tf);H^1(\O))} \le C \quad \text{ and } \quad \left\| \pi_{\Dd_m, \bdt_m} \xi(\u_m) \right\|_{L^2(Q_\tf)} \le C.
$$
\end{lem} 
\begin{proof} 
It follows from the Estimate~\eqref{eq:dissip-xi_m} and from Lemma~\ref{lem:34-BM13} 
that for all $m \ge 1$, 
\begin{multline} \label{eq:Dd-Tt-xi} 
\left\| \pi_{\Dd_m,\bdt_m} \xi(\u_m) -  \pi_{\Tt_m,\bdt_m} \xi(\u_m) \right\|_{L^2((0,\tf);L^1(\O))} \\
\le  \meas(\O)^{1/2} \left\| \pi_{\Dd_m,\bdt_m} \xi(\u_m) -  \pi_{\Tt_m,\bdt_m} \xi(\u_m) \right\|_{L^2(Q_\tf)} \le C
\end{multline} 
for some $C$ depending only on~$\L$, $\O$, $\tf$, $\ell^\star$, $\theta^\star$, $p$, and $u_0$ (but not on $m$).

Moreover, it follows from Assumption~\eqref{eq:C_xi} that 
$$
\left\|\pi_{\Dd_m, \bdt_m} \xi(\u_m) \right\|_{L^\infty((0,\tf); L^1(\O))} \le 
C\left(1+ \left\|\pi_{\Dd_m, \bdt_m} \G(\u_m) \right\|_{L^\infty((0,\tf); L^1(\O))} \right) \le C
$$
thanks to the estimate~\eqref{eq:G-int_m}. Combining this inequality with~\eqref{eq:Dd-Tt-xi} provides that 
$$
\left\|\pi_{\Tt_m, \bdt_m} \xi(\u_m) \right\|_{L^1(Q_\tf)} \le C, 
$$
whence the sequence $\left(\pi_{\Tt_m, \bdt_m} \xi(\u_m) \right)_{m\ge1} $ is bounded 
in $L^2((0,\tf);W^{1,1} (\O))$ thanks to~\eqref{eq:dissip-xi_m}. 
A classical bootstrap argument using Sobolev inequalities allows to claim that it is 
bounded in $L^2((0,\tf);H^1(\O))$, thus in particular in $L^2(Q_\tf)$. 
One concludes that $\left(\pi_{\Tt_m, \bdt_m} \xi(\u_m) \right)_{m\ge1} $ is 
also bounded in $L^2(Q_{\tf})$ thanks to~\eqref{eq:Dd-Tt-xi}.
\end{proof} 

\begin{rem} \label{rem:Igbida} 
An alternative way to prove the key-point of Lemma~\ref{lem:xi-L2}, namely
$$
\left\| \pi_{\Tt_m, \bdt_m} \xi(\u_m) \right\|_{L^2((0,\tf);H^1(\O))} \le C, 
$$
would consist in using~\cite[Lemma A.1]{Igb07}, that shows that 
$$
v \in L^1(\O) \text{ and } \grad \xi(v) \in L^2(\O)^d \quad  \implies \quad \xi(v) \in H^1(\O). 
$$
\end{rem} 

As a consequence of Lemma~\ref{lem:xi-L2}, we know that the sequence 
$\left(\pi_{\Tt_m, \bdt_m} \xi(\u_m)\right)_{m\ge1} $ is relatively compact for the $L^2((0,\tf);H^1(\O))$-weak 
topology. Moreover, the space $H^1(\O)$ being locally compact in $L^2(\O)$, a uniform information on the 
time translates of  $\pi_{\Tt_m, \bdt_m} \xi(\u_m)$ will provide the relative compactness of 
$\left(\pi_{\Tt_m, \bdt_m} \xi(\u_m)\right)_{m\ge1} $ in the $L^2(Q_\tf)$-strong topology (see e.g.~\cite{Sim87}).
Such a uniform time-translate estimate can be obtained by using directly the numerical scheme (see e.g.~\cite{EGH00,CG16_MCOM}). One can also make use of black-boxes like e.g.~\cite{And11,ACM}. Note 
that the result of~\cite{GL12} does not apply here because of the degeneracy of the problem. 
We do not provide the proof of next proposition here, since a suitable black-box will be contained 
in the forthcoming contribution~\cite{ACM}. A more classical but calculative possibility would consist in 
mimicking the proof of~\cite[Lemmas 4.3 and 4.5]{CG16_MCOM}.

\begin{prop} \label{prop:compact} 
Let $(\Dd_m, \bdt_m)$ be a sequence of discretizations of $Q_\tf$ satisfying Assumptions~\eqref{eq:conv-mesh}, 
and let $(\u_m)_{m\ge1} $ be the corresponding sequence of solutions to the scheme~\eqref{eq:syst}.
Then there exists a measurable function $u: Q_\tf \to \R$ with $\xi(u) \in L^2((0,\tf);H^1(\O))$ such that, 
up to a subsequence, one has 
\be\label{eq:ae-conv} 
\pi_{\Dd_m, \bdt_m} \u_m \underset{m\to\infty} {\longrightarrow} u \quad \text{ a.e. in } Q_\tf.
\ee
\end{prop} 

\begin{coro} \label{coro:convL1} 
Keeping the assumption and notations of Proposition~\ref{prop:compact}, one has 
$$
\pi_{\Dd_m, \bdt_m} \u_m \underset{m\to\infty} {\longrightarrow} u \quad \text{ strongly in } L^1(Q_\tf).
$$
\end{coro} 
\begin{proof} 
As a result of Proposition~\ref{prop:compact}, the almost everywhere convergence property~\eqref{eq:ae-conv} 
holds. On the other hand, it follows from Assumption~({\bf A}2), more precisely from the fact that 
$\lim_{u \to \infty} p(u) = +\infty$ that the function $\G$ defined by~\eqref{eq:G} is superlinear, i.e., 
$$\lim_{u \to +\infty} \frac{\G(u)} {u} = +\infty.$$
Therefore, Estimate~\eqref{eq:G-int_m} implies that $\left(\pi_{\Dd_m, \bdt_m} \u_m\right)_{m\ge1} $ is uniformly 
equi-integrable. Hence we can apply Vitali's convergence theorem to conclude the proof of Corollary~\ref{coro:convL1}.
\end{proof}

\begin{lem} \label{lem:conv-xi} 
Under the assumptions of Proposition~\ref{prop:compact}, one has 
$$
\pi_{\Tt_m, \bdt_m} \xi(\u_m) \underset{m\to\infty} {\longrightarrow} \xi (u) \quad \text{ weakly in } L^2((0,\tf);H^1(\O)), 
$$
where $u$ is the solution exhibited in Proposition~\ref{prop:compact}.
\end{lem} 
\begin{proof} 
Thanks to Lemma~\ref{lem:xi-L2}, 
the sequence 
$
\left(\grad_{\Tt_m, \bdt_m} \xi(\u_m) \right)_{m\ge1} 
$
is uniformly bounded in $L^2(Q_\tf)^d$. Therefore, there exists $\Xi \in L^2((0,\tf);H^1(\O))$ such that 
$$
\pi_{\Tt_m, \bdt_m} \xi(\u_m) \underset{m\to\infty} {\longrightarrow} \Xi \quad \text{ weakly in } L^2((0,\tf);H^1(\O)).
$$
But in view of Proposition~\ref{prop:compact} and of the continuity of $\xi$, we know that 
$$
\xi\left(\pi_{\Dd_m, \bdt_m} \u_m\right) = \pi_{\Dd_m, \bdt_m} \xi(\u_m) \underset{m\to\infty} {\longrightarrow} \xi(u) \quad \text{a.e. in } Q_\tf. 
$$
Since $\pi_{\Tt_m, \bdt_m} \xi(\u_m)$ and $\pi_{\Dd_m, \bdt_m} \xi(\u_m)$ have the same limit (cf. Lemma~\ref{lem:34-BM13}), 
we get that $\Xi = \xi(u)$.
\end{proof}

\begin{lem} \label{lem:eta_M} 
Let $u$ be the limit value of $\pi_{\Dd_m, \bdt_m} \u_m$ obtained in Proposition~\ref{prop:compact}, 
then
\be\label{eq:piD-eta} 
\pi_{\Dd_m, \bdt_m} \eta(\u_m) \underset{m\to\infty} {\longrightarrow} \eta(u) \quad \text{ strongly in } L^1(Q_\tf),
\ee
and
\be\label{eq:piM-eta} 
\pi_{\Mm_m, \bdt_m} \eta(\u_m)\underset{m\to\infty} {\longrightarrow} \eta(u) \quad \text{ strongly in } L^1(Q_\tf).
\ee
\end{lem} 
\begin{proof} 
Let us first establish~\eqref{eq:piD-eta}. Thanks to the entropy estimate~\eqref{eq:G-int_m}, we know that 
the sequence $\left( \pi_{\Dd_m, \bdt_m} \G(\u_m)\right)_{m\ge1} $ is uniformly bounded in $L^\infty((0,\tf);L^1(\O))$, 
thus in $L^1(Q_{\tf})$. Then Assumption~\eqref{eq:dlVP} allows to use the de la Vall\'ee-Poussin theorem 
to claim that $\left( \pi_{\Dd_m, \bdt_m} \eta(\u_m)\right)_{m\ge1} $ is uniformly equi-integrable on $Q_\tf$.
Moreover, the continuity of $\eta$ and Proposition~\ref{prop:compact} provide that 
$$
\left( \pi_{\Dd_m, \bdt_m} \eta(\u_m)\right)_{m\ge1} \underset{m\to\infty} \longrightarrow \eta(u) 
\quad \text{ a.e. in } Q_\tf. 
$$
Therefore, we obtain~\eqref{eq:piD-eta} by applying Vitali's theorem.
\smallskip

Let us now prove~\eqref{eq:piM-eta} by proving that $\pi_{\Dd_m, \bdt_m} \eta(\u_m)$ and 
$\pi_{\Mm_m, \bdt_m} \eta(\u_m)$ (that is uniformly equi-integrable for the same reasons as 
$\pi_{\Dd_m, \bdt_m} \eta(\u_m)$ is) have the same limit $\eta(u)$ as $m$ tends to $\infty$. 
It follows from a combination of~\eqref{eq:dissip-xi_m} with 
Lemma~\ref{lem:34-BM13} that, still up to an unlabeled subsequence,  
\be\label{eq:xiDM}
\pi_{\Dd_m,\bdt_m} \xi(\u_m) - \pi_{\Mm_m,\bdt_m} \xi(\u_m)\;  \underset{m\to\infty}{\longrightarrow} \; 0 \quad \text{a.e. in } Q_\tf. 
\ee
Since the function $\sqrt{\eta\circ\xi^{-1} } $ is assumed to be uniformly continuous (cf.~\eqref{eq:H-unif-cont}), 
it admits a non-decreasing modulus of continuity $\varpi \in C(\R_+ ; \R_+)$ with $\varpi(0) = 0$ such that, 
for all $v,\h v$ in the range of $\xi$, 
\be\label{eq:varpi} 
\left| \sqrt{\eta\circ\xi^{-1} } (v) - \sqrt{\eta\circ\xi^{-1} } (\h v) \right| \le  \varpi(|v-\h v|), 
\ee
so that 
$$
\left|\sqrt{\pi_{\Dd_m, \bdt_m}\eta( \u_m)} - \sqrt{\pi_{\Mm_m, \bdt_m} \eta(\u_m)}\right| 
\le  \varpi(|\pi_{\Dd_m,\bdt_m} \xi(\u_m) - \pi_{\Mm_m,\bdt_m} \xi(\u_m)|).
$$
Therefore, it follows from~\eqref{eq:xiDM} that 
$$
\sqrt{\pi_{\Dd_m, \bdt_m}\eta( \u_m)} - \sqrt{\pi_{\Mm_m, \bdt_m} \eta(\u_m)}  
 \underset{m\to\infty}{\longrightarrow} \; 0 \quad \text{a.e. in } Q_\tf. 
$$
Thus $\left(\pi_{\Dd_m, \bdt_m}\eta( \u_m)\right)_{m\ge1}$ and $\left(\pi_{\Mm_m, \bdt_m} \eta(\u_m)\right)_{m\ge1}$ 
share the same limit. 
\end{proof}

\subsection{Identification of the limit as a weak solution} \label{ssec:identify}

\begin{prop} \label{prop:identify} 
Let $u$ be a limit value of the sequence $\left(\pi_{\Dd_m, \bdt_m} \u_m \right)_{m\ge1} $ exhibited in Proposition~\ref{prop:compact}, 
then $u$ is the unique weak solution to the problem~\eqref{eq:1} in the sense of Definition~\ref{def:weak}.
\end{prop} 
\begin{proof} 
In order to check that $u$ is a weak solution, it only remains to check that the weak formulation~\eqref{eq:weak} holds. 
Let $\psi \in C^\infty_c(\ov\O\times[0,\tf))$, then, for all $m \ge 1$, for all $\beta \in \Mm_m \cup \Vv_m$ and 
all $n \in \{0,\dots, N_m\} $, we denote by $\bdt_m  = \left(\dt_{1,m},\dots, \dt_{N_m,m} \right)$, by $t_{n,m} = \sum_{i=1} ^n \dt_{i,m} $, 
by $\psi_\b^n = \psi(\x_\b, t_{n,m})$, 
by $\bpsi^n_m = \left(\psi_\beta^n\right)_{\beta\in\Mm_m\cup\Vv_m} \in W_{\Dd_m} $, and by 
$\bpsi_m = \left(\bpsi_m^n\right)_{0\le n\le N_m} \in W_{\Dd_m, \bdt_m} $.
Note that since $\psi(\cdot,\tf) = 0$, one has $\bpsi^{N_m} _m = \0$ for all $m\ge 1$. 

Setting $\bv  = \bpsi_m^{n-1} $ in~\eqref{eq:syst-w} and summing over $n \in \{1,\dots, N_m\} $ leads 
after a classical reorganization of the sums~\cite{EGH00} to 
\be\label{eq:ABCD} 
A_m + B_m + C_m + D_m = 0, 
\ee
where we have set 
\begin{align*} 
A_m  =& \sum_{n=1} ^{N_m} \Delta t_{n,m} \int_\O \pi_{\Dd_m} \u_m^n(\x) \pi_{\Dd_m} \left( \frac{\bpsi^{n-1} - \bpsi^{n} } {\dt_{n,m} } \right)(\x)\d\x,\\
B_m = & - \int_\O \pi_{\Dd_m} \u_m^0(\x) \pi_{\Dd_m} \bpsi^0(\x) \d\x, 	\\
C_m = & \sum_{n=1} ^{N_m} \dt_{n,m} \sum_{\k \in \Mm_m} \bd_\k p(\u_m^n) \cdot \B_\k(\u_m^n)   \bd_\k \bpsi_m^{n-1}, \\
D_m = & \sum_{n=1} ^{N_m} \dt_{n,m} \sum_{\k \in \Mm_m} \bd_\k \bV_m \cdot \B_\k(\u_m^n)   \bd_\k \bpsi_m^{n-1}, 
\end{align*} 
and $\bV_m = \left(V(\x_{\k}), V(\x_{\s})\right)_{\k \in \Mm_m, \s \in \Vv_m} $.

The regularity of $\psi$ yields 
$$
\sum_{n=1} ^{N_m} \pi_{\Dd_m} \left( \frac{\bpsi^{n-1} - \bpsi^{n} } {\dt_{n,m} } \right) \1_{[t_{n-1,m},t_{n,m})} \underset{m\to\infty} {\longrightarrow} - \p_t \psi 
\quad \text{ uniformly on  } Q_\tf 
$$
where $t_{n,m} = \sum_{i=1} ^n \dt_{i,m} $, so that, using Corollary~\ref{coro:convL1}, one gets 
\be\label{eq:Am} 
\lim_{m\to\infty} A_m = - \iint_{Q_\tf} u \p_t \psi \; \d\x \d t. 
\ee

The function $\pi_{\Dd_m} \u_m^0(\x)$ tends strongly in $L^1(\O)$ towards $u_0$ and $\pi_{\Dd_m} \bpsi^0$ 
converges uniformly towards $\psi(\cdot, 0)$ as $m$ tends to $+\infty$, leading to 
\be\label{eq:Bm} 
\lim_{m\to\infty} B_m = - \int_{\O} u_0(\x)  \psi(\x,0) \; \d\x. 
\ee

We split the term $C_m$ into three parts
\be\label{eq:Cm} 
C_m  = C_{1,m} + C_{2,m} + C_{3,m}, \qquad m\ge1, 
\ee
where, setting $\h\bpsi_m^0 = \bpsi_m^{0} $, $\h\bpsi_m^n = \bpsi_m^{n-1} \in W_{\Dd_m} $ for all $n \in \{1,\dots, N_m\} $, and 
 $\h\bpsi_m = \left(\h\bpsi_m^n\right)_{0 \le n \le N_m} \in W_{\Dd_m, \bdt_m} $, one has  
\begin{align*} 
C_{1,m} = & \iint_{Q_\tf} \pi_{\Mm_m,\bdt_m} \sqrt{\eta(\u_m)} \grad_{\Tt_m, \bdt_m} \xi(\u_m) 
\cdot \L  \grad_{\Tt_m, \bdt_m} \h\bpsi_m \d\x \d t, \\
C_{2,m} = & \sum_{n=1} ^{N_m} \dt_{n,m} \sum_{\k\in\Mm_m} \sum_{\s\in\Vv_\k} \sum_{\s'\in\Vv_\k} 
\sqrt{\eta_{\k,\s}^n}\left( p(u_\k^n) - p(u_\s^n) \right) 
a_{\s,\s'} ^\k  \\
& \hspace{4cm} \times\;  \left( \sqrt{\eta_{\k,\s'} ^n} - \sqrt{\eta(u_\k^n)} \right)\left( \psi_\k^{n-1} - \psi_{\s'} ^{n-1} \right),\\
C_{3,m} = & \sum_{n=1} ^{N_m} \dt_{n,m} \sum_{\k\in\Mm_m} \sqrt{\eta(u_\k^n)} \sum_{\s\in\Vv_\k} 
\left( \sqrt{\eta_{\k,\s} ^n} (p(u_\k^n)-p(u_\s^n)) - \left( \xi(u_\k^n) - \xi(u_\s^n) \right)\right) \\
&\hspace{6cm} \times\;\sum_{\s'\in\Vv_\k} a_{\s,\s'} ^\k  
\left( \psi_\k^{n-1} - \psi_{\s'} ^{n-1} \right).
\end{align*} 
Thanks to Lemma~\ref{lem:eta_M}, we know that  
$$
 \pi_{\Mm_m,\bdt_m} \sqrt{\eta(\u_m)} \underset{m\to\infty} {\longrightarrow} \sqrt{\eta(u)} \quad \text{ strongly in } L^2(Q_\tf).
$$
Hence, it follows from the weak convergence in $L^2(Q_\tf)$ of $\grad_{\Tt_m, \bdt_m} \xi(\u_m)$ towards $\grad \xi(u)$ (cf. Lemma~\ref{lem:conv-xi}) and from the uniform convergence of  $\grad_{\Tt_m, \bdt_m} \h\bpsi_m$ 
towards $\grad \psi$ as $m$ tends to $+\infty$ (see for instance~\cite[Theorem 16.1]{Ciarlet_Handbook}) 
that 
\be\label{eq:C'm} 
\lim_{m\to\infty} C_{1,m} = \iint_{Q_\tf} \sqrt{\eta(u)} \grad \xi(u) \cdot \L \grad \psi \; \d\x\d t.
\ee

Let us focus now of $C_{2,m} $. 
Using the inequality $ab \le \eps a^2 + \frac{1} {4\eps} b^2$, one gets that 
\be\label{eq:Qq_decomp} 
C_{2,m} \le   \eps C_{2,m} ' + \frac{C_{2,m} ''} {4 \eps}, \qquad \forall \eps >0, 
\ee
where 
\begin{align*} 
C_{2,m} '   = & \sum_{n=1} ^{N_m} \dt_{n,m} \sum_{\k\in\Mm_m} \sum_{\s\in\Vv_\k} \left(\sum_{\s'\in\Vv_\k} \left| a_{\s,\s'} ^\k \right| \right)
\eta_{\k,\s}^n\left( p(u_\k^n) - p(u_\s^n) \right)^2, \\
C_{2,m} '' = & \sum_{n=1} ^{N_m} \dt_{n,m} \!\!\! \sum_{\k\in\Mm_m} \sum_{\s\in\Vv_\k} \left(\sum_{\s'\in\Vv_\k} \left| a_{\s,\s'} ^\k \right| \right)
\!\! \left( \sqrt{\eta_{\k,\s} ^n} - \sqrt{\eta(u_\k^n)} \right)^2\!\! \left( \psi_\k^{n-1} - \psi_{\s} ^{n-1} \right)^2.
\end{align*} 
We deduce from Estimate~\eqref{eq:dissip-abs_m} that
\be\label{eq:Qq'} 
C_{2,m} '  \le C, \qquad \forall m \ge 1,
\ee
for some $C$ depending only on $u_0$, $p$, $\O$, $\L$, $\theta^\star$ and $\ell^\star$. 

Define $\bmu_m = \left(\mu_\k^n, \mu_\s^n\right)_{\k,\s,n} \in W_{\Dd_m, \bdt_m} $ by 
\be\label{eq:mu_def} 
\begin{cases} 
\mu_\s^n = 0, & \forall \s \in \Vv_m,\\
 \mu_\k^n = \max_{\s \in \Vv_\k} \left| \sqrt{\eta(u_\s^n)} - \sqrt{\eta(u_\k^n)}\right|
 & \forall \k \in \Mm_m, 
 \end{cases} 
\qquad  \forall n \in \{0, \dots, N_m\}.
\ee
The definition~\eqref{eq:etaks} of $\eta_{\k,\s}^n$ implies that 
$$
\left| \sqrt{\eta_{\k,\s} ^n} - \sqrt{\eta(u_\k^n)} \right| \le \mu_\k^n, \qquad \forall \k \in \Mm_m, \; \forall \s \in \Vv_\k.
$$
Therefore, we get that 
$$
C_{2,m} '' \le  \sum_{n=1} ^{N_m} \dt_{n,m}  \sum_{\k\in\Mm_m} \left(\mu_\k^n\right)^2
\sum_{\s\in\Vv_\k} \left(\sum_{\s'\in\Vv_\k} \left| a_{\s,\s'} ^\k \right| \right) \left( \psi_\k^{n-1} - \psi_{\s} ^{n-1} \right)^2.
$$
Then thanks to Lemma~\ref{lem:abs-Ak}, there exists $C$ depending only on $\L$, $\theta^\star$ and $\ell^\star$ such that 
$$
C_{2,m} '' \le C \iint_{Q_\tf} \left(\pi_{\Mm_m, \bdt_m} \bmu_m\right)^2 \; \grad_{\Tt_m, \bdt_m} \h\bpsi_m 
\cdot \L \grad_{\Tt_m, \bdt_m} \h\bpsi_m \d\x \d t, 
\qquad \forall m \ge 1. 
$$
Since $\pi_{\Mm_m, \bdt_m} \bmu_m$ converges to $0$ strongly in $L^2(Q_\tf)$ as $m$ tends to $\infty$ 
(this is the purpose of Lemma~\ref{lem:mu} hereafter), 
and since $\grad_{\Tt_m, \bdt_m} \h\bpsi_m$ remains bounded in $L^\infty(Q_\tf)$ uniformly w.r.t. $m \ge 1$, one gets that 
\be\label{eq:Qq''} 
\lim_{m\to \infty} C_{2,m} ''  = 0.
\ee
Therefore, it follows from~\eqref{eq:Qq_decomp}--\eqref{eq:Qq''} that 
$\limsup_{m\to\infty} C_{2,m} \le C \eps$ for arbitrary small values of $\eps>0$, 
whence
\be\label{eq:Qq} 
\lim_{m\to\infty} C_{2,m} = 0. 
\ee

As a preliminary before considering $C_{3,m} $, let us set, for all $\k \in \Mm_m$, all $\s \in \Vv_\k$, and all $n \in \{1, \dots, N_m\} $, 
$$
\widetilde \eta_{\k,\s} ^n = \begin{cases} 
\displaystyle \left( \frac{\xi(u_\k^n) - \xi(u_\s^n)} {p(u_\k^n) - p(u_\s^n)} \right)^2 & \text{ if } u_\k^n \neq u_\s^n, \\[10pt]
\eta(u_\k^n) & \text{ if } u_\k^n = u_\s^n.
\end{cases} 
$$
Thanks to the mean value theorem, we can claim that, for all $\k \in \Mm_m$, all $\s \in \Vv_\k$ and all $n\in \{1,\dots,N_m\} $, 
there exists $\widetilde u_{\k,\s} ^n \in \Ii_{\k,\s} ^n = [\min(u_\k^n, u_\s^n), \max(u_\k^n, u_\s^n)]$ such that 
$\widetilde \eta_{\k,\s} ^n = \eta(\widetilde u_{\k,\s} ^n).$ 
In particular, this ensures that 
$$
\left|\sqrt{\eta_{\k,\s}^n} - \sqrt{ \widetilde \eta_{\k,\s}^n}\right| \le \mu_\k^n, \qquad \forall \k \in \Mm_m, \, \forall \s \in \Vv_\k
$$
where $\mu_\k$ was defined by~\eqref{eq:mu_def}.
Using moreover that $\eta(u)_\k^n \le 2 \eta_{\k,\s}^n$, one gets that 
$$
C_{3,m} \le 2 \sum_{n=1} ^{N_m} \dt_{n,m} \sum_{\k\in\Mm_m} \mu_\k^n \sum_{\s\in\Vv_\k} 
 \sqrt{\eta_{\k,\s} ^n} (p(u_\k^n)-p(u_\s^n))  
\sum_{\s'\in\Vv_\k} a_{\s,\s'} ^\k  
\left( \psi_\k^{n-1} - \psi_{\s'} ^{n-1} \right).
$$
Cauchy-Schwarz inequality and Estimate~\eqref{eq:dissip-abs_m} yield
$$
C_{3,m} \le  C \left( \sum_{\k \in \Mm_m} \left(\mu_\k^n\right)^2 \sum_{\s \in \Vv_\k} 
\left(\sum_{\s' \in \Vv_\k} |a^\k_{\s,\s'}| \right) \left(\psi_\k^{n-1} - \psi_\s^{n-1}\right)^2\right)^{1/2}.
$$
Lemma~\ref{lem:abs-Ak}, and the 
regularity of $\psi$ provide that
$$
C_{3,m} \le C \left\| \pi_{\Mm_m,\bdt_m} \bmu_m\right\|_{L^2(Q_\tf)}.
$$
Applying Lemma~\ref{lem:mu}, we get 
\be\label{eq:Rr} 
\lim_{m\to \infty} C_{3,m} = 0.
\ee

Putting~\eqref{eq:Cm} together with~\eqref{eq:C'm}, \eqref{eq:Qq}, and~\eqref{eq:Rr}, one gets that 
\be\label{eq:Cm2} 
\lim_{m\to \infty} C_m = \iint_{Q_\tf} \sqrt{\eta(u)} \grad \xi(u) \cdot \L \grad \psi \; \d\x\d t.
\ee 
\bigskip

Now, we focus on the term $D_m$ that can be decomposed into 
\be\label{eq:D123} 
D_m = D_{1,m} + D_{2,m} + D_{3,m}, \qquad \forall m \ge 1, 
\ee
where we have set 
\begin{align*} 
D_{1,m} =&\iint_{Q_\tf} \pi_{\Mm_m,\bdt_{m} } \eta(\u_m)\grad_{\Tt_m} \bV_m\cdot\L\grad_{\Tt_m,\bdt_m} 
\h\bpsi_m\d\x\d t,\\
D_{2,m} =&\frac12\sum_{n = 1} ^{N_m} \dt_{n,m} \sum_{\k \in \Mm_m} \sum_{\s \in \Vv_\k} \sum_{\s' \in \Vv_\k} 
a_{\s,\s'} ^\k\left(\sqrt{\eta_{\k,\s} ^n} -\sqrt{\eta(u_\k^n)} \right)(V_\k - V_\s)\\
&\hspace{5cm} \times\left(\sqrt{\eta_{\k,\s'} ^n} +\sqrt{\eta(u_\k^n)} \right)\left( \psi_\k^{n-1} -\psi_{\s'} ^{n-1} \right), \\
D_{3,m} =&\frac12\sum_{n = 1} ^{N_m} \dt_{n,m} \sum_{\k \in \Mm_m} \sum_{\s \in \Vv_\k} \sum_{\s' \in \Vv_\k} 
a_{\s,\s'} ^\k\left(\sqrt{\eta_{\k,\s} ^n} + \sqrt{\eta(u_\k^n)} \right)(V_\k - V_\s)\\
&	\hspace{5cm} \times \left(\sqrt{\eta_{\k,\s'} ^n} - \sqrt{\eta(u_\k^n)} \right) \left( \psi_\k^{n-1} - \psi_{\s'} ^{n-1} \right).
\end{align*} 

It follows from Lemma~\ref{lem:eta_M}, from the uniform convergence of $\grad_{\Tt_m} \bV_m$ towards 
$\grad V$ and of $\grad_{\Tt_m, \bdt_m} \h \bpsi_m$ towards $\grad \psi$ as $m$ tends to $+\infty$ that 
\be\label{eq:D1} 
\lim_{m\to \infty} D_{1,m} = \int_{Q_\tf} \eta(u) \grad V \cdot \L \grad \psi \, \d\x \d t. 
\ee

Let $\eps >0$, using again the inequality $|ab| \le \eps a^2 + \frac{b^2} {4\eps} $, we obtain that 
\be\label{eq:D2m0} 
|D_{2,m} | \le \eps D'_{2,m} + \frac1{16 \eps} D_{2,m} '', \qquad \forall m \ge 1,
\ee
where we have set 
\begin{align*} 
D_{2,m} '=&\sum_{n = 1} ^{N_m} \dt_{n,m} \sum_{\k\in\Mm_m} \sum_{\s\in\Vv_\k} 
\left(\sum_{\s' \in \Vv_\k} |a_{\s,\s'} ^\k|\right)\left(\sqrt{\eta_{\k,\s} ^n} + \sqrt{\eta(u_\k^n)} \right)^2
\left(\psi_\k^{n-1} -\psi_\s^{n-1} \right)^2,\\
D_{2,m} ''=&\sum_{n = 1} ^{N_m} \dt_{n,m} \sum_{\k \in \Mm_m} \sum_{\s \in \Vv_\k} 
\left(\sum_{\s' \in \Vv_\k} |a_{\s,\s'} ^\k|\right)\left(\sqrt{\eta_{\k,\s} ^n} -\sqrt{\eta(u_\k^n)} \right)^2
\left(V_\k - V_\s\right)^2.
\end{align*} 
Define  $\ov \bfeta_m = (\ov \eta_\k^n, \ov \eta_\s^n)_{\k,\s} \in W_{\Dd_m, \bdt_m} $ by 
$$
\ov \eta_\k^n = \max \Big( \eta(u_\k^n), \max_{\s \in \Vv_\k} u_\s^n \Big), \qquad \ov \eta_\s^n = 0, \qquad 
\forall \k \in \Mm_m, \; \forall \s \in \Vv_m, 
$$
then one has 
$$
\left(\sqrt{\eta_{\k,\s} ^n} + \sqrt{\eta(u_\k^n)} \right)^2 \le 4 \ov \eta_\k^n, 
\qquad \forall \k \in \Mm_m, \; \forall \s \in \Vv_\k, 
$$
whence 
$$
D_{2,m} '  \le C \iint_{Q_\tf} \pi_{\Mm_m, \bdt_m} \ov\bfeta_m \; 
\L\grad_{\Tt_m, \bdt_m} \bpsi_m\cdot \grad_{\Tt_m, \bdt_m} \bpsi_m  \d\x \d t, \qquad \forall m \ge 1
$$
thanks to Lemma~\ref{lem:abs-Ak}.
Using Lemma~\ref{lem:Aovbv}, we know that there exists $C$ depending on the data of the 
continuous problem and of the regularity factors $\theta^\star$, $\ell^\star$ and $\zeta^\star$ such that 
$$
\left\|\pi_{\Mm_m, \bdt_m} \ov \bfeta_m\right\|_{L^1(Q_\tf)} \le C, \qquad \forall m \ge 1, 
$$
while the regularity of $\psi$ ensures that 
$$
\left\| \grad_{\Tt_m, \bdt_m} \bpsi_m\right\|_{L^\infty(Q_\tf)} \le C, \qquad \forall m \ge 1. 
$$
Therefore, there exists $C$ depending only on the data of the continuous problem and 
the regularity factors $\theta^\star$, $\ell^\star$ and $\zeta^\star$ such that 
\be\label{eq:D2m'} 
D_{2,m} '  \le C, \qquad \forall m \ge 1.
\ee
The term $D_{2,m} ''$ can be studied as $C_{2,m} ''$ was, leading to 
\be\label{eq:D2m''} 
\lim_{m\to \infty} D_{2,m} ''  =0, 
\ee
whence, taking~\eqref{eq:D2m'}--\eqref{eq:D2m''} into account in~\eqref{eq:D2m0}, one gets that 
\be\label{eq:D2m} 
\lim_{m \to \infty} D_{2,m} = 0.
\ee
Reproducing the calculations carried out for dealing with $D_{2,m} $ allows to show that 
\be\label{eq:D3m} 
\lim_{m \to \infty} D_{3,m} = 0.
\ee
Combining~\eqref{eq:D123}--\eqref{eq:D1} and \eqref{eq:D2m}--\eqref{eq:D3m}, we obtain that 
\be\label{eq:Dm} 
\lim_{m\to \infty} D_{m} = \iint_{Q_\tf} \eta(u) \L\grad V  \cdot \grad \psi \d\x\d t. 
\ee

Finally, it follows from~\eqref{eq:ABCD}, \eqref{eq:Am}--\eqref{eq:Bm}, \eqref{eq:Cm2} 
and~\eqref{eq:Dm} that the limit $u$ 
of the discrete reconstructions $\left( \pi_{\Dd_m, \bdt_m} \u_m \right)_{m\ge1} $ is a weak solution 
to the problem~\eqref{eq:1} in the sense of Definition~\ref{def:weak}.
\end{proof} 

\begin{lem} \label{lem:mu} 
Let $\bmu_m =  \left(\mu_\k^n, \mu_\s^n\right)_{\k,\s,n} \in W_{\Dd_m, \bdt_m} $ be defined by~\eqref{eq:mu_def}, 
then 
$$
\pi_{\Mm_m, \bdt_m} \bmu_m \underset{m\to\infty} {\longrightarrow} 0 \quad \text{ strongly in } L^2(Q_\tf), 
$$
\end{lem} 
\begin{proof} 
Using $(a-b)^2 \le 2 \left(a^2 + b^2\right)$, one gets that 
$$
\left( \mu_\k^n \right)^2 \le 
 2 \left( \eta(u_\k^n) + \max_{\s \in \Vv_\k} \eta(u_\s^n) \right)
 \le 2 \left( \eta(u_\k^n) + \sum_{\s \in \Vv_\k} \eta(u_\s^n) \right)
$$
for all $\k \in \Mm_m$ and all $n \in \{1, \dots, N_m\}$.
Using~\eqref{eq:gamma-theta-star}--\eqref{eq:zeta*}, which ensure that 
$$
m_\k \ge \frac{\zeta^\star} d \meas(\k) 
\quad \text{ and } \quad 
m_\s \ge \frac{\zeta^\star} {d \ell^\star} \meas(\k), 
$$
we deduce that there exists $C$ depending on $d$, $\ell^\star$ and $\zeta^\star$ such that 
$$
\left(\pi_{\Mm_m, \bdt_m} \bmu_m\right)^2 \le C \pi_{\Dd_m, \bdt_m} \eta(\u_m), \qquad \forall m \ge 1. 
$$
As a particular consequence of Lemma~\ref{lem:eta_M}, we know that 
$\left( \pi_{\Dd_m, \bdt_m} \eta(\u_m) \right)_{m\ge1} $ is uniformly equi-integrable, 
 whence 
\be\label{eq:mum-equi} 
\text{$\left(\pi_{\Mm_m, \bdt_m} \bmu_m  \right)_{m\ge1} $ is uniformly $L^2$-equi-integrable.} 
\ee

Let us introduce $\bw_m = (w_\k^n, w_\s^n)_{\k,\s,n} \in W_{\Dd_m, \bdt_m} $ defined 
for all $\k \in \Mm_m$, all $\s \in \Vv_m$ and all $n \in \{1,\dots, N_m\} $ by 
\be\label{eq:wkn} 
w_\s^n = 0, \quad w_\k^n = \max_{\s\in\Vv_\k} \left| \xi(u_\k^n) - \xi(u_\s^n)\right|.
\ee
It follows from a straightforward generalization of Lemma~\ref{lem:A7} and from estimate~\eqref{eq:dissip-xi_m} 
that $\pi_{\Mm_m, \bdt_m} \bw_m$ converges strongly in $L^2(Q_\tf)$ towards $0$. Therefore, up to an unlabeled 
subsequence, it converges almost everywhere. As a consequence, 
\be\label{eq:nonlin-phi} 
\pi_{\Mm_m, \bdt_m} \phi (\bw_m) \underset{m\to\infty} \longrightarrow 0 \quad \text{ a.e. in } Q_\tf
\ee
for all continuous function $\phi:\R_+ \to \R$ such that $\phi(0) = 0$. 

It follows from the definition of $\bmu_m$ that 
$$
\mu_\k^n \le  \varpi(w_\k^n), \quad \forall \k \in \Mm_m, \; \forall n \in \{1,\dots, N_m\}, 
$$
where $\varpi$ is a modulus of continuity of $\sqrt{\eta\circ\xi^{-1}}$ (cf.~\eqref{eq:varpi}). Hence, 
we obtain that 
$$
0 \le \pi_{\Mm_m, \bdt_m} \bmu_m \le \pi_{\Mm_m, \bdt_m} \varpi(\bw_m). 
$$
Thanks to~\eqref{eq:nonlin-phi}, we obtain that 
\be\label{eq:mum-ae} 
\pi_{\Mm_m, \bdt_m} \bmu_m \underset{m\to\infty} \longrightarrow 0 \quad \text{ a.e. in } Q_\tf.
\ee

In order to conclude, it only remains to remark that~\eqref{eq:mum-equi} and~\eqref{eq:mum-ae} 
allow us to use Vitali's convergence theorem. 
\end{proof}

\section{Numerical implementation and results} \label{sec:num} 
This section is devoted to the numerical resolution of the nonlinear system \eqref{eq:syst}.
First, we discuss in~\S\ref{ssec:schur} the strategy that we used for solving the nonlinear system~\eqref{eq:syst}.
Then we present in~\S\ref{ssec:res} two $2$-dimensional cases with analytical solutions in order to illustrate the 
numerical convergence of the method. 

\subsection{Newton method, Schur complement and time-step adaptation}\label{ssec:schur}

The nonlinear system  \eqref{eq:syst} obtained at each time step is solved by a Newton-Raphson algorithm. 
Given $\u^{n-1} \in W_\Dd$, this leads to the computation of a sequence $\left(\u^{n,i}\right)_{i\ge 0} \subset W_\Dd$ 
such that $\u^n = \lim_{i \to \infty}\u^{n,i}$ is a solution to~\eqref{eq:syst}. 
The variation of the discrete unknowns between two Newton-Raphson algorithm iterations is denoted as follows,
$$
\d\u^{n,i} ~=~ \left( ~ \d u_\s^{n,i},~ \d u_\k^{n,i}  ~\right)_{\s\in \Vv,\k\in \Mm} = \u^{n,i+1} - \u^{n,i}, \qquad \forall i \ge 0.
$$
Let us briefly detail the practical implementation of the iterative procedure allowing to deduce $\u^n$ from $\u^{n-1}$.
\begin{enumerate}
 \item In the case where $p(0)$ is finite, the initial guess for the Newton algorithm is, as usual, taken as 
 $\u^{n,0}=\left(  u_\s^{n-1} , u_\k^{n-1} \right)_{\k,\s}$ for all $\s\in \Vv, \k\in \Mm$. In the singular case 
 $p(0) = -\infty$, it was proved in Lemma~\ref{lem:harnack} that the solution $\u^{n} = (u_\k^n,u_\s^n)_{\k,\s}$ 
 of~\eqref{eq:syst} is such that 
 $
\min_{\beta \in \Mm\cup\Vv} u_\beta^n >0.
 $
Therefore, we can initialize the Newton algorithm by 
$$
\u^{n,0}=\left(  \max\left(\eps, u_\s^{n-1}\right) , \max\left(\eps,u_\k^{n-1}\right) \right)_{\k,\s}. 
$$
In the computations, we fixed $\eps = 10^{-10}$.
 \item The Newton-Raphson algorithm iterations are done until a convergence criterion on 
 the $L^{\infty}(\O)$ norm of the variation of the discrete unknowns is reached or 
 until the maximum number of iterations is reached.
 At each iteration, the Jacobian matrix resulting of \eqref{eq:syst} is computed and has the
 following block structure
$$
  \begin{pmatrix}
    {\bf A} & {\bf B} \\
    {\bf C} & {\bf D}
  \end{pmatrix}
    {\d \u^{n,i}}
=
  \begin{pmatrix}
    {\bf b}_{1} \\
    {\bf b}_{2}
  \end{pmatrix},
$$
where the sub-matrices have the following sizes: ${\bf A} \in \R^{\#\Vv} \otimes \R^{\#\Vv}$,
${\bf B} \in \R^{\#\Vv} \otimes \R^{\#\Mm}$, ${\bf C} \in\R^{\#\Mm} \otimes \R^{\#\Vv}$, and 
${\bf D} \in \R^{\#\Mm} \otimes \R^{\#\Mm}$. 
The sub-vectors at the right hand side have thus the following sizes: ${\bf b}_{1}\in\R^{\#\Vv}$ and ${\bf b}_{2}\in\R^{\#\Mm}$.
The dependence of the sub-matrices and the sub-vectors w.r.t. $n$ and $i$ was not highlighted here 
for the ease of notations. 
A main characteristic of this block structure is that the block ${\bf D}$ is a non singular diagonal matrix, 
thus the Schur complement can be easily computed without fill-in to
eliminate the variation of the cell unknowns. This allow to reduce the linear system 
to the variation of the vertices unknowns as is usual when using the VAG scheme. 
The resulting linear system that we have to solve in order to obtain the variation of the vertices unknowns is given by,
\be
({\bf A}-{\bf B D}^{-1}{\bf C} ) (\d u_\s^{n,i})_{\s\in\Vv} = {\bf b}_1 -{\bf B D}^{-1} {\bf b}_2,
\label{eq:schur_complement}
\ee
and then the variation of the cell unknowns can be easily deduced by the matrix-vector product below,
$$
(\d u_\k^{n,i})_{\k\in\Mm} = {\bf D}^{-1}({\bf b}_2-{\bf C} (\d u_\s^{n,i})_{\s\in\Vv} ).
$$
As for the initial step, we have to take into account the singular case at each Newton-Raphson iteration by,
$$
 \u^{n,i+1} \ = \ \max( \u^{n,i} + \d\u^{n,i} \ , \ \eps ).
$$
\item If the Newton-Raphson algorithm  stops before the  maximum number of iterations is reached, 
the next time iteration is proceeded by increasing the time step. Otherwise, the current time iteration is recomputed
by reducing the time step. 
The time step is bounded by a maximum value denoted ${\dt_{\max}}$.
A maximum number of convergence failures of the nonlinear methods is imposed
in order to abort the simulation in case of a non-convergence.  
\end{enumerate}

\subsection{Definitions of the test-cases and numerical results}\label{ssec:res}

We present here four $2$-dimensional numerical cases where $\Omega$ is the unit square.
The space domain is discretized by using meshes obtained from a benchmark on anisotropic 
diffusion problem~\cite{bench_FVCA5}.
In the following numerical experiments, the tensor is defined by  
$$\L = \left( \begin{array}{cc} 
l_x & 0 \\ 0 & l_y \end{array} \right)
$$
where $l_x$ and $l_y$ are chosen constant in $\O$, and
the exterior potential is defined by $V(\x)=-{\mathbf g}\cdot\x$ for all $\x\in\O$ where
${\mathbf g}=(g,0)^t$ with $g\in\R_+$. 
The weights of the VAG scheme defined in \eqref{eq:aks} are defined 
by $\a_{\k,\s} = \frac{0.1}{\sharp \Vv_\k}$ for all $\k \in \Mm$, ${\s\in\Vv_\k}$.
We refer to~\cite{EGHM12,BGGM15} for a discussion on the mass distribution 
for heterogeneous problems.
The linear solver applied to solve \eqref{eq:schur_complement}  is a home-made direct 
solver using a gaussian elimination with an optimal reordering.

In some of the test cases presented hereafter, Dirichlet boundary conditions are considered instead of no-flux boundary conditions. 
This allows to construct analytical solutions to the continuous problem and to perform a 
convergence study. Even though it has not been done in this paper, the convergence proof for 
the scheme can still be carried out when (sufficiently regular) Dirichlet boundary conditions are considered. 
However, the gradient flow structure is destroyed and the free energy might not be decreasing 
anymore in this case.

Errors are computed in the classical discrete $L^2(Q_\tf)$, $L^1(Q_\tf)$ and $L^{\infty}(Q_\tf)$ norms. 
All the results are presented in the Tables below. Each table provides the mesh size $h$, 
the initial and maximum time steps, the discrete errors, their associated convergence rate
and the minimum value of the discrete solution. It also contains the total (integrated over time) 
number of Newton-Raphson iterations  needed to compute the solution as a indicator of the 
cost of the numerical method. 

\subsubsection{Test 1: Linear Fokker-Planck equation with no-flux boundary condition}

This first test case matches with the problem defined by \eqref{eq:1} with 
the functions $\eta(u)=u$ on $\R_+$ and $p(u)=\log(u)$, and with the gravitational 
potential $V(x,y) = - g x$ where the constant $g$ is fixed to $1$.
Setting ${\mathbf g}  = (g,0)^T$, the problem \eqref{eq:1} leads to the linear equation
\be
\p_t u - \div\left( \L \ ( \grad u - u {\mathbf g}  ) \ \right) = 0 \;\; \text{ in } Q_\tf.
\label{eq:ifct2}
\ee
We compare the results obtained with the nonlinear scheme~\eqref{eq:syst} with those obtained using
the definition of the fluxes
\begin{multline}\label{eq:Fks2}
\widetilde{F}_{\k,\s}(\u^n) = \sum_{\s'\in\Vv_\k} a^\k_{\s,\s'} (u_\k^n - u_{\s'}^n) +
\frac{u_\k^n+u_\s^n}2  \sum_{\s'\in\Vv_\k} a^\k_{\s,\s'}(V_\k -  V_{\s'}), \\
\quad \forall \k \in \Mm, \; \forall \s \in \Vv_\k
\end{multline}
instead of \eqref{eq:Fks}.
The resulting scheme is called the {\em linear scheme}.
The numerical convergence of both schemes has been compared on the following analytical solution 
(built from a $1$-dimensional case):
\begin{multline}\label{eq:sol2}
\widetilde{u}(x,y,t) =  
\exp\left(- \alpha t + \frac g 2 x\right) \left(\pi \cos( \pi x) + \frac g 2 \sin(\pi x)\right) + 
 \pi \exp\left(g(x-\frac 1 2)\right), \\ \quad \forall ((x,y),t)\in\Omega\times(0,\tf),
\end{multline}
where $\alpha =  l_x \ (\pi^2 + \frac {g^2} 4)$.
This function satisfies the homogeneous Neumann boundary condition and 
the property $\widetilde{u}(x,y,t)>0$ for all $(x,y,t)\in Q_\tf$.

In order to make a numerical convergence study, we have used a family of triangular meshes. 
These triangle meshes show no symmetry which could artificially increase the convergence rate. 
This family of meshes is built  through the same pattern, which is reproduced at different scales: 
the first (coarsest) mesh and the third mesh are shown by Figure~\ref{fig:meshGrossier}.
Although the analytical solution is one-dimensional and the permeability tensor is diagonal, 
the discrete problem is really 2D because of the non-structured grids. 
The 2D aspect of the problem is amplified by the choice of a stronger diffusion 
in the transversal direction.
\begin{figure}[h!]
\includegraphics[width=4cm]{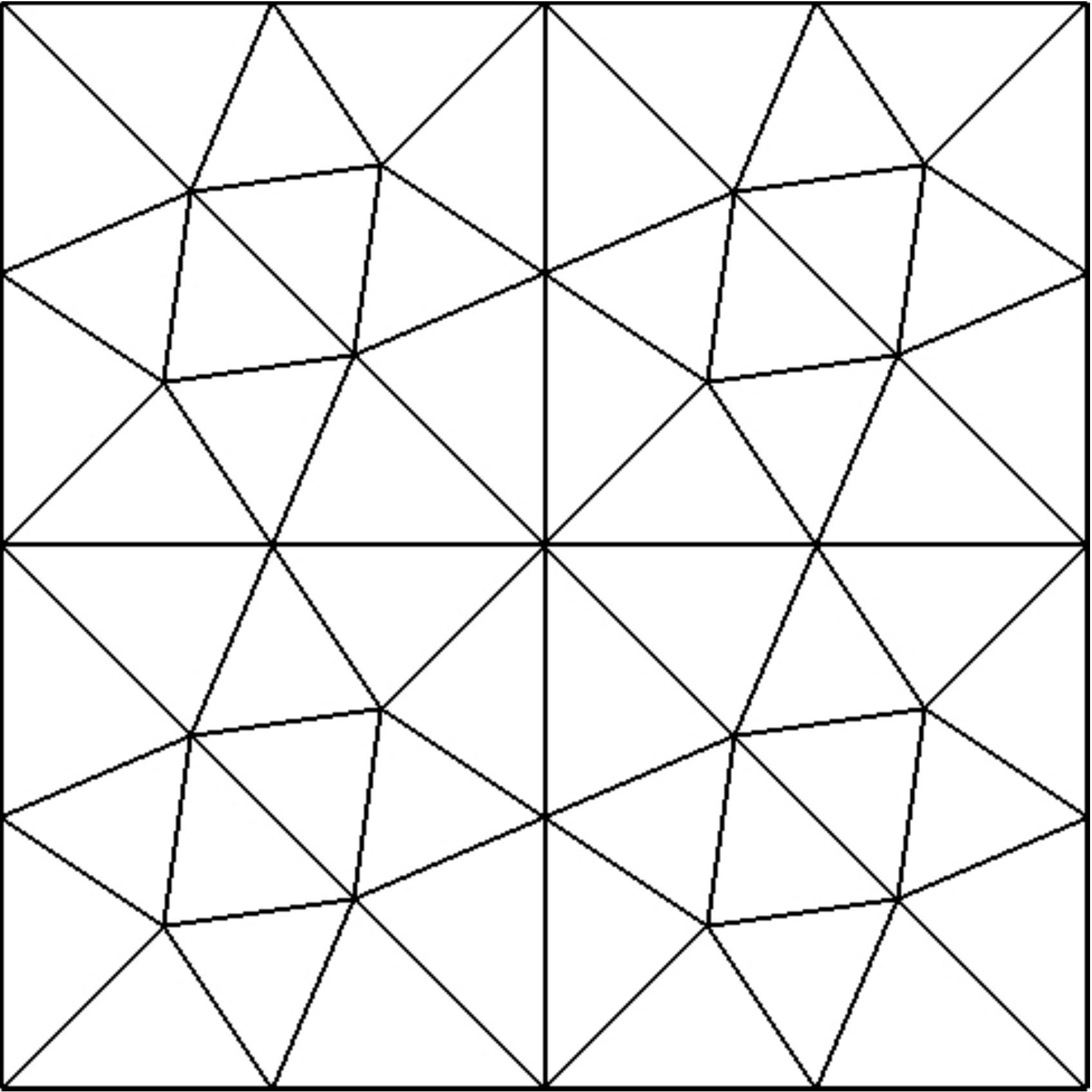} \hspace*{1cm}
\includegraphics[width=4cm]{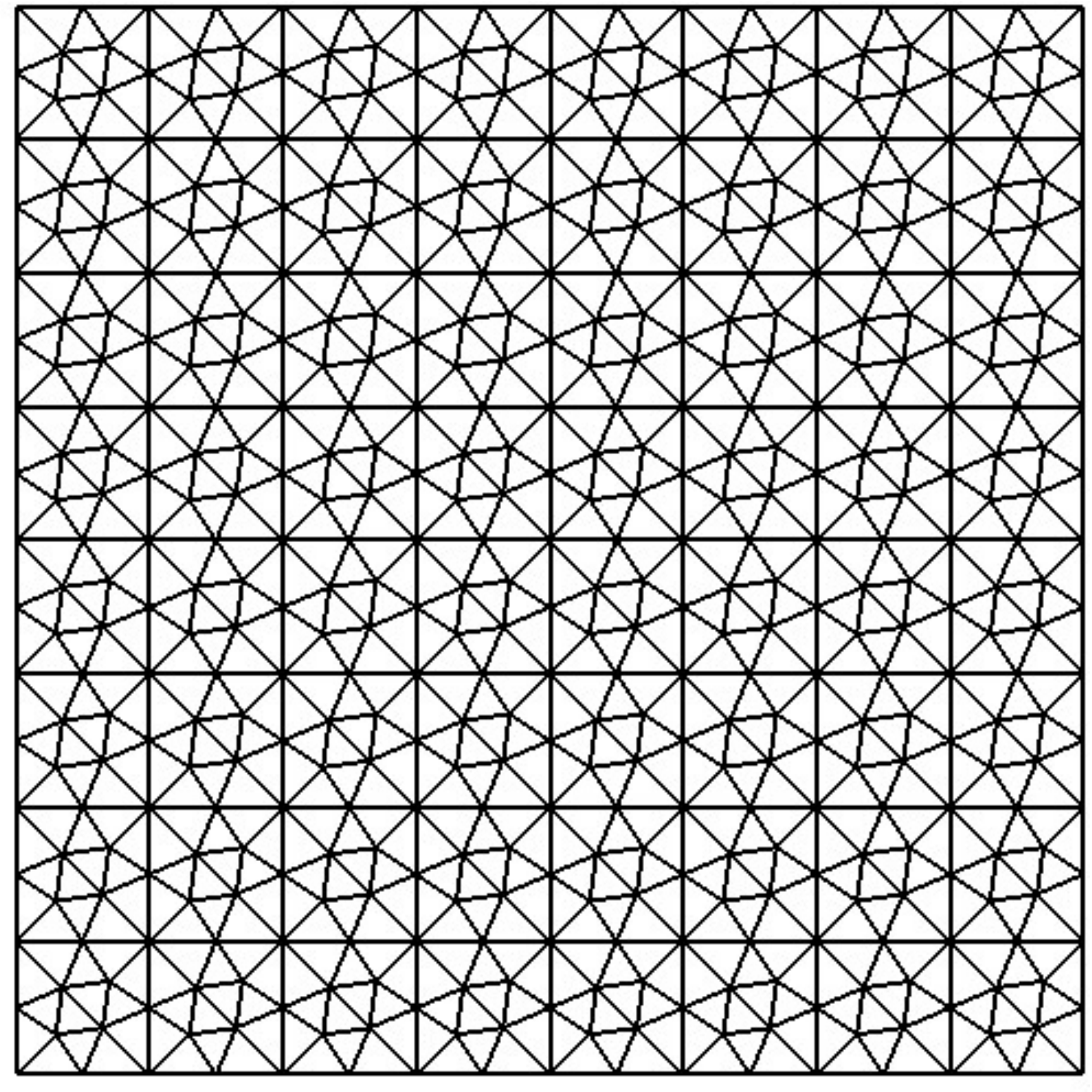}
\caption{First and third mesh used in the numerical examples.}
\label{fig:meshGrossier}
\end{figure}
For the tests on triangular grids, the final time $\tf$ has been chosen to $0.25$ and  
an anisotropic tensor has been consider: $l_x=1$ and $l_y=10$.

\begin{table}[h!]
\resizebox{\textwidth}{!}{
 \begin{tabular}{|cc|cc|cc|cc|cc|c|c|}
\hline
$h$  & $\#\Vv$ & $\dt_{\textrm{init}}$ &  $\dt_{\max}$  & $\textrm{err}_{L^2}$ & rate &
$\textrm{err}_{L^1}$ & rate &  $\textrm{err}_{L^\infty}$ & rate &
$u_{\textrm{min}}$ & \#Newton  \\
\hline
 0.250 & 37   & 0.001   & 0.01024  &  0.196E-01 &    -   &  0.754E-02 &    -   &  0.216E+00 &    -   &  0.022 & 204\\
 0.125 & 129  & 0.00025 & 0.00256  &  0.512E-02 &  1.935 &  0.178E-02 &  2.084 &  0.600E-01 &  1.848 &  0.004 & 456\\
 0.063 & 481  & 0.00006 & 0.00064  &  0.129E-02 &  1.986 &  0.430E-03 &  2.050 &  0.157E-01 &  1.931 &  0.001 &1307\\
 0.031 & 1857 & 0.00002 & 0.00016  &  0.324E-03 &  1.997 &  0.107E-03 &  2.007 &  0.473E-02 &  1.734 &  0.000 &3935\\
\hline   
\end{tabular}
}
\caption{Triangles. Nonlinear scheme~\eqref{eq:syst}.}
\label{tab:tri-lx1ly10-VAG}
\end{table}

\begin{table}[h!]
\resizebox{\textwidth}{!}{
 \begin{tabular}{|cc|cc|cc|cc|cc|c|c|}
\hline
$h$  & $\#\Vv$ & $\dt_{\textrm{init}}$ &  $\dt_{\max}$  & $\textrm{err}_{L^2}$ & rate &
$\textrm{err}_{L^1}$ & rate &  $\textrm{err}_{L^\infty}$ & rate &
$u_{\textrm{min}}$ & \#Newton  \\
\hline
 0.250 & 37   & 0.001   & 0.01024 &  0.187E-01 &    -   &  0.708E-02 &    -   &  0.225E+00 &    -   & -0.155 &33\\
 0.125 & 129  & 0.00025 & 0.00256 &  0.469E-02 &  1.993 &  0.165E-02 &  2.100 &  0.786E-01 &  1.515 & -0.046 &106\\
 0.063 & 481  & 0.00006 & 0.00064 &  0.117E-02 &  1.999 &  0.406E-03 &  2.023 &  0.228E-01 &  1.784 & -0.012 &400\\
 0.031 & 1857 & 0.00002 & 0.00016 &  0.293E-03 &  1.999 &  0.102E-03 &  1.999 &  0.611E-02 &  1.901 & -0.003 &1570\\
\hline   
\end{tabular}
}
\caption{Triangles. Linear scheme, fluxes defined by~\eqref{eq:Fks2}.}
\label{tab:tri-lx1ly10-linear}
\end{table}

Let us first observe that
the numerical order of convergence is close to 2 for both schemes.
The nonlinear scheme is of course more expensive than the linear one but 
it preserves the positivity of the solution, unlike the linear scheme. This numerical behavior
is a verification of the theoretical result mentioned in the Lemma \ref{lem:harnack}. 
In the linear case, the number of Newton-Raphson iterations is equal to the number of time steps. 
On the finest mesh, the ratio of the number of Newton iterations between the nonlinear and the linear schemes 
is about $2.5$. It seems to be acceptable in cases where preserving the positivity is mandatory.

Now, in order to exhibit the ability of the VAG scheme to deal with general meshes, the same test case has been applied on a so-called Kershaw grid (cf. Figure \ref{fig:MeshKer}). Instead of an irrelevant numerical convergence study 
--- it is difficult to define a refinement factor for this type of grids ---, we aim to give an evidence that the scheme is 
free energy diminishing (thus positivity preserving) and that the long-time behavior of the continuous problem is preserved at the discrete level by the scheme.
The final time $\tf$ has been chosen to $250$ and  
an anisotropic tensor has been consider: $l_x=0.001$ and $l_y=1$.
The results are listed on the Table \ref{tab:ker-linear} and
we can check again that the nonlinear scheme is positivity preserving despite the irregular grid.

\begin{table}[h!]
\resizebox{\textwidth}{!}{
 \begin{tabular}{c|c|c|c|c|c|c|c|c|}
 & $\#\Vv$ & $\dt_{\textrm{init}}$ &  $\dt_{\max}$  & $\textrm{err}_{L^2}$ & 
$\textrm{err}_{L^1}$ & $\textrm{err}_{L^\infty}$ & $u_{\textrm{min}}$ & \#Newton  \\ 
\hline \hline
nonlinear scheme & 324   & 2.E-04   & 1 &  3.99E-02 &  0.404 &  1.42E-02 & 8.92E-04 &1148\\ 
\hline \hline
linear scheme & 324  & 2.E-04 & 1 &  3.47E-02 &   0.377 & 2.01E-02 & -1.49E-02 &259\\
\hline   
\end{tabular}
}
\caption{Kershaw grid. Nonlinear and linear scheme, with an anisotropic tensor.}
\label{tab:ker-linear}
\end{table}

Denoting by $w= \pi \exp\left(g(x-\frac 1 2)\right)$ the long-time asymptotic of $\wt u$ defined 
by~\eqref{eq:sol2}, then the relative entropy of a function $u:\O\to\R_+$ w.r.t. $w$ is defined by 
\be\label{eq:entrel}
E^w(u)= \int_{\O} \left( u \log\left( \frac u w \right) - u  + w \right) \d\x. 
\ee
It is simple to verify that 
$$
\frac{\d}{\d t}E^w(u) \frac{\d}{\d t} E(u)=0.
$$
Therefore the decay of the free energy is equivalent to the decay of the relative entropy.
Note that $E^w$ is undefined (or is set to $+\infty$) if $u<0$ on a positive measure set. 
It is well known (see e.g.~\cite{CJMTU01,Lisini09}) that 
the relative entropy $E^w(\wt u(\cdot,t))$ converges 
exponentially fast towards $0$ as $t$ tends to $+\infty$. 
Exponential convergence 
results in the discrete setting were proved for instance in~\cite{CHJS14,CH14_FVCA7, BCCH16_HAL} in the case of 
a monotone discretization of dissipative equation (see also \cite{Bessemoulin-Chatard_PhD}). 
In order to check this asymptotic behavior 
at the discrete level, we introduce the discrete relative entropy $E_\Dd^w(\u)$ defined for all nonnegative 
$\u = {(u_\k,u_\s)}_{\k,\s} \in W_\Dd$ (i.e., such that $u_\b \ge 0$ for all $\b \in \Mm\cup\Vv$) by 
\be\label{eq:entrel_D}
E^w_\Dd(\u) = \sum_{\b \in \Mm \cup \Vv}m_\b \left( u_\b \log\left(\frac{u_\b}{w(\x_\b)}\right) - u_\b + w(\x_\b)\right).
\ee
The exponential convergence towards equilibrium is recovered as it appears clearly on Figure~\ref{fig:MeshKer}.
\begin{figure}[h!]
\includegraphics[width=4.cm]{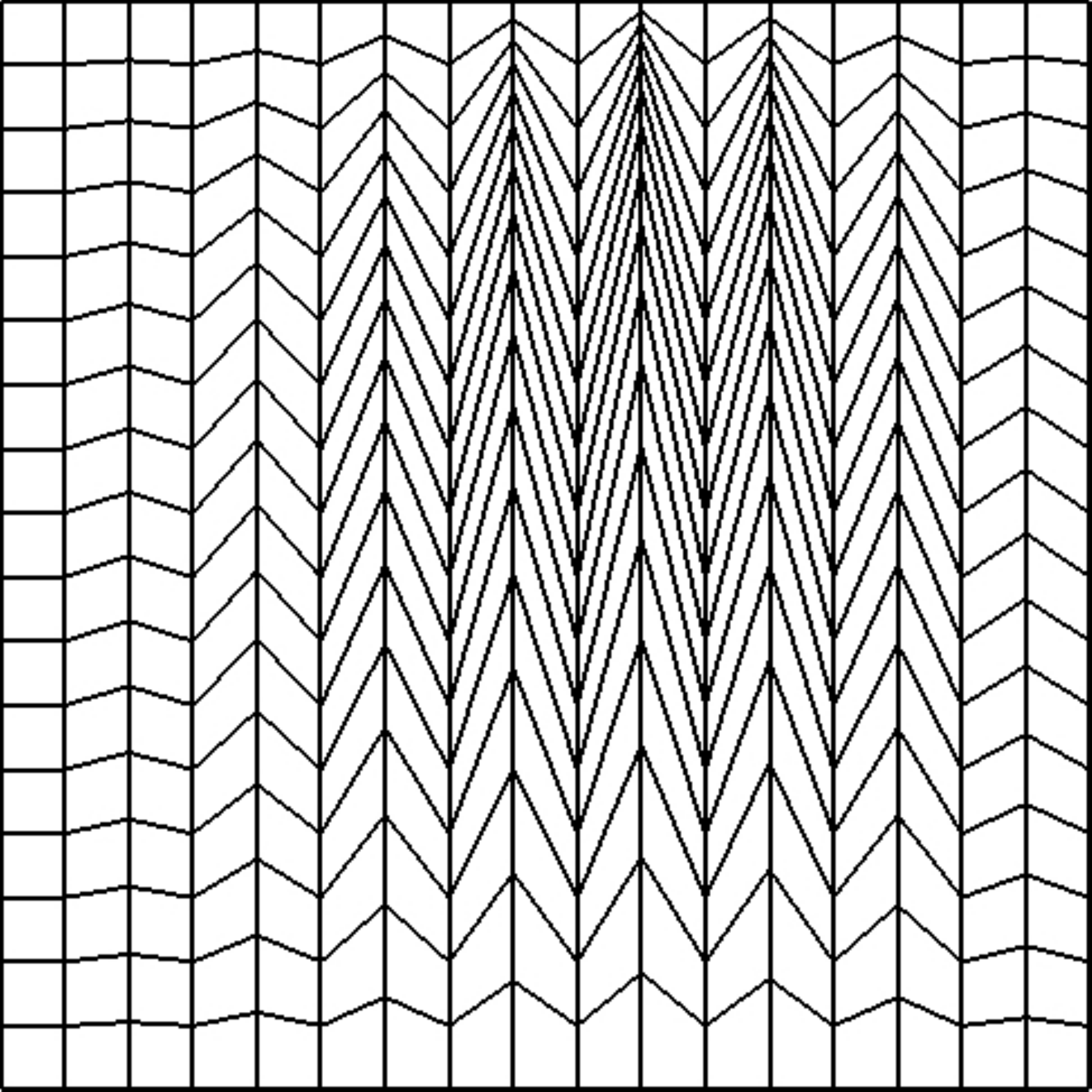} 
\hspace{1cm}
\includegraphics[width=6.cm]{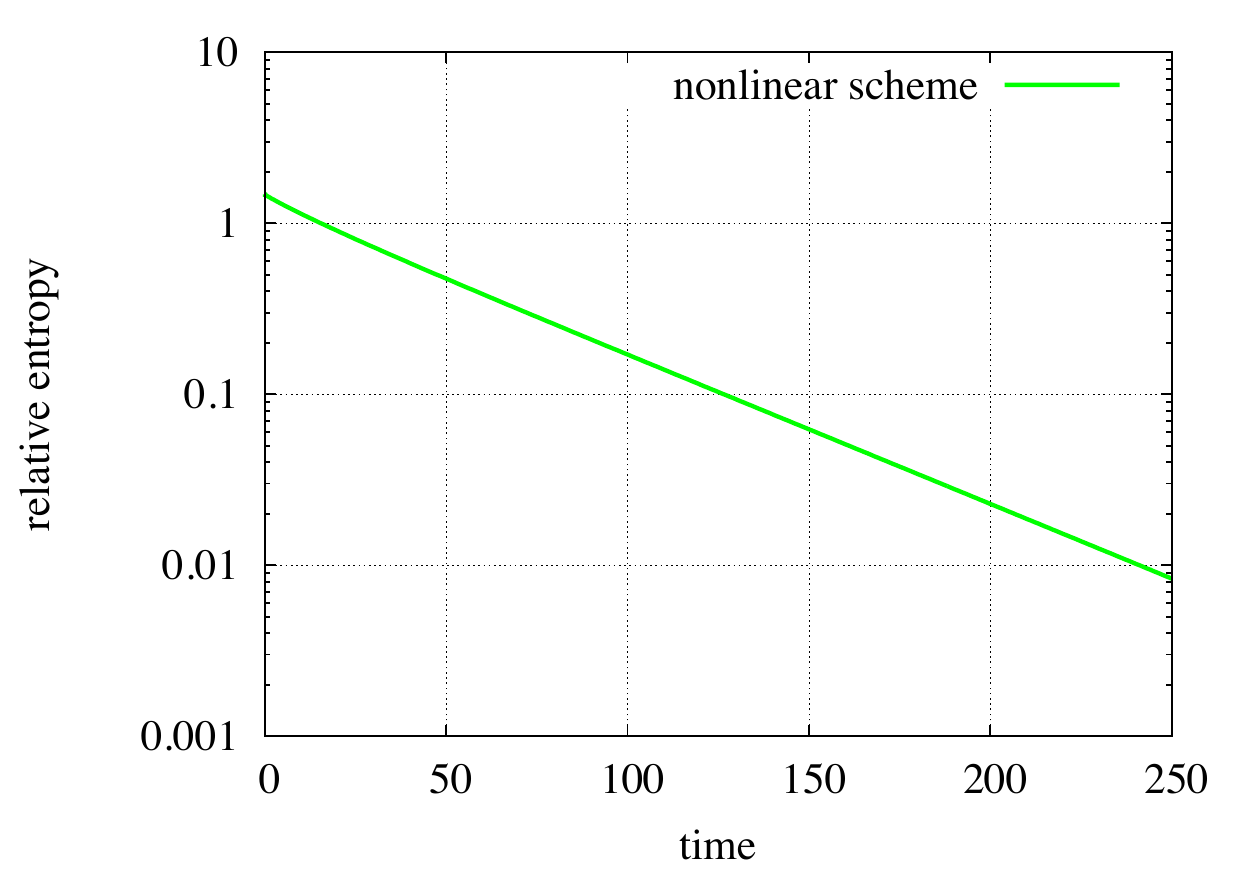}
\caption{Left: Kershaw mesh. Right: Evolution of the relative entropy $t \mapsto E^w_\Dd(\u(\cdot,t))$ 
on a logarithmic scale in function of time.}
\label{fig:MeshKer}
\end{figure}

\subsubsection{Test 2: Porous medium equation with Dirichlet boundary condition}

In this section, we apply our scheme to the case of the anisotropic porous medium equation 
\be
\p_t u - \div\left( \L \ \grad (u^2 ) \ \right) = 0 \;\; \text{ in } Q_\tf
\label{eq:ifct4}
\ee
for different choices of functions $\eta$ and $p$ with $\int^u \eta(a) p'(a) \d a = u^2$, 
namely 
\begin{enumerate}
\item[{\bf (a)}] $\eta(u) = 2 u^2$ and $p(u) = \log(u)$,
\item[{\bf (b)}] $\eta(u) =  2 u$ and $p(u) = u$, 
\item[{\bf (c)}] $\eta(u) = 1$ and $p(u) = u^2$.
\end{enumerate}

For the choice {\bf (a)}, the function $\eta$ is strictly convex. 
Therefore, the rigorous gradient flow structure of the problem 
corresponding to this choice of mobility function $\eta$ is unclear~\cite{DNS09}. 
The pressure function $p$ is singular near $0$, hence Lemma~\ref{lem:harnack} 
implies that the corresponding scheme is positivity preserving.
 
The choice {\bf (b)} with a  linear mobility corresponds to the now classical setting highlighted in~\cite{Otto01, Lisini09}.

Finally, the choice {\bf (c)} corresponds to the usual approach for discretizing 
the porous medium equation. The corresponding scheme enters into the framework of~\cite{EFGGH13}, where 
its convergence is proved. 

The problem is closed here with Dirichlet boundary conditions (destroying by the way the gradient flow structure 
but not the convergence of the scheme).
\medskip

\noindent
\emph{Comparison with a one-dimensional analytical solution.}
The numerical convergence of the three schemes has first been compared 
thanks to the 
following analytical solution (again built in $1$-dimension),
\be\label{eq:sol4}
\widehat{u}(x,y,t) =  
\max \left( 2 l_x t - x \ ,\ 0 \right), \quad \forall ((x,y),t)\in Q_\tf.
\ee
Note that \eqref{eq:sol4} is the unique weak solution corresponding to the initial condition
$u_0(x,y) = \widehat{u}(x,y,0)$
and to the Dirichlet boundary condition $u_D(x,y,t) = \widehat{u}(x,y,t)$ on $\partial\Omega\times(0,\tf)$.
Our numerical convergence study makes use of the family of triangular meshes 
already used for Test 1. 
Once again, the final time $\tf$ is fixed to $0.25$ and an anisotropic tensor is given by $l_x=1$ and $l_y=10$.

\begin{table}[h!]
\resizebox{\textwidth}{!}{
 \begin{tabular}{|cc|cc|cc|cc|cc|c|c|}
\hline
$h$  & $\#\Vv$ & $\dt_{\textrm{init}}$ &  $\dt_{\max}$  & $\textrm{err}_{L^2}$ & rate &
$\textrm{err}_{L^1}$ & rate &  $\textrm{err}_{L^\infty}$ & rate &
$u_{\textrm{min}}$ & \#Newton  \\
\hline
 0.306 & 37   & 0.001   & 0.01024  & 0.523E-02 &    -   &  0.997E-03 &    -   &  0.105E+00 &    -   &  0.000  &479\\
 0.153 & 129  & 0.00025 & 0.00256  & 0.205E-02 &  1.352 &  0.344E-03 &  1.535 &  0.522E-01 &  1.013 &  0.000  &1143\\
 0.077 & 481  & 0.00006 & 0.00064  & 0.898E-03 &  1.190 &  0.123E-03 &  1.490 &  0.259E-01 &  1.012 &  0.000  &2218\\
 0.038 & 1857 & 0.00002 & 0.00016  & 0.380E-03 &  1.240 &  0.417E-04 &  1.554 &  0.128E-01 &  1.012 &  0.000  &5652\\
\hline   
\end{tabular}
}
\caption{Test 2: Choice {\bf (a)} of mobility and pressure functions, convergence towards~\eqref{eq:sol4}.}
\label{tab:t2-lx1ly100-VAG}
\end{table}

\begin{table}[h!]
\resizebox{\textwidth}{!}{
\begin{tabular}{|cc|cc|cc|cc|cc|c|c|}
\hline
$h$  & $\#\Vv$ & $\dt_{\textrm{init}}$ &  $\dt_{\max}$  & $\textrm{err}_{L^2}$ & rate &
$\textrm{err}_{L^1}$ & rate &  $\textrm{err}_{L^\infty}$ & rate &
$u_{\textrm{min}}$ & \#Newton  \\
\hline
0.306 & 37   & 0.001   & 0.01024  &  0.769E-02 &    -   &  0.210E-02 &    -   &  0.645E-01 &    -   & -0.032  & 138\\
0.153 & 129  & 0.00025 & 0.00256  &  0.263E-02 &  1.546 &  0.613E-03 &  1.775 &  0.326E-01 &  0.983 & -0.017  & 383\\
0.077 & 481  & 0.00006 & 0.00064  &  0.897E-03 &  1.554 &  0.173E-03 &  1.823 &  0.164E-01 &  0.996 & -0.009  & 1246\\
0.038 & 1857 & 0.00002 & 0.00016  &  0.306E-03 &  1.551 &  0.481E-04 &  1.849 &  0.821E-02 &  0.996 & -0.005  & 4234\\
\hline
\end{tabular}
}
\caption{Test 2: Choice {\bf (b)} of mobility and pressure functions, convergence towards~\eqref{eq:sol4}.}
\label{tab:t2-lx1ly100-linear-b}
\end{table}

\begin{table}[h!]
\resizebox{\textwidth}{!}{
 \begin{tabular}{|cc|cc|cc|cc|cc|c|c|}
\hline
$h$  & $\#\Vv$ & $\dt_{\textrm{init}}$ &  $\dt_{\max}$  & $\textrm{err}_{L^2}$ & rate &
$\textrm{err}_{L^1}$ & rate &  $\textrm{err}_{L^\infty}$ & rate &
$u_{\textrm{min}}$ & \#Newton  \\
\hline
 0.306 & 37   & 0.001   & 0.01024  & 0.116E-01 &    -   &  0.371E-02 &    -   &  0.764E-01 &    -   & -0.065  &148\\
 0.153 & 129  & 0.00025 & 0.00256  & 0.423E-02 &  1.461 &  0.116E-02 &  1.672 &  0.388E-01 &  0.977 & -0.039  &436\\
 0.077 & 481  & 0.00006 & 0.00064  & 0.149E-02 &  1.501 &  0.337E-03 &  1.788 &  0.233E-01 &  0.737 & -0.021  &1438 \\
 0.038 & 1857 & 0.00002 & 0.00016  & 0.524E-03 &  1.513 &  0.932E-04 &  1.856 &  0.129E-01 &  0.856 & -0.010  &4912 \\
\hline   
\end{tabular}
}
\caption{Test 2: Choice {\bf (c)} of mobility and pressure functions, convergence towards~\eqref{eq:sol4}.}
\label{tab:t2-lx1ly100-linear}
\end{table}

We  observe in Tables~\ref{tab:t2-lx1ly100-VAG}--\ref{tab:t2-lx1ly100-linear} 
that second order convergence is destroyed for all the three schemes because 
of the lack of regularity of the exact solution. 
As expected, the discrete solution corresponding to the choice {\bf (a)} remains positive
while the discrete solutions to the schemes corresponding to the choices {\bf (b)} and {\bf (c)} 
suffer of undershoots.
The choice {\bf (b)} appears to be both cheaper and more accurate 
than the choice {\bf (c)}, and the amplitude of the undershoots is smaller. 

\begin{figure}[h!]
 \includegraphics[width=6.cm]{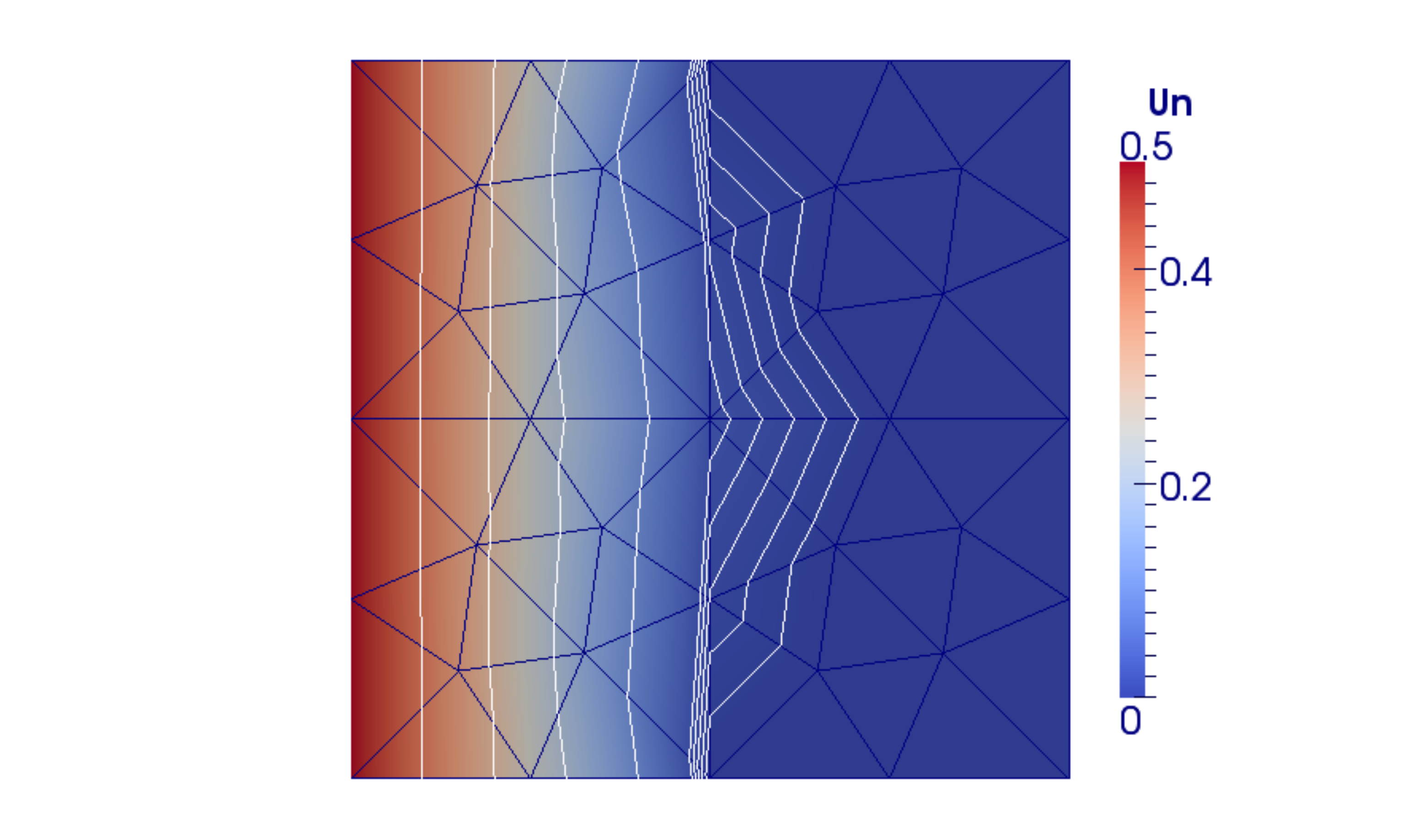} 
 \includegraphics[width=6.cm]{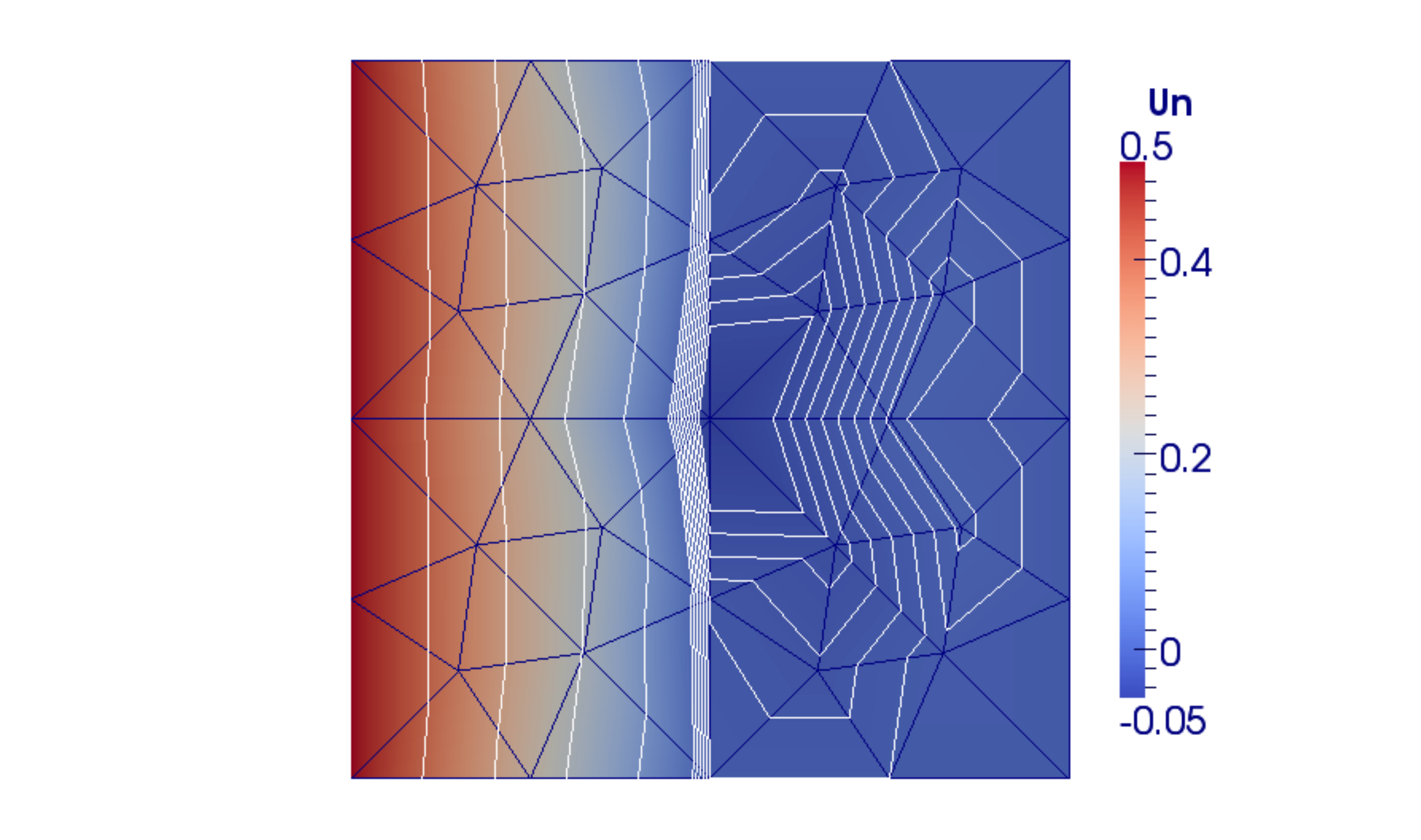}
\caption{Test 2. Coarsest grid. Discrete unknown $(u_{\s})_{\s\in \Vv_\k}$ and its iso-values. Choice {\bf (a)} (left) and {\bf (c)} (right) for $\eta$ and $p$.}
\label{fig:paraDege}
\end{figure}

Figure \ref{fig:paraDege} illustrates the iso-values of the piecewise affine functions defined on the triangular 
mesh $\Mm$ reconstructed thanks to its nodal values $\left(u_\s^n\right)_{\s\in\Vv}$ for the coarsest triangle grid 
at the final time $\tf$.
For the choice {\bf (a)} of the mobility and the pressure (left), the iso-values are chosen from $0$ to $0.025$ by step of $0.05$ and then from $0.1$ to $0.5$ by step of $0.1$. For the choice {\bf (c)} of the mobility and the pressure (right), the iso-values are taken from $-0.025$ to $0.025$ by step of $0.05$ and also from $0.1$ to $0.5$ by step of $0.1$.

\medskip

\noindent
\emph{Comparison with a two-dimensional analytical solution.}
We test our approach on the two-dimensional analytical solution 
\be\label{eq:sol-anal-2D}
\widehat{u}(x,y,t) = \dfrac{ \alpha (x-0.5)^2 + \beta (y-0.5)^2 }{1-t} \qquad  \mathrm{in}\;Q_\tf
\ee
of the anisotropic porous medium equation~\eqref{eq:ifct4}, where $\tf$ has been set to 0.25. 
The permeability tensor is still assumed to be diagonal with $l_x = 0.1$ and $l_y = 10$, 
and we have set $\alpha = \frac{1}{16 l_x}$ and $\beta = \frac{1}{16 l_y}$.
The problem is closed with Dirichlet boundary conditions and the initial condition 
corresponding to~\eqref{eq:sol-anal-2D}. The results are gathered in 
Tables~\ref{tab:t4-lx01ly10-VAG}--\ref{tab:t2-lx01ly10-linear}.

\begin{table}[h]
\resizebox{\textwidth}{!}{
 \begin{tabular}{|cc|cc|cc|cc|cc|c|c|}
\hline
$h$  & $\#\Vv$ & $\dt_{\textrm{init}}$ &  $\dt_{\max}$  & $\textrm{err}_{L^2}$ & rate &
$\textrm{err}_{L^1}$ & rate &  $\textrm{err}_{L^\infty}$ & rate &
$u_{\textrm{min}}$ & \#Newton  \\
\hline
 0.306 & 37   & 0.001   & 0.01024  &  0.270E-02 &    -   &  0.114E-02 &    -   &  0.120E-01 &    -   &  0.000 &  89\\
 0.153 & 129  & 0.00025 & 0.00256  &  0.942E-03 &  1.517 &  0.381E-03 &  1.578 &  0.473E-02 &  1.341 &  0.000 &  223\\
 0.077 & 481  & 0.00006 & 0.00064  &  0.293E-03 &  1.688 &  0.119E-03 &  1.686 &  0.163E-02 &  1.534 &  0.000 &  805\\
 0.038 & 1857 & 0.00002 & 0.00016  &  0.802E-04 &  1.868 &  0.334E-04 &  1.828 &  0.461E-03 &  1.826 &  0.000 &  3142\\
\hline   
\end{tabular}
}
\caption{Test 2: Choice {\bf (a)} of mobility and pressure functions, convergence towards~\eqref{eq:sol-anal-2D}.}
\label{tab:t4-lx01ly10-VAG}
\end{table}

\begin{table}[h!]
\resizebox{\textwidth}{!}{
\begin{tabular}{|cc|cc|cc|cc|cc|c|c|}
\hline
$h$  & $\#\Vv$ & $\dt_{\textrm{init}}$ &  $\dt_{\max}$  & $\textrm{err}_{L^2}$ & rate &
$\textrm{err}_{L^1}$ & rate &  $\textrm{err}_{L^\infty}$ & rate &
$u_{\textrm{min}}$ & \#Newton  \\
\hline
0.306 & 37   & 0.001   & 0.01024  & 0.645E-02 &    -   &  0.271E-02 &    -   &  0.271E-01 &    -   & -0.027 & 99\\
0.153 & 129  & 0.00025 & 0.00256  & 0.202E-02 &  1.676 &  0.828E-03 &  1.709 &  0.992E-02 &  1.447 & -0.008 & 237\\
0.077 & 481  & 0.00006 & 0.00064  & 0.604E-03 &  1.742 &  0.246E-03 &  1.753 &  0.341E-02 &  1.540 & -0.002 & 801\\
0.038 & 1857 & 0.00002 & 0.00016  & 0.161E-03 &  1.905 &  0.667E-04 &  1.882 &  0.966E-03 &  1.821 &  0.000 & 3140\\
\hline
\end{tabular}
}
\caption{Test 2: Choice {\bf (b)} of mobility and pressure functions, convergence towards~\eqref{eq:sol-anal-2D}.}
\label{tab:t4-lx01ly10-VAG-b}
\end{table}

\begin{table}[h]
\resizebox{\textwidth}{!}{
 \begin{tabular}{|cc|cc|cc|cc|cc|c|c|}
\hline
$h$  & $\#\Vv$ & $\dt_{\textrm{init}}$ &  $\dt_{\max}$  & $\textrm{err}_{L^2}$ & rate &
$\textrm{err}_{L^1}$ & rate &  $\textrm{err}_{L^\infty}$ & rate &
$u_{\textrm{min}}$ & \#Newton  \\
\hline
 0.306 & 37   & 0.001   & 0.01024  &  0.102E-01 &    -   &  0.448E-02 &    -   &  0.575E-01 &    -   & -0.046 &  128\\
 0.153 & 129  & 0.00025 & 0.00256  &  0.321E-02 &  1.661 &  0.134E-02 &  1.739 &  0.194E-01 &  1.569 & -0.011 &  250\\
 0.077 & 481  & 0.00006 & 0.00064  &  0.933E-03 &  1.783 &  0.383E-03 &  1.811 &  0.553E-02 &  1.808 & -0.003 &  810\\
 0.038 & 1857 & 0.00002 & 0.00016  &  0.244E-03 &  1.933 &  0.101E-03 &  1.927 &  0.147E-02 &  1.914 & -0.001 &  3140\\
\hline   
\end{tabular}
}
\caption{Test 2. Choice {\bf (c)} of mobility and pressure functions, convergence towards~\eqref{eq:sol-anal-2D}. } 
\label{tab:t2-lx01ly10-linear}
\end{table}

As expected, the choice {\bf (a)} leads to a positivity preserving scheme, contrarily to 
the choices {\bf (b)} and {\bf (c)}. Moreover, the scheme  {\bf (a)} is the most accurate 
and does not come with an additional cost. 

\subsubsection{Test 3: Porous medium equation with drift}

In this third test case, we have set $\eta(u)=u$ on $\R_+$ and $p(u)=2u$ and $g = 1$,
leading to the degenerate problem
\be
\p_t u - \div\left( \L \ ( \grad (u^2) - u {\mathbf g}  ) \ \right) = 0 \;\; \text{ in } Q_\tf.
\label{eq:pmedrift}
\ee
The problem is endowed with Dirichlet boundary conditions.
The tensor $\L$ is chosen to be diagonal with $l_x = 1$ and $l_y = 100$.
We compare the results obtained by \eqref{eq:syst} with those obtained using,
instead of \eqref{eq:Fks}, this particular definition of the fluxes
\begin{multline}\label{eq:Fks-pmedrift}
\widehat{F}_{\k,\s}(\u^n) = \sum_{\s'\in\Vv_\k} a^\k_{\s,\s'} ((u_\k^n)^2 - (u_{\s'}^n)^2) +
\frac{u_\k^n+u_\s^n}2
  \sum_{\s'\in\Vv_\k} a^\k_{\s,\s'}(V_\k -  V_{\s'}), \\
\quad \forall \k \in \Mm, \; \forall \s \in \Vv_\k.
\end{multline}
The resulting scheme is called the {\em quasilinear scheme}.
The numerical convergence of both schemes has been compared 
on the sequence of triangular meshes already used in the previous tests, 
thanks to the following analytical solution (again built in $1$-dimension),
\be\label{eq:sol-pmedrift}
\widehat{u}(x,y,t) =  
\max \left( \beta t - x \ ,\ 0 \right), \quad \forall ((x,y),t)\in\Omega\times(0,\tf),
\ee
with $\beta= l_x ( 2 + g)$. The profile \eqref{eq:sol-pmedrift} is the unique weak solution 
corresponding to the initial condition $u_0(x,y) = \widehat{u}(x,y,0)$ in $\O$ and 
the Dirichlet boundary condition $u_D(x,y,t) = \widehat{u}(x,y,t)$ on $\p\O\times(0,\tf)$.

\begin{table}[h!]
\resizebox{\textwidth}{!}{
 \begin{tabular}{|cc|cc|cc|cc|cc|c|c|}
\hline
$h$  & $\#\Vv$ & $\dt_{\textrm{init}}$ &  $\dt_{\max}$  & $\textrm{err}_{L^2}$ & rate &
$\textrm{err}_{L^1}$ & rate &  $\textrm{err}_{L^\infty}$ & rate &
$u_{\textrm{min}}$ & \#Newton  \\
\hline
 0.306 & 37   & 0.001   & 0.01024  &  0.130E-01 &    -   &  0.423E-02 &    -   &  0.890E-01 &    -   & -0.046 &187\\
 0.153 & 129  & 0.00025 & 0.00256  &  0.495E-02 &  1.398 &  0.133E-02 &  1.675 &  0.496E-01 &  0.843 & -0.032 &552\\
 0.077 & 481  & 0.00006 & 0.00064  &  0.184E-02 &  1.428 &  0.397E-03 &  1.741 &  0.283E-01 &  0.808 & -0.017 &1609\\
 0.038 & 1857 & 0.00002 & 0.00016  &  0.660E-03 &  1.479 &  0.116E-03 &  1.771 &  0.145E-01 &  0.970 & -0.009 &5586\\
\hline   
\end{tabular}
}
\caption{Test 3. Nonlinear scheme  \eqref{eq:syst}.}
\label{tab:t3-lx1ly100-VAG}
\end{table}

\begin{table}[h!]
\resizebox{\textwidth}{!}{
 \begin{tabular}{|cc|cc|cc|cc|cc|c|c|}
\hline
$h$  & $\#\Vv$ & $\dt_{\textrm{init}}$ &  $\dt_{\max}$  & $\textrm{err}_{L^2}$ & rate &
$\textrm{err}_{L^1}$ & rate &  $\textrm{err}_{L^\infty}$ & rate &
$u_{\textrm{min}}$ & \#Newton  \\
\hline
 0.306 & 37   & 0.001   & 0.01024  &  0.154E-01 &    -   &  0.568E-02 &    -   &  0.939E-01 &    -   & -0.068  &193\\
 0.153 & 129  & 0.00025 & 0.00256  &  0.671E-02 &  1.201 &  0.213E-02 &  1.416 &  0.613E-01 &  0.615 & -0.048  &642\\
 0.077 & 481  & 0.00006 & 0.00064  &  0.271E-02 &  1.309 &  0.702E-03 &  1.600 &  0.326E-01 &  0.910 & -0.027  &2178\\
 0.038 & 1857 & 0.00002 & 0.00016  &  0.104E-02 &  1.384 &  0.212E-03 &  1.725 &  0.170E-01 &  0.938 & -0.015  &7365\\
\hline   
\end{tabular}
}
\caption{Test 3. Quasilinear scheme, fluxes defined by~\eqref{eq:Fks-pmedrift}.}
\label{tab:t3-lx1ly100-linear}
\end{table}

Here again, the convergence orders of both scheme are similar, but strictly lower than $2$ because of the 
lack of regularity of the exact solution. 
Both schemes violate the positivity of the solution in this case, but the amplitude of the 
undershoots is smaller for the nonlinear scheme. There is no contradiction 
here with Lemma~\ref{lem:harnack} since $p$ is not singular at $u=0$.
Our nonlinear scheme is slightly more accurate, produces undershoots 
with a smaller amplitude, and is cheaper than the quasilinear one. 

\subsubsection{Test 4. A heterogeneous test case}\label{sssec:hetero}

The last test aims to illustrate the ability of the scheme to deal with 
heterogeneous situations. Motivated by an application to complex flows 
in porous media (see for instance~\cite{CP12,CGM15}), we test 
the nonlinear VAG scheme in a slightly more complicated configuration 
where both the permeability tensor $\L$ and the pressure function $p$ depend 
on $\x$ in a discontinuous way. 
More precisely, the domain $\O = (0,1)^2$ is made of two open subdomains 
$\O_1$ (the \emph{drain}) and $\O_2$ (the {\em barrier}) with 
$\ov \O = \ov \O_1 \cup \ov \O_2$ and $\O_1 \cap \O_2 = \emptyset$ (see 
Figure~\ref{fig:meshHete} for a representation of $\O_1$ and $\O_2$). 
The permeability tensor and the pressure function are defined by 
$$
\L(\x) = \begin{cases}
\L_1 = {\bf I}_d & \text{if}\; \x \in \O_1, \\
\L_2 = \begin{pmatrix}
1 & 0 \\ 0 & 0.01 
\end{pmatrix}
& \text{if}\; \x \in \O_2, \\
\end{cases}
$$
and 
$$
p(u,\x) = \begin{cases}
p_1(u) = 3\log(u) & \text{if}\; \x \in \O_1,\\
p_2(u) = \log(u)& \text{if}\; \x \in \O_2. 
\end{cases}
$$
The mobility function is linear and does not depend on $\x$, i.e., $\eta(u) = u$. 
For the sake of simplicity, we have set $V=0$.
At the interface $J$ between $\O_1$ and $\O_2$, the flux and the pressure are assumed to be continuous, 
i.e., denoting by $u_i$ the restriction of $u$ to $\O_i$ and by ${\n}_{i}$ the normal to $J$ 
outward w.r.t. $\O_i$, we require
\be\label{eq:jump}
3 \L_1 \grad u_1 \cdot {\n}_{1} +  \L_2 \grad u_2 \cdot \n_{2} = 0, \quad\text{and}\quad p_1(u_1) = p_2(u_2) 
\quad \text{on}\; J\times(0,\tf).
\ee
The problem is complemented with the boundary conditions 
\begin{itemize}
\item $p(u,\x) = 0$ (hence $u=1$) on the bottom boundary,
\item $p(u,\x) = -4$ (hence $u \simeq 0.018$) on the top boundary,
\item $\L \grad u \cdot \n = 0$ on the lateral boundaries.
\end{itemize}
The initial data is chosen at equilibrium, with $p(u_0, \x) = -4$ in the whole $\O$.
Existence and uniqueness for this problem follow from the 
analysis carried out in~\cite{NoDEA}.

Since $u$ is discontinuous across $J$ (in opposition to the pressure $p(u,\cdot)$ following~\eqref{eq:jump}), 
it is natural to choose $p$ rather than $u$ as the primary variable of the numerical scheme 
(cf.~\cite{HF08,EGHM14_zamm}, we refer to~\cite{BCtau} for an alternate strategy that 
improves robustness) in order to avoid the complex treatment of the jump condition~\eqref{eq:jump} 
at the interface performed for instance in~\cite{NoDEA,FVbarriere,EMS09,BCH13}. 

The mesh $\Mm$ is assumed to be compatible with the geometry of $\O$, in the sense that 
$\k \in \Mm$ is either contained in $\O_1$ or $\O_2$, but $J \cap \k = \emptyset$ for all $\k \in \Mm$ (cf. Figure~\ref{fig:meshHete}).
Define the functions $u_\k:\R \to (0,\infty)$ as the inverse of $p(\cdot,\x_\k)$ for all $\k \in \Mm$. 
The subset of $\Vv$ made of vertices belonging to the top or bottom boundaries where 
Dirichlet boundary conditions hold is denoted by $\Vv_{\rm ext}$. We also make use of the notations 
$\Vv_{\rm int} = \Vv \setminus \Vv_{\rm ext}$ and $\Vv_{\k,\rm int} = \Vv_{\rm int} \cap \Vv_\k$ for $\k \in \Mm$.
The scheme~\eqref{eq:syst} 
expressed with $p$ as a primary variable consists in finding $\boldsymbol{p} = \left(p_\k^n, p_\s^n\right)_{\k,\s,n}$ 
in $W_{\Dd,\bdt}$ such that for all $n\ge 1$, 
$$
\left\{
\begin{array}{ll}
\ds \frac{u_\k(p_\k^n) - u_\k(p_\k^{n-1})}{\dt_n} m_\k + \sum_{\s \in \Vv_\k} F_{\k,\s}^n = 0, & \forall \k \in \Mm, \\[5pt]
\ds \sum_{\k \in \Mm_\s} \frac{u_\k(p_\s^n) - u_\k(p_\s^{n-1})}{\dt_n} m_{ \k,\s} + \sum_{\k \in \Mm_\s} F_{\s,\k}^n = 0, &\forall \s \in \Vv_{\rm int}, \\[5pt]
F_{\k,\s}^n + F_{\s,\k}^n = 0, & \forall \k \in \Mm, \forall \s \in \Vv_{\k,\rm int}, \\[5pt]
\ds F_{\k,\s}^n = \sqrt{\eta_{\k,\s}^n} \sum_{\s' \in \Vv_\k} a_{\s,\s'}^\k \sqrt{\eta_{\k,\s'}^n} (p_\k^n - p_{\s'}^n), 
& \forall \k \in \Mm, \; \forall \s \in \Vv_\k, \\[5pt]
\ds \eta_{\k,\s}^n = {\frac{u_\k(p_\k^n) + u_\k(p_\s^n)}2},& \forall \k \in \Mm, \, \forall \s \in \Vv_\k.
\end{array}\right.
$$
\begin{figure}[h!]
\centering
\includegraphics[scale=0.125]{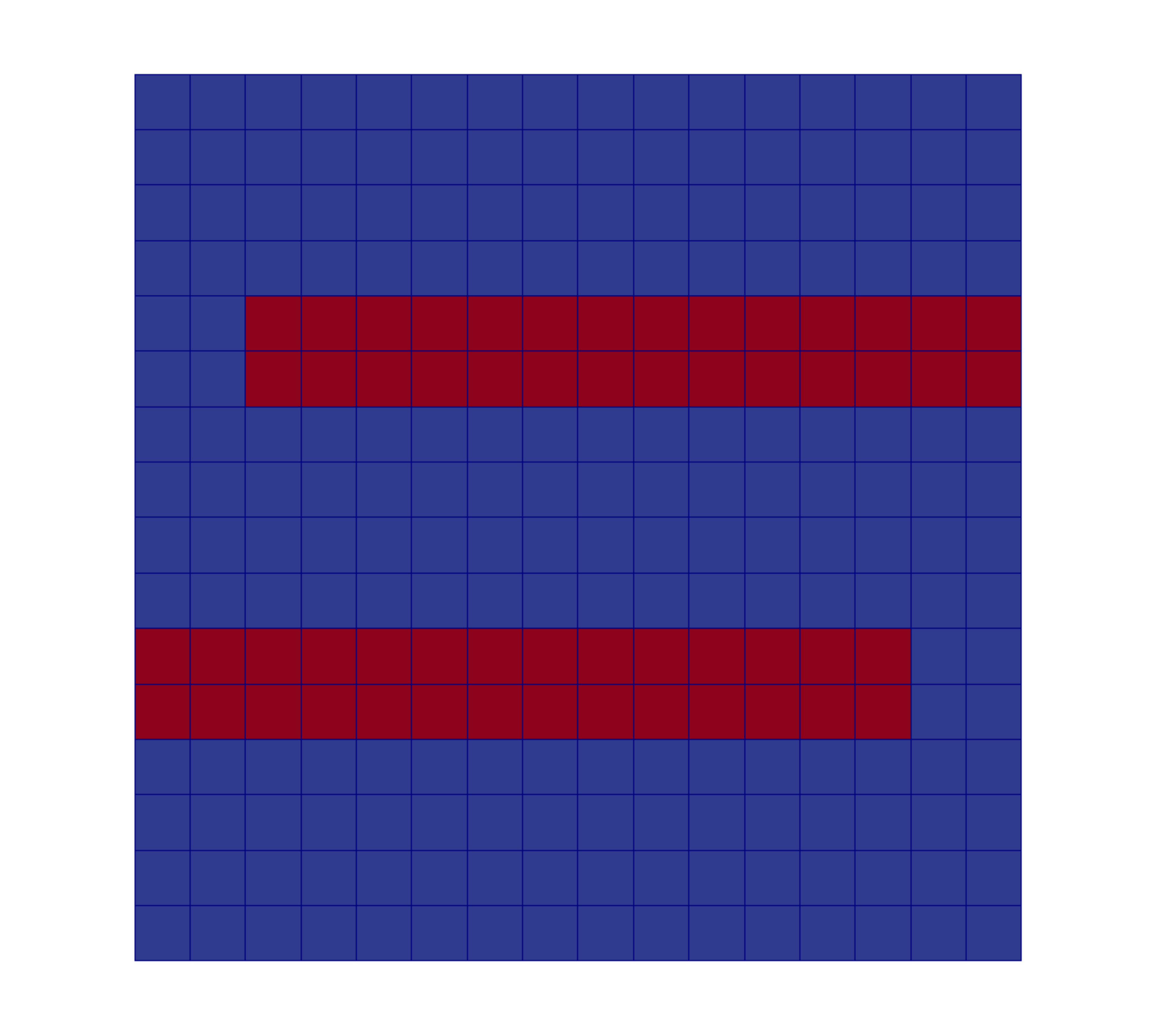}
\includegraphics[scale=0.125]{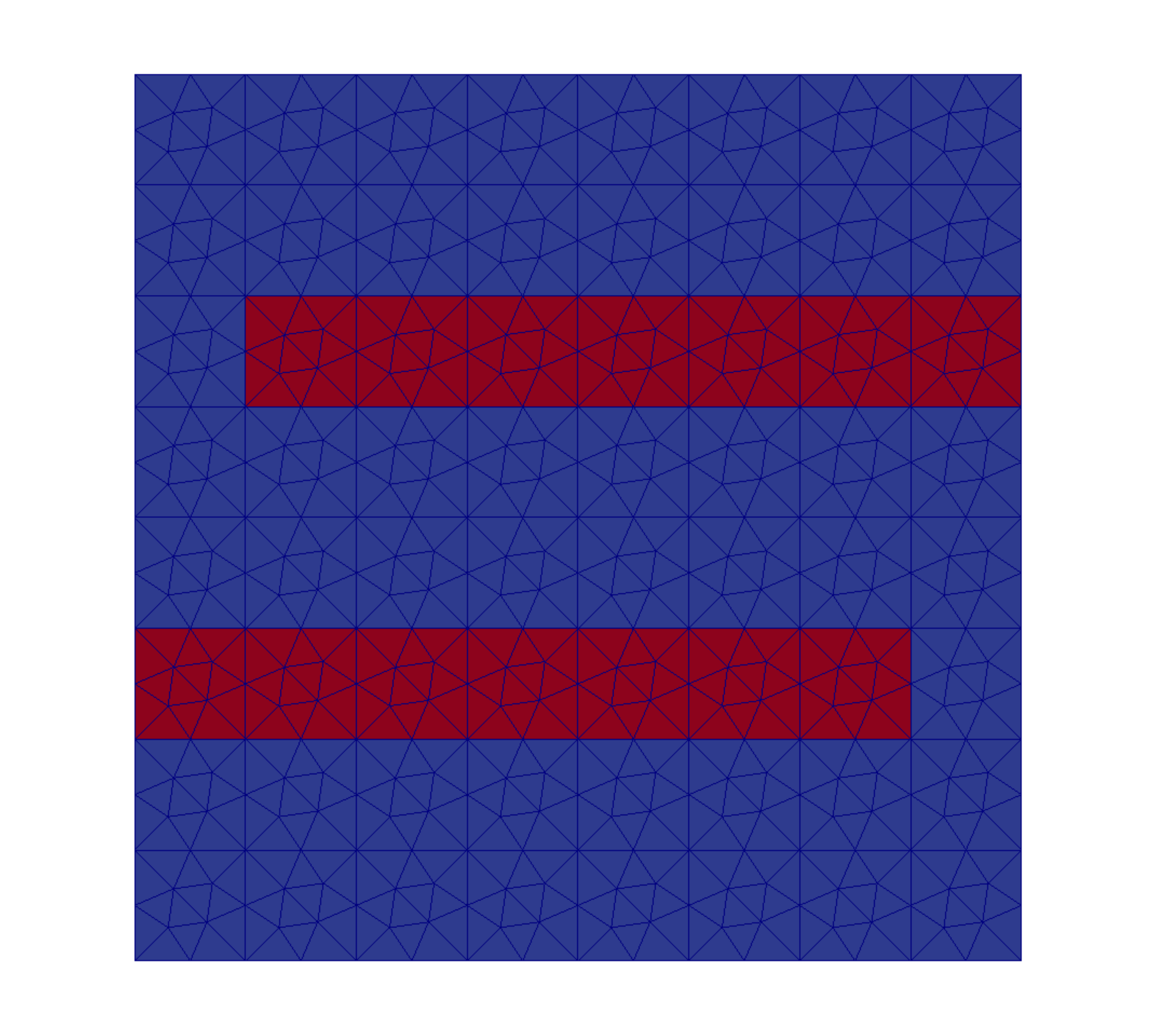}
\caption{Test 4. Illustration of the two sub-domains: the {\em drain} $\O_1$ in blue 
and the {\em barriers} $\O_2$ in red. Left: cartesian grid. Right: unstructured triangular grid.}
\label{fig:meshHete}
\end{figure}

We observe on Figure \ref{fig:resu-hete} that the results on the triangular mesh (with $481$ nodes) 
and on the cartesian mesh (with $289$ nodes) are similar. 
Moreover, the numbers of Newton-Raphson iteration needed to compute both solutions 
are of the same order.

\begin{figure}[htp]
  \centering
  \subfloat{\includegraphics[scale=0.15]{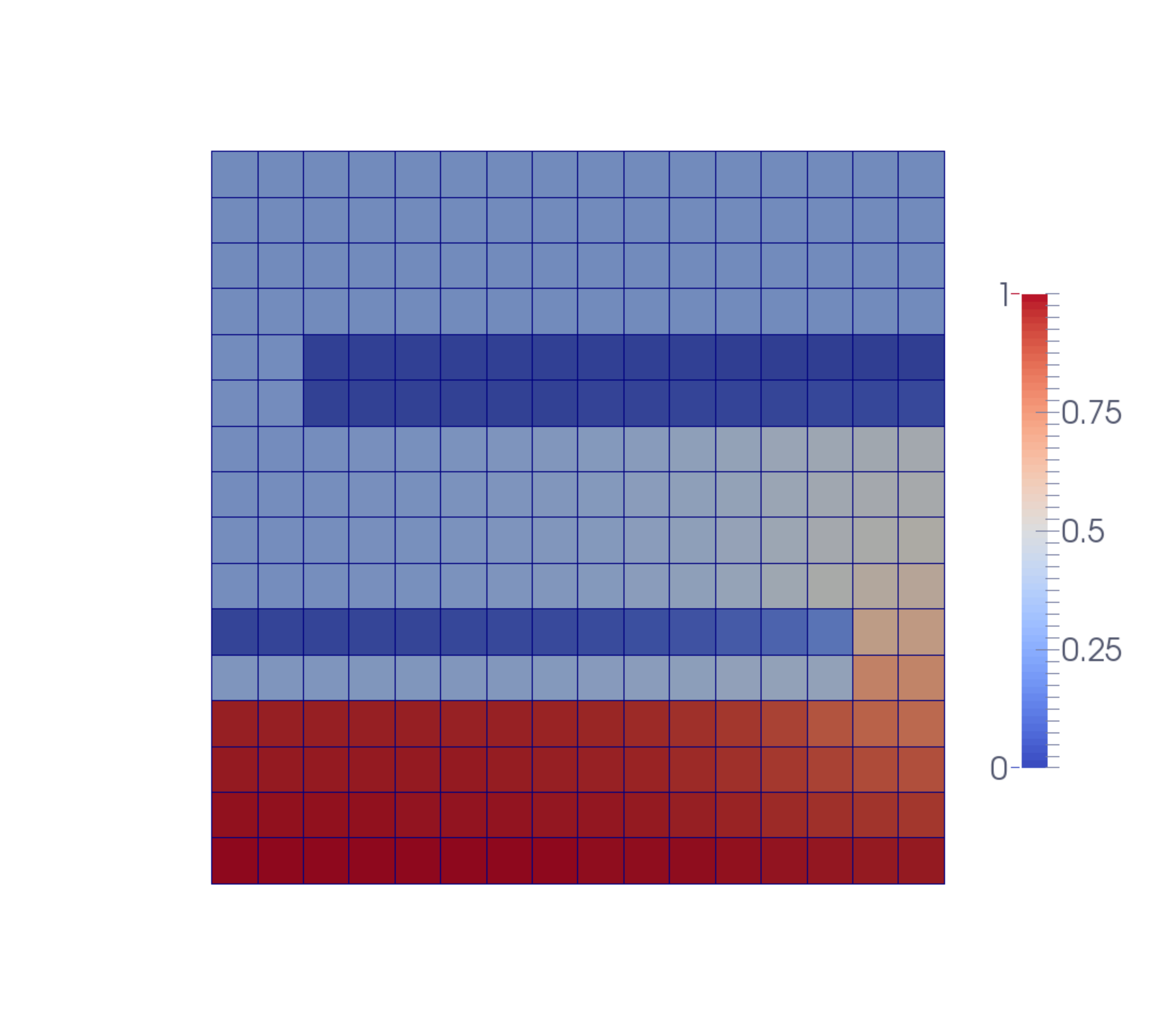}}
  \hspace{5pt}
  \subfloat{\includegraphics[scale=0.15]{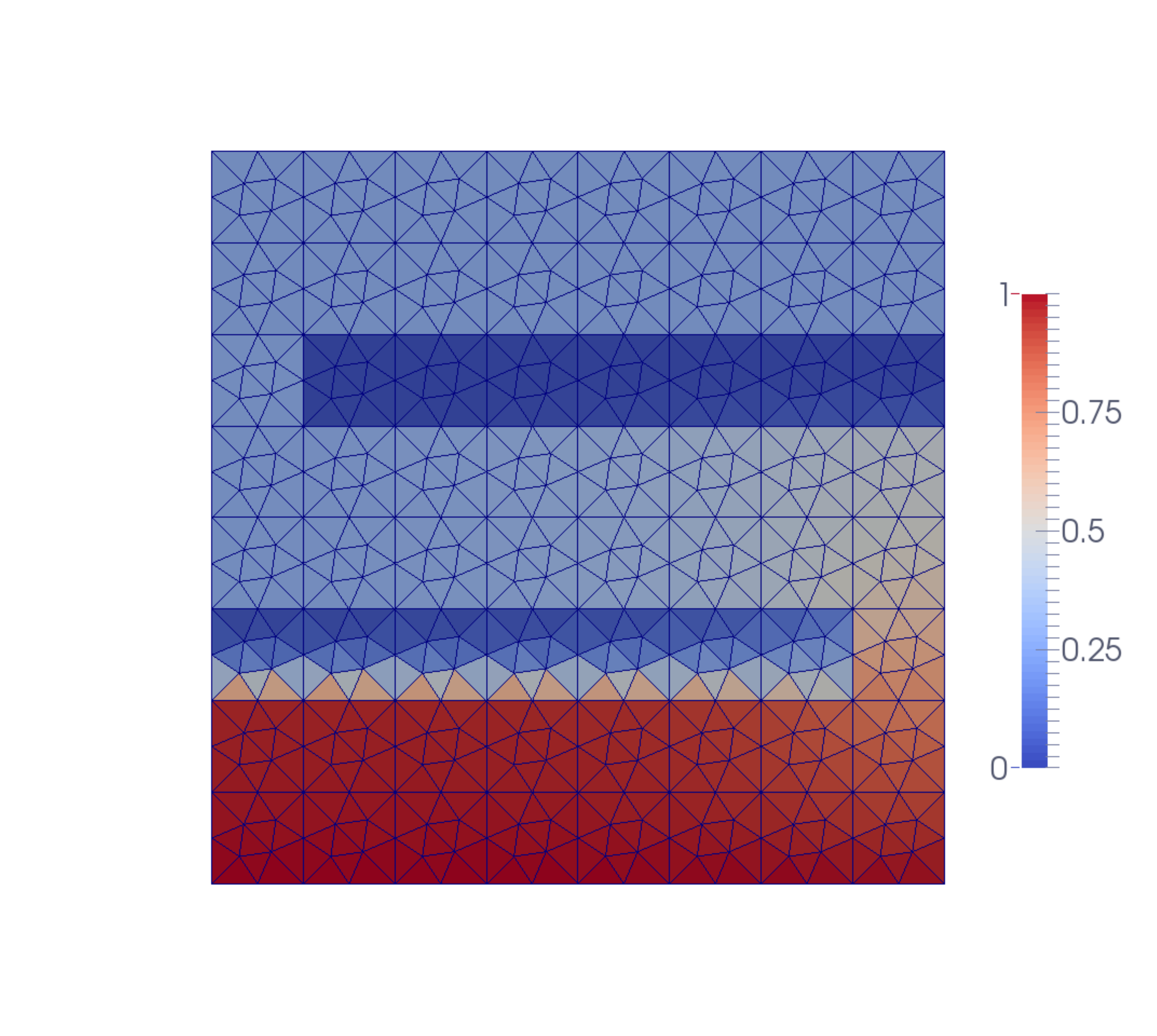}}
  \vspace*{-20pt}
  \subfloat{\includegraphics[scale=0.15]{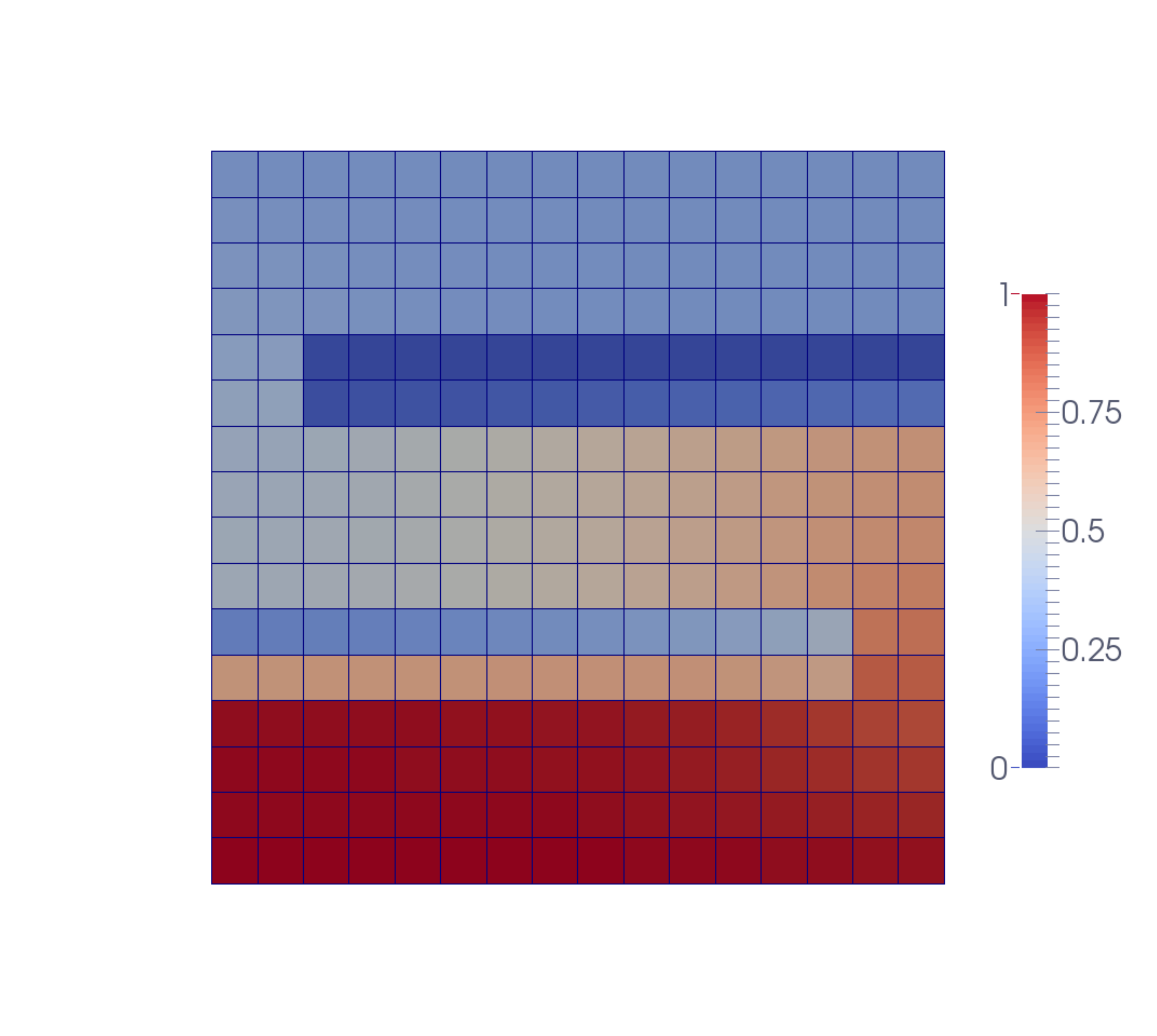}}
  \hspace{5pt}
  \subfloat{\includegraphics[scale=0.15]{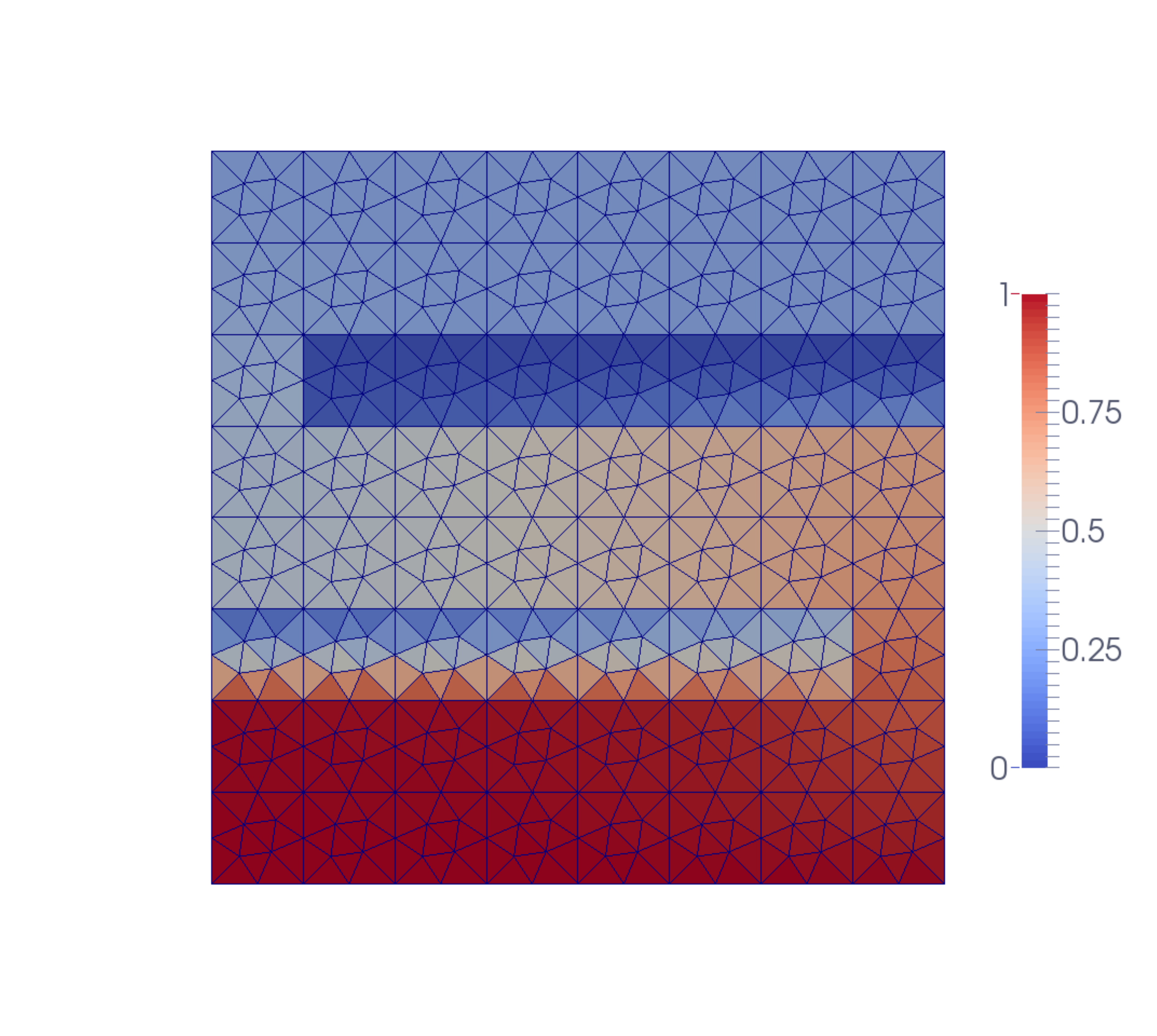}}
  \vspace*{-20pt}
  \subfloat{\includegraphics[scale=0.15]{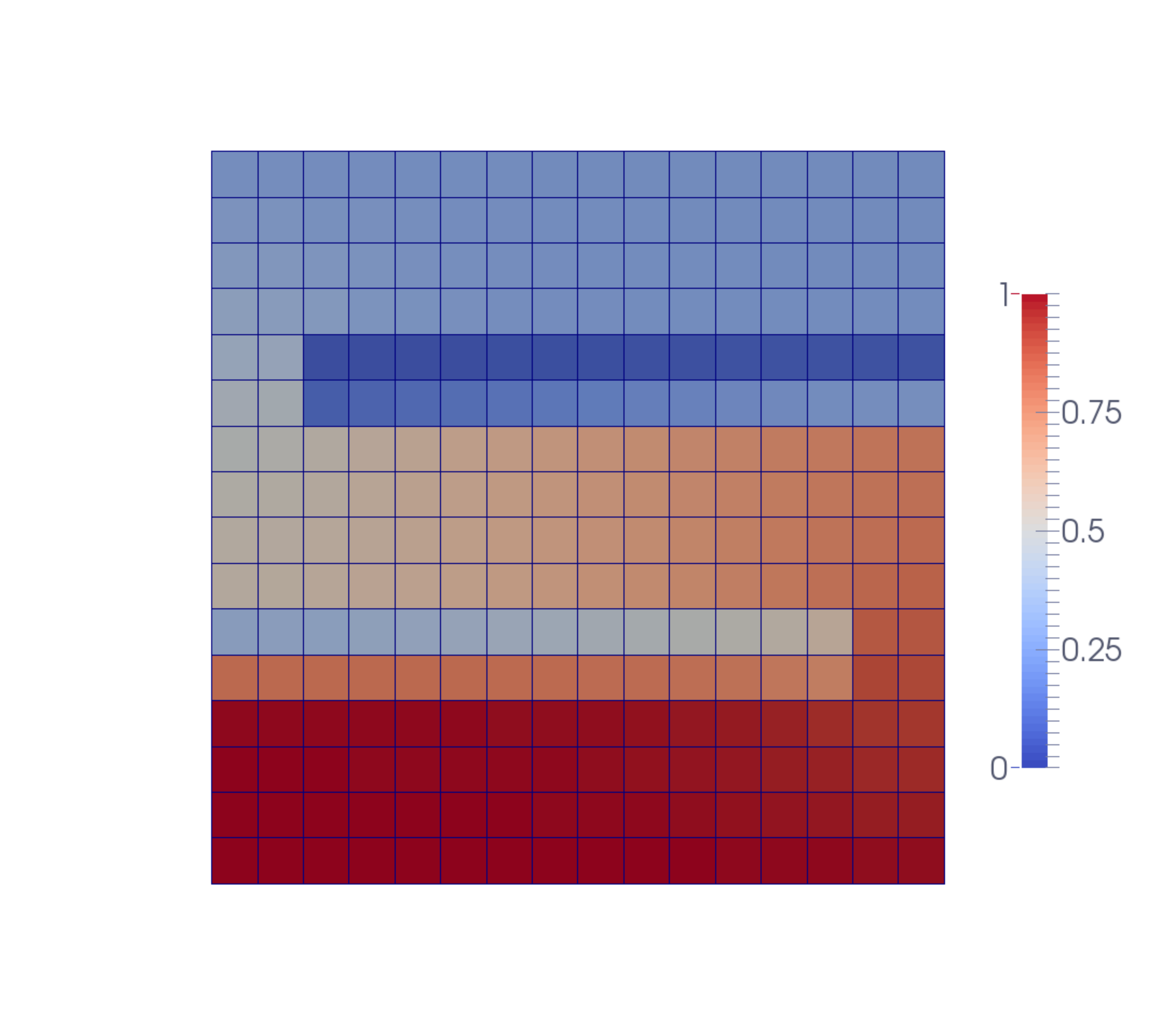}}
  \hspace{5pt}
  \subfloat{\includegraphics[scale=0.15]{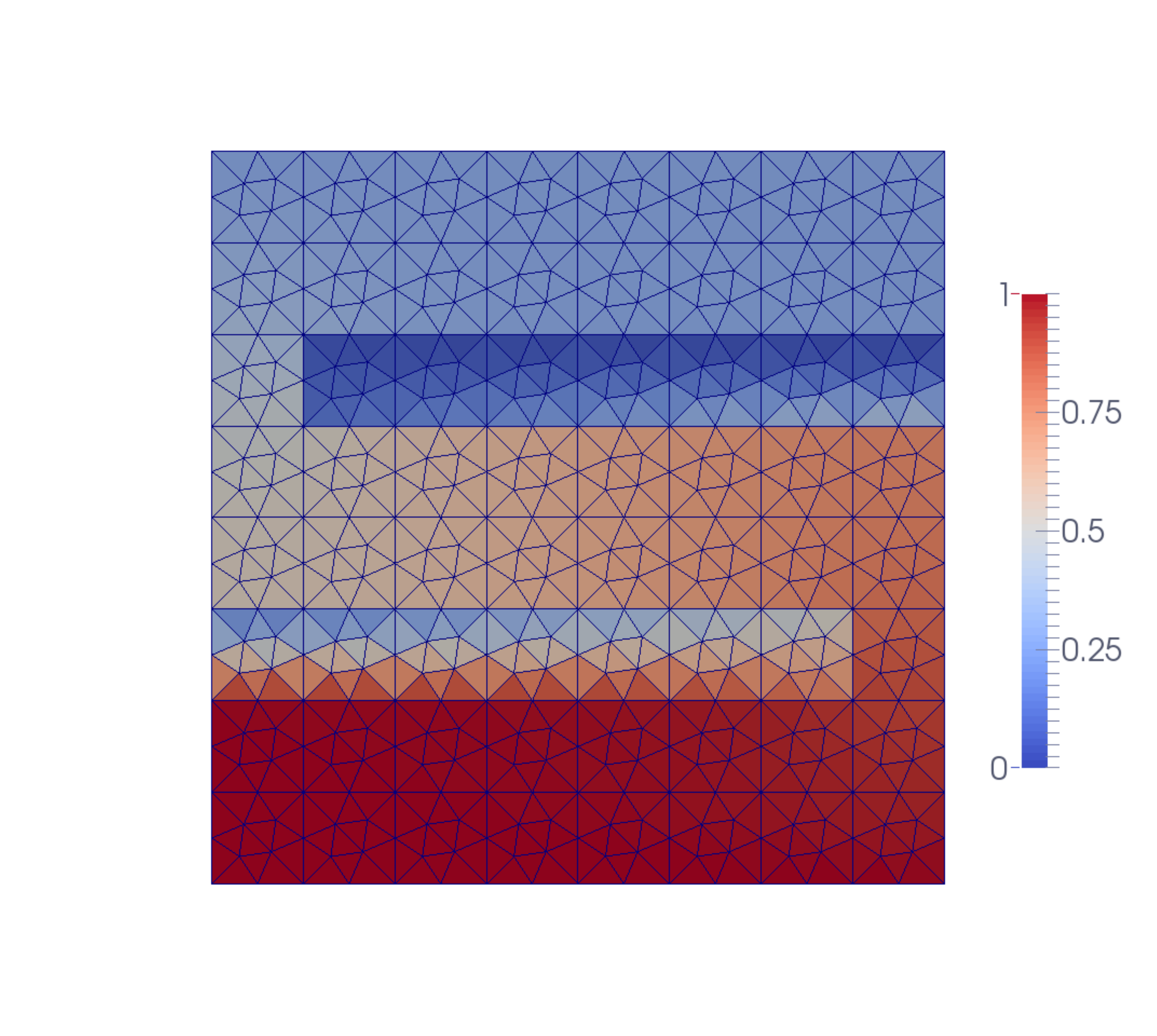}}
  \caption{Approximation of $u(p,\x)$ at times $t = 0.05$,  $t=0.2$, and  $t=1$ for the two different meshes.}
  \label{fig:resu-hete}
\end{figure}

\appendix

\section{Some lemmas related to the VAG discretization} \label{sec:VAG-lem} 

This appendix gathers lemmas on some properties of the VAG discretization that 
are independent of the continuous problem (and thus of the scheme).
In what follows,  $\Dd = (\Mm,\Tt)$ denote a discretization of $\O$ 
as prescribed in \S\ref{sssec:mesh-space}, and $\pi_\Tt, \pi_\Mm, \pi_\Dd$ 
and $\grad_\Tt$ are the corresponding reconstruction operators.

\begin{lem} \label{lem:cond-Ak} 
For $\k \in \Mm$, let $\A_\k = \left(a^\k_{\s,\s'} \right)_{\s,\s' \in \Vv_\k} $ be the matrix defined by~\eqref{eq:akss}, 
then there exists $C$ depending only on $\L$, $\theta_\Tt$ and $\ell_\Dd$ (but not on $\k$) such that ${\rm Cond} _2(\A_\k) \le C$.
\end{lem} 
\begin{proof} 
Following~\cite[Lemma 3.2]{BM13}, there exist $C_1,C_2 >0$ depending only on $\theta_\Tt$ and $\ell_\Dd$ 
such that, for all $\u \in W_\Dd$ and all $\k \in \Mm$, one has 
 $$
 C_1 \frac{\meas(\k)} {(h_\k)^2} \sum_{\s \in \Vv_\k} \left(u_\s - u_\k\right)^2 \le {\|\grad_\Tt \u\|} ^2_{L^2(\k)} \le   C_2 \frac{\meas(\k)} {(h_\k)^2} \sum_{\s \in \Vv_\k} \left(u_\s - u_\k\right)^2, 
 $$
 where $h_\k$ denotes the diameter of the cell $\k \in \Mm$.
 As a consequence, one has 
 $$
 \lambda_\star C_1 \frac{\meas(\k)} {(h_\k)^2} |\bd_\k \u|^2 \le \bd_\k \u \cdot \A_\k \bd_\k \u = \int_\k \L \grad_\Tt \u \cdot \grad_\Tt \u\, \d\x \le \lambda^\star C_2 \frac{\meas(\k)} {(h_\k)^2} |\bd_\k \u|^2.
 $$
Since the application $\bd_\k: W_\Dd \to \R^{\ell_\k} $ is onto, we deduce that
$$
\lambda_\star C_1 \frac{\meas(\k)} {(h_\k)^2} |\bv|^2 \le \bv \cdot \A_\k \bv \le \lambda^\star C_2 \frac{\meas(\k)} {(h_\k)^2} |\bv|^2, 
\qquad \forall \bv \in \R^{\ell_\k}, 
$$
and thus that 
$
{\rm Cond} _2(\A_\k) \le \frac{\lambda^\star C_2} {\lambda_\star C_1}.
$
\end{proof}

\begin{lem} \label{lem:abs-Ak} 
There exists $C\ge 1$ depending only on $\L$, $\theta_\Tt$ and $\ell_\Dd$  such that, for all $\k \in \Mm$ 
and all $\bv = {(v_\s)} _{\s \in \Vv_\k} \in \R^{\ell_\k} $, 
one has 
$$
\sum_{\s \in \Vv_\k} \left( \sum_{\s' \in \Vv_\k} |a^\k_{\s,\s'} |\right) \left(v_\s\right)^2 \le C\; \bv \cdot \A_\k \bv.
$$
\end{lem} 
\begin{proof} 
Denoting by $\| \cdot \|_{q} $ the usual matrix $q$-norm, 
one has 
$$
\sum_{\s \in \Vv_\k} \left( \sum_{\s' \in \Vv_\k} |a^\k_{\s,\s'} |\right) \left(v_\s \right)^2 \le {\| \A_\k \|} _1 | \bv |^2.
$$
Since the dimension of the space $\R^{\ell_\k} $ is bounded by $\ell_\Dd$, there exists $C_1$ depending only on 
$\ell_\Dd$ such that $\|\A_\k\|_1 \le C_1 \|\A_\k\|_2$, so that 
\be\label{eq:abs-Ak1} 
\sum_{\s \in \Vv_\k} \left( \sum_{\s' \in \Vv_\k} |a^\k_{\s,\s'} |\right) \left(v_\s \right)^2 \le C_1 {\| \A_\k \|} _2 |  \bv |^2.
\ee
On the other hand, since $\A_\k$ is symmetric definite and positive, one has 
$$
 \bv \cdot \A_\k \bv \ge \frac{{\| \A_\k \|} _2} {{\rm Cond} _2(\A_\k)} |\bv|^2.
$$
Using Lemma~\ref{lem:cond-Ak}, we obtain that there exists $C_2>0$ depending only on $\L$, $\theta_\Tt$ and $\ell_\Dd$ 
such that 
\be\label{eq:abs-Ak2} 
\bv \cdot \A_\k \bv \ge C_2 {{\| \A_\k \|} _2} | \bv |^2.
\ee
Putting~\eqref{eq:abs-Ak1} and \eqref{eq:abs-Ak2} together, we conclude the proof of Lemma~\ref{lem:abs-Ak} by choosing $C = \frac{C_1} {C_2} $.
\end{proof} 

\begin{lem} \label{lem:poids-Aks} 
Let $\k \in \Mm$ and $\A_\k = (a^\k_{\s,\s'})_{\s,\s'\in\Vv_\k} \in \R^{\ell_\k\times\ell_\k} $ be the matrix defined by~\eqref{eq:akss}.
Let $\bmu_\k = \left(\mu_{\k,\s} \right)_{\s \in \Vv_\k} \in \R^{\ell_\k} $ and $\bv \in W_\Dd$, then 
$$
\sum_{\s \in \Vv_\k} \sum_{\s' \in \Vv_\k} (v_\s - v_\k) \mu_{\k,\s} a^\k_{\s,\s'} \mu_{\k,\s'} (v_{\s'} - v_\k) 
\le \max_{\s \in \Vv_\k} (\mu_{\k,\s})^2 \sum_{\s \in \Vv_\k} \left( \sum_{\s' \in \Vv_\k} |a^\k_{\s,\s'} |\right) \left(v_\s - v_\k \right)^2.
$$
\end{lem} 
\begin{proof} 
Using $ab \le \frac{a^2} 2 +  \frac{b^2} 2$, we obtain that 
\begin{multline*} 
\sum_{\s \in \Vv_\k} \sum_{\s' \in \Vv_\k} (v_\s - v_\k) \mu_{\k,\s} a^\k_{\s,\s'} \mu_{\k,\s'} (v_{\s'} - v_\k) \\
\le \max_{\s \in \Vv_\k} (\mu_{\k,\s})^2\sum_{\s \in \Vv_\k} \sum_{\s' \in \Vv_\k} |v_\s - v_\k| |a^\k_{\s,\s'} |  |v_{\s'} - v_\k| \\
\le \frac{\max_{\s \in \Vv_\k} (\mu_{\k,\s})^2} 2\sum_{\s \in \Vv_\k} \left(\sum_{\s' \in \Vv_\k} |a^\k_{\s,\s'} |\right) (v_\s - v_\k)^2 \\
+ \frac{\max_{\s \in \Vv_\k} (\mu_{\k,\s})^2} 2\sum_{\s' \in \Vv_\k} \left(\sum_{\s \in \Vv_\k} |a^\k_{\s,\s'} |\right) (v_\s' - v_\k)^2.
\end{multline*} 
One concludes the proof of Lemma~\ref{lem:poids-Aks} by noticing that, since $\A_\k$ is symmetric, the two terms in the right-hand side 
of the above inequality are equal.
\end{proof}

\begin{lem} \label{lem:taille-maille} 
There exists $C$ depending only on $\theta_\Tt$ and $\ell_\Dd$ such that 
$$ \meas(T)  \le \meas(\k) \le C \meas(T), \qquad \forall \k \in \Mm, \; \forall T \in \Tt \text{with } T \subset \k.$$
\end{lem} 

\begin{proof} 
Let $\k \in \Mm$, then there exist $T_1, \dots, T_{r} $ simplexes, with $r=\ell_\k$ if $d = 2$ and 
$r = 2 \#\Ee_\k$ if $d=3$, 
 such that 
$$
\bigcup_{i=1} ^{r_\k} \ov T_i = \ov \k, \qquad  T_i \cap T_j = \emptyset \; \text{if } \; i \neq j.
$$
The Euler-Descartes theorem ensures that $r \le 4 (\ell_\Dd - 1)$ if $d=3$.

If $T_i$ and $T_j$ share a common edge, one gets that 
$$
\meas(T_i) \le \theta^d \; \meas(T_j).
$$
Let $i_0, i_1 \in \{1,\dots, r_\k\} $ be arbitrary but different, 
we deduce from the previous inequality the following non-optimal 
estimate:
$$
\meas(T_{i_0}) \le \theta^{4(\ell_\Dd - 1) d} \; \meas(T_{i_1}).
$$
Let $i_{\rm max} $ be such that $\meas(T_{i_{\rm max} }) = \max_{1 \le i \le r} \meas(T_i)$, 
then 
$$
\meas(\k) \le r \; \meas(T_{i_{\rm max} }) \le 4 (\ell_\Dd - 1) \theta^{4(\ell_\Dd - 1) d} \meas(T_i), \qquad 
\forall i \in \{1,\dots, r\}.
$$
\end{proof}

We state now a slight generalization of \cite[Lemma 3.4]{BM13}, where the same result is proven in the particular case $q=2$. The straightforward adaptation of the proof given in~\cite{BM13} to the case $q \neq 2$ is left to the reader. 
\begin{lem} \label{lem:34-BM13} 
There exists $C$ depending only on $\ell_\Dd$ and $\theta_\Tt$ defined in~\eqref{eq:gamma_M} and~\eqref{eq:theta_T} 
respectively such that, for all $\bv \in W_{\Dd}$ and all $q \in [1,\infty]$, one has 
$$
\left\| \pi_{\Dd} \bv -  \pi_{\Tt} \bv \right\|_{L^q(\O)} + \left\| \pi_{\Dd} \bv -  \pi_{\Mm} \bv \right\|_{L^q(\O)} \le C h_\Tt \left\| \grad_{\Tt} \bv \right\|_{L^q(\O)}. 
$$
\end{lem} 
\medskip

\begin{lem} \label{lem:equiv-TtDd} 
Let $\Dd$ be a discretization of $\O$ as introduced in~\S\ref{sssec:mesh-space} such that $\zeta_\Dd >0$, 
then there exist $C_1 >0$ depending only on $q$, $\theta_\Tt$ and $\ell_\Dd$ 
and $C_2$ depending moreover on  $\zeta_\Dd$
such that 
\be\label{eq:equiv-TtDd} 
C_1 \| \pi_{\Dd} \bv\|_{L^q(\O)} \le  \| \pi_{\Tt} \bv\|_{L^q(\O)} \le C_2 \| \pi_{\Dd} \bv\|_{L^q(\O)}, 
\qquad \forall  \bv \in W_{\Dd}.
\ee
\end{lem} 
\begin{proof} 
Let $\h T$ be a reference tetrahedron, and let $\h v: \h T \to \R$ be an affine function 
with nodal values $v_i$, $i \in \{1,\dots 4\} $, then for all $q>0$, there exists $C$ depending on $q$ such that 
$$
\frac1C \sum_{i = 1} ^4 |v_i|^q \le \| \h v \|_{L^q(\h T)} ^q \le C \sum_{i = 1} ^4 |v_i|^q. 
$$
Therefore, using classical properties of the affine change of variable between simplexes, one gets the existence 
of $C$ depending only on $q$, $\theta_\Tt$, and $\ell_\Dd$  such that, for all $\bv \in W_{\Dd} $,
\begin{multline} \label{eq:BMv} 
\frac1C \sum_{\k \in \Mm} \meas(\k) \left(  |v_\k|^q  + \sum_{\s \in \Vv_\k} | v_\s|^q \right) \\ 
\le \| \pi_{\Tt} \bv \|_{L^q(\O)} ^q \le
 C  \sum_{\k \in \Mm} \meas(\k) \left(  |v_\k|^q  + \sum_{\s \in \Vv_\k} | v_\s|^q \right).
\end{multline} 
On the other hand, one has 
$$
\| \pi_{\Dd} \bv\|_{L^q(\O)} ^q =  \sum_{\k \in \Mm} m_\k |v_\k|^q + \sum_{\s \in \Vv} m_\s |v_\s|^q. 
$$
A classical geometrical property and~\eqref{eq:zeta_Dd} yield 
\be\label{eq:mkk} 
m_\k \le  \meas(\k)  = d \int_{\O} \pi_\Tt \bfe_{\k} (\x) \d\x \le \frac{d} {\zeta_\Dd} m_\k \qquad \forall \k \in \Mm,
\ee
and similarly 
$$ m_\s \le d \int_{\O} \pi_\Tt \bfe_{\s} (\x) \d\x \le \frac{d} {\zeta_\Dd} m_\s, 
\qquad \; \forall \s \in \Vv.
$$
Notice now that the following geometrical identity holds:
$$
d \int_{\O} \pi_\Tt \bfe_{\s} (\x) \d\x  = \sum_{\substack{T \in \Tt \\ \x_\s \in \p T} } \meas(T), \qquad \forall \s \in \Vv.
$$
Lemma~\ref{lem:taille-maille}  yields the existence of $C>0$ depending on $\theta_\Tt$ and $\ell_\Dd$ such that
$$
\frac1C \sum_{\k \in \Mm_\s} \meas(\k) \le d \int_{\O} \pi_\Tt \bfe_{\s} (\x) \d\x \le \sum_{\k \in \Mm_\s} \meas(\k), \qquad \forall \s \in \Vv, 
$$
and the result of Lemma~\ref{lem:equiv-TtDd}  follows. 
\end{proof} 

\begin{lem} \label{lem:equiv-DdMm} 
Let $\Dd$ be a discretization of $\O$ as introduced in~\S\ref{sssec:mesh-space} such that $\zeta_\Dd >0$,
then, for all $q \in [1,\infty]$, one has 
$$
\left\| \pi_{\Mm} \bv \right\|_{L^q(\O)} 
\le \left(\frac d{\zeta_\Dd} \right)^{1/q} \left\| \pi_{\Dd} \bv \right\|_{L^q(\O)}, \qquad  \forall \bv \in W_{\Dd}.
$$
\end{lem} 
\begin{proof} 
Let $\bv = \left(v_\k, v_\s\right)_{\k \in \Mm, \s \in \Vv} \in W_{\Dd},$ then it follows from~\eqref{eq:mkk} 
that
\begin{align*} 
\left\| \pi_{\Mm} \bv \right\|_{L^q(\O)} ^q =& \sum_{\k \in \Mm} \meas(\k) \left| v_\k \right|^q  \\
\le& \left(\frac{d} {\zeta_\Dd} \right) \sum_{\k \in \Mm} m_\k \left| v_\k \right|^q \le 
\left(\frac{d} {\zeta_\Dd} \right) \left\| \pi_{\Dd} \bv \right\|_{L^q(\O)} ^q.
\end{align*} 
\end{proof} 

\begin{lem} \label{lem:Aovbv} 
Let $\bv = {(v_\k, v_\s)} _{\k,\s} \in W_\Dd$ be such that $v_\beta \ge 0$ for all $\beta \in \Mm \cup \Vv$, and 
define $\ov \bv = {(\ov v_\k, \ov v_\s)} _{\k,\s} \in W_\Dd$ by 
$$
\ov v_\s = 0, \qquad \ov v_\k = \max \left( v_\k, \max_{\s' \in \Vv_\k} v_{\s'} \right), \qquad \forall \s \in \Vv, \forall \k \in \Mm. 
$$
Then there exists $C$ depending only on $\theta_\Tt$, $\ell_\Dd$ and $\zeta_\Dd$ such that 
$$
\left\| \pi_\Mm \ov \bv \right\|_{L^1(\O)} \le C \left\| \pi_\Dd \bv \right\|_{L^1(\O)}.
$$
\end{lem} 
\begin{proof} 
Let $\bv \in W_\Dd$ be a vector with positives coordinates, and let $\ov \bv$ be constructed as above. 
It follows from the construction of $\ov \bv$ that 
$$
\ov v_\k \le v_\k + \sum_{\s \in \Vv_\k} v_\s, \qquad \forall \k \in \Mm, 
$$
whence, applying~\eqref{eq:BMv}  with $q=1$, one gets 
$$
\| \pi_\Mm \ov \bv \|_{L^1(\O)} \le C \| \pi_\Tt \bv \|_{L^1(\O)}. 
$$ 
The result now directly follows from Lemma~\ref{lem:equiv-TtDd}.
\end{proof}

\begin{lem} \label{lem:A7} 
Let $\u = (u_\k, u_\s)_{\k \in \Mm, \s \in \Vv} \in W_\Dd$, then for all $\k \in \Mm$, we define 
$\ov \bd \u = \left( \ov \delta_\k \u, \ov \delta_\s \u\right)_{\k \in \Mm, \s \in \Vv} \in W_\Dd$ by 
$$ 
 \ov \delta_\s \u = 0 \quad \text{and } \quad \ov \delta_\k \u = \max_{\s' \in \Vv_\k} | u_\k - u_\s |, 
 \qquad \forall \k \in \Mm, \; \forall \s \in \Vv,
$$
then, for all $q \in [1,\infty]$, there exists $C$ depending only on $q$, $\theta_\Tt$, 
and $\ell_\Dd$ such that 
\be\label{eq:bdu-Lq} 
\left\| \pi_\Mm \ov \bd \u \right\|_{L^q(\O)} \le C h_\Tt \| \grad_\Tt \u \|_{L^q(\O)}.
\ee
\end{lem} 
\begin{proof} 
Let $\k \in \Mm$ and $\s \in \Vv_\k$, then there exists a simplicial sub-element $T \in \Tt$ of 
$\k \in \Mm$ such that $\x_\k$ and $\x_\s$ are vertices of $T$. 
Then it follows from classical finite element arguments (see e.g.~\cite{Ciarlet_Handbook, EG04}) 
that 
$$
\meas(T)^{1/q} |u_\k - u_\s| \le c \frac{\left(h_T\right)^2} {\rho_T} \| \grad_\Tt \u \|_{L^q(T)} \le C h_\Tt  \| \grad_\Tt \u \|_{L^q(\k)}, 
$$
where $c$ depends only on the dimension $d$ and on $q$, while $C$ depends additionally  on $\theta_\Tt$.
Thanks to Lemma~\ref{lem:taille-maille}, we get the existence of $C$ depending on $d$, $q$, $\theta_\Tt$ and $\ell_\Dd$ 
such that, 
$$
\meas(\k)^{1/q} |u_\k - u_\s| \le C  h_\Tt \left\| \grad_\Tt \u \right\|_{L^q(\k)}, \quad  \forall \k \in \Mm, \; \forall \s \in \Vv_\k. 
$$
Summing over $\k \in \Mm$ provides that~\eqref{eq:bdu-Lq}  holds.
\end{proof} 

\noindent
{\em Acknowledgements.} The authors are grateful to the anonymous referees for their valuable 
comments on the paper. They also warmly thank Flore Nabet and Thomas Rey for their precious 
feedback.


\def\ocirc#1{\ifmmode\setbox0=\hbox{$#1$}\dimen0=\ht0 \advance\dimen0
  by1pt\rlap{\hbox to\wd0{\hss\raise\dimen0
  \hbox{\hskip.2em$\scriptscriptstyle\circ$}\hss}}#1\else {\accent"17 #1}\fi}

\end{document}